\documentclass{article}
\usepackage{lipsum}
\usepackage{amsmath, amsthm, amssymb}

\DeclareMathOperator*{\argmin}{arg\,min}
\usepackage{bm}
\usepackage{dsfont}
\usepackage{subfigure}
\usepackage{amscd,amsthm,xspace}
\usepackage{lineno,hyperref}
\usepackage[T1]{fontenc}
\usepackage[utf8]{inputenc}
\usepackage{authblk}
\usepackage{enumitem}
\usepackage[pdftex]{graphicx}
\usepackage[toc,titletoc,page]{appendix}
\makeatletter
\g@addto@macro{\endabstract}{\@setabstract}
\newcommand{\authorfootnotes}{\renewcommand\thefootnote{\@fnsymbol\c@footnote}}%
\makeatother

\newtheorem{prop}{Proposition}[section]
\newtheorem{lemma}[prop]{Lemma}

\newtheorem{theorem}[prop]{Theorem}
\newtheorem{defn}[prop]{Definition}

\newtheorem{corollary}[prop]{Corollary}

\newtheorem{remark}[prop]{Remark}
\newtheorem{example}[prop]{Example}

\newcommand{\BX}{{\bf X}}

\def\Var{{\rm Var}}
\def\Cov{{\rm Cov}\,}

\def\E{{\rm\mathbb{E}}}

\def\det{{\rm det}\,}

\def\sign{{\rm sign}\,}
\def\diag{{\rm diag}\,}

\def\sp{{\rm span}\,}

\def\Rank{{\rm Rank}\,}

\def\Matri{{\rm Matri}\,}

\usepackage[margin=3cm]{geometry}
\usepackage{lipsum}

\usepackage[utf8]{inputenc}

\title{Local behavior of critical points of isotropic\\ Gaussian random fields}
\author{Paul Marriott, Weinan Qi, Yi Shen}
\date{}

\begin{document}

%\lipsum

\maketitle

\authorfootnotes
 Department of Statistics and Actuarial Science, University of Waterloo. Waterloo, ON N2L 3G1, Canada.\par\bigskip
\thanks{Email: pmarriott@uwaterloo.ca, w7qi@uwaterloo.ca, yi.shen@uwaterloo.ca}\par
\thanks{This work is supported by NSERC grant 2020-04356.}

\abstract
In this paper we examine isotropic Gaussian random fields defined on $\mathbb R^N$ satisfying certain conditions. Specifically, we investigate the type of a critical point situated within a small vicinity of another critical point, with both points surpassing a given threshold. It is shown that the Hessian of the random field at such a critical point is equally likely to have a positive or negative determinant. Furthermore, as the threshold tends to infinity, almost all the critical points above the threshold are local maxima and the saddle points with index $N-1$. Consequently, we conclude that the closely paired critical points above a high threshold must comprise one local maximum and one saddle point with index $N-1$.

\section{Introduction}
Let $X=\{X(\bm{t}),\bm{t}\in\mathbb{R}^N\}$ be an isotropic Gaussian random field defined on $R^N, N=1,2,...$. The critical points of $X$ are the points at which its gradient vanishes. They are naturally classified into different types such as local maxima, saddle points of different kinds, and local minima. This paper explores the local interactions between critical points for which the values of $X$ exceed a threshold $u$. More specifically, we pose and address the following question:\\

\textit{When two critical points surpassing $u$ are situated in close proximity to one another, what types can they belong to?}\\

The motivation for studying the critical points of Gaussian random fields largely stems from random topology, an emerging area of probability that deals with the topological features of random objects, such as random sets or random graphs. Its application in statistics is known as topological data analysis, and has also garnered significant research interest. For an overview of this field, readers can refer to \cite{adl10} or \cite{was}. An important mathematical tool employed in these areas is Morse theory, which suggests that topological information related to the homology groups of a set can be derived from the critical points of a function defined on that set, provided that the function meets certain non-degeneracy conditions \cite{mor}. As a result, understanding the behavior of critical points, such as their locations and heights, proves highly beneficial.

It has long been believed that for a stationary Gaussian random field satisfying some mild conditions, the locations of critical points exceeding $u$ will converge to a (homogeneous) Poisson point process as $u$ tends to infinity \cite{ald}. Putting its proof aside, another issue when applying this so-called Poisson clumping heuristic is that in real-world scenarios, it is not feasible for $u$ to truly tend to infinity. Consequently, it is often necessary to consider the deviation from the Poisson limit, which is due to the interactions between the critical points, as Poisson limit corresponds to the independent case. For example, if the covariance is positive and a critical point above a high threshold already exists at a particular location, the value of the random field will be elevated around that point, conditional on this event. This could intuitively increase the chance of a point in the affected area to be a critical point exceeding the threshold. This is what pushes us to understand the behavior of the critical points in close vicinity of another critical point surpassing a threshold.

The remaining sections of this paper are structured as follows. In Chapter \ref{proj2:sec:bs} we introduce the basic settings and notations. Chapter \ref{proj2:sec:cs} consists of a detailed analysis of the covariance of the random field and its second derivatives. Using these findings, Chapter \ref{proj2:sec:mr1} explores the behavior of two critical points when their distance approaches 0, and shows that in this case, the determinants of their Hessian matrices must have opposite signs. Chapter \ref{proj2:sec:mr2} further reveals that as the threshold $u$ tends to infinity, only the local maxima and the saddle points with index $N-1$ will remain. Therefore, a pair of closely situated critical points of a Gaussian random field above a high threshold predominantly comprises one local maximum and one saddle point with index $N-1$.

\section{Basic Settings and Notations}\label{proj2:sec:bs}

Let $\mathbb{R}^N$ be the $N-$dimensional Euclidean space, and endow it with the usual Euclidean norm $\|\cdot\|_N$, or simply $\|\cdot\|$ when the dimension is clear. Let $\lambda^n$ be the $n$-dimensional Lebesgue measure. We use $\partial$ and $\bar{\quad}$ for the boundary and the closure of a set, respectively.

In the following sections, we will make heavy calculation of matrices and their sub-matrices.
Let $\mathbb{N}:=\{0,1,\dots\}$ and
$\mathbb{Z}^+:=\{1,2,\dots\}$.
For any matrix $\bm{M}\in\mathbb{R}^{m\times n}$ ($m,n\in\mathbb{Z}^+$), denote by $\bm{M}_{(i)}$ its $i$-th row and by $\bm{M}^{(j)}$ its $j$-th column for $i=1,\dots,m$ and $j=1,\dots,n$.
For any $a,b\in\mathbb{Z}^+$ such that $a\le b$, denote
by $a:b$ the set of all integers in $[a,b]$. 
For any $S_1\subset 1:m$ and $S_2\subset 1:n$, denote by
$\bm{M}[S_1,S_2]$ the sub-matrix of $\bm{M}$ formed by the entries with row indices in $S_1$ and column indices in $S_2$. For a real vector $\bm{a}\in \mathbb{R}^n$ (row or column), we use $\bm{a}[i]$ to denote its $i$-th coordinate for $i=1,\dots,n$. Denote by $\bm{I}_n$ the identity matrix of size $n$,  by $\bm{0}_{n}$ and $\bm{0}_{m\times n}$ the zero vector and zero matrix, respectively. All vectors are column vectors by default. Given a vector space $V$ over $\mathbb{R}$, for any $S\subset V$, denote by $\sp(S)$ the linear span of $S$, i.e., 
$$\sp(S):=\left\{\sum_{i=1}^n k_i\bm{v}_i\Bigg|n\in\mathbb{N},k_i\in\mathbb{R},\bm{v}_i\in S\right\}.$$

For any two times differentiable function $f:\mathbb{R}^N\to\mathbb{R}$, 
let $\triangledown f$ and $\triangledown^2 f$ be the gradient and the Hessian matrix of $f$, respectively. 
A point $\bm{t}\in\mathbb{R}^N$ is said to be a critical point of $f$ if $\triangledown f(\bm{t})=\bm{0}$.
The index or type of a critical point $\bm{t}$ is defined to be the number of negative eigenvalues (counted with their multiplicities) of $\triangledown^2 f(\bm{t})$. For example, a local maximum is typically a critical point with index $N$. 

For any two non-negative real-valued functions $h(x)$ and $g(x)$, $x\in\mathbb
R$, we write
$$h(x)=\Theta(g(x))\text{ as }x\to\infty$$
if there exist positive constants $k_1$, $k_2$ and $C$ such that for any $x>C$,
$$k_1g(x)\le h(x)\le k_2g(x).$$

For any symmetric matrix $\bm{M}=(m_{ij})\in\mathbb{R}^{N\times N}$, a vector $\bm{a}\in\mathbb{R}^{N(N+1)/2}$  is said to be the usual vectorization of $\bm{M}$ if 
$$\bm{a}[i+j(j-1)/2]=m_{ij}\text{ for any $1\le i\le j\le N$.}$$

\begin{defn}(Matriculation) \label{Def:matri}
A matrix $\bm{M}\in\mathbb{R}^{N\times N}$ is said to be the $N$-th order matriculation of a vector $\bm{a}\in\mathbb{R}^{n}$ ($n\ge N(N+1)/2$), written as 
$\bm{M}=\Matri_N(\bm{a})$, if $\bm{a}[1:(N(N+1)/2)]$ is the usual vectorization of $\bm{M}$.
\end{defn}
Note that 
different vectors, even with different lengths, can share the same $N$-th order matriculation since only the first $N(N+1)/2$ coordinates of them are considered.

For completeness, we include the following result, which gives the relation between the covariance of the partial derivatives of a Gaussian random field and its covariance function.

\begin{lemma}(Section 5.5, \cite{adl})\label{Lem:MSD} 
Let $\{X(\bm{t})$, $\bm{t}\in\mathbb{R}^N\}$ be a Gaussian random field with covariance function $r(\bm{s},\bm{t}):=\Cov[X(\bm{s}),X(\bm{t})]$, where $\bm{s}=(s_1,\dots,s_N)^T$ and $\bm{t}=(t_1,\dots,t_N)^T$.

\begin{enumerate}[label=(\roman*)]
    \item For any positive integers $k$ and $i_1,\dots,i_k\in\{1,2,\dots, N\}$, the $k$ times mean square derivatives $X_{i_1\dots i_k}(\bm{t})$ exists if and only if the derivative $\partial^{2k}r(\bm{s},\bm{t})/\partial s_{i_1}\cdots \partial s_{i_k} \partial t_{i_1}\cdots \partial t_{i_k}$ exists and is finite at the point $(\bm{t},\bm{t})\in\mathbb{R}^{2N}$. 
    
    \item For some positive integer $k$, suppose that the derivative $\partial^{2k}r(\bm{s},\bm{t})/\partial s_{i_1}\cdots \partial s_{i_k} \partial t_{j_1}\cdots \partial t_{j_k}$ exists and is finite for any $\bm{s},\bm{t}\in\mathbb{R}^N$ and $i_1,\dots,i_k,j_1,\dots,j_k\in\{1,2,\dots, N\}$. Then we have for any $0\le k_1,k_2\le k$ and $\bm{s},\bm{t}\in\mathbb{R}^N$,
    \begin{equation}\label{eq:covmean}
    \Cov\left[X_{i_1\dots i_{k_1}}(\bm{s}),X_{j_1\dots j_{k_2}}(\bm{t})\right]=\frac{\partial^{k_1+k_2}r(\bm{s},\bm{t})}{\partial s_{i_1}\cdots \partial s_{i_{k_1}} \partial t_{j_1}\cdots \partial t_{j_{k_2}}}. 
    \end{equation}
    
\end{enumerate}
\end{lemma}

\begin{defn}\label{Def:qualified}
An isotropic Gaussian random field $X(\bm{t})$, $\bm{t}=(t_1,\dots,t_N)^T\in\mathbb{R}^N$ ($N\ge 2$) with covariance structure $$\rho(\|\bm{t}-\bm{s}\|^2)\equiv R(\bm{t}-\bm{s}):=\Cov[X(\bm{s}),X(\bm{t})]$$
is said to be \textbf{qualified} if the following conditions are satisfied:
\begin{enumerate}[label=(\arabic*)]
    \item $X$ is centered with unit variance, i.e., $\E[X(\bm{0})]=0$ and $\rho(0)=1$.
    \item $X$ has almost surely partial derivatives of up to second order.
    
    %\item $X$ satisfies Condition (\ref{cfi2}) on any compact subset of $\mathbb{R}^N$; for any $z\in\mathbb{R}$,  $X(\bm{t})$ conditional on $\{\triangledown X(\bm{0})=\bm{0}_{N},X(\bm{0})=z\}$ satisfies Condition (\ref{cfi2}) on any compact subset of $\mathbb{R}^N\setminus\{\bm{0}_N\}$.
    
    \item The sixth derivative of $\rho(x)$ at $x=0$ exists, which implies there exists a constant $\delta_{\rho}>0$ such that
   the fifth order derivative of $\rho(x)$ exists and is bounded on $x\in [0,\delta_{\rho}^2]$.
    
    \item 
    The distribution of 
    $$(\triangledown^2 X(\bm{t}),\triangledown X(\bm{t}),X(\bm{t}),\triangledown X(\bm{0}),X(\bm{0}))$$
    is non-degenerate for any $\bm{t}\in\mathbb{R}^{N}\setminus\{\bm{0}_{N}\}$, and the distribution of $(X_{111}(\bm{0}),\triangledown X(\bm{0}))$ is also non-degenerate, where $X_{111}(\bm{0})$ is the third mean square derivative of $\BX$ at $\bm{0}$ along direction $t_1$.
 
\end{enumerate}
\end{defn}

The existence of $X_{111}(\bm{0})$ in the above definition is guaranteed by Condition (3) and Lemma \ref{Lem:MSD}. Also, note that we restrict the definition to dimensions $N\geq 2$. The case $N=1$ is slightly different and will be discussed separately.

\begin{remark}\label{Rem:Condition} 
As we will see later, Condition (3) is needed for the asymptotic expansion of the covariance matrix of $(\triangledown^2 X(\bm{t}),X(\bm{t}),X(\bm{0}))$ conditional on $\triangledown X(\bm{t})=\triangledown X(\bm{0})=\bm{0}_{N}$ as $\|\bm{t}\|\to 0$. Conditions (3) and (4) are also associated with the convergence speeds of the ordered eigenvalues of this covariance matrix as $\|\bm{t}\|\to 0$. 
\end{remark}

%apply Lemma \ref{Lem:RiceFormula} to the Gaussian random field $X(\bm{t})$ conditional on $\{\triangledown X(\bm{0})=\bm{0}_{N},X(\bm{0})=z\}$ for any $z\in\mathbb{R}$ on any compact set $T\subset\mathbb{R}^N\setminus\{\bm{0}_N\}$ with $\lambda_{N-1}(\partial T)<\infty$ (see Appendix \ref{App:discuss_condition}). 

%allow us to apply Proposition 3.3 in \cite{che} to guarantee the non-degeneracy of the joint distributions of $X$ and its derivatives (see Lemma \ref{Lem:sixorder}),

Indeed, the assumption that $X$ is qualified imposes some constraints on the derivatives of $\rho$. Let $\rho^{(i)}(x)$, $x\ge 0$ and $i\ge 1$ be the $i$-th derivative (if exists) of $\rho$ at $x$.
Firstly, by Condition (4) in Definition \ref{Def:qualified},
$X_1$, $X_{11}$ and $X_{111}$ are all non-degenerate. For any $i_1,\dots,i_6\in\{1,2,\dots,N\}$ and $\bm{t}\in\mathbb{R}^N$, straightforward calculation leads to
\begin{equation}\label{R1}
R_{i_1}(\bm{t})=2t_{i_1}\rho^{(1)}(\|\bm{t}\|^2) ,   
\end{equation}
\begin{equation}\label{R2}
R_{i_1i_2}(\bm{t})=2\rho^{(1)}(\|\bm{t}\|^2)\delta_{i_1,i_2}+4t_{i_1}t_{i_2}\rho^{(2)}(\|\bm{t}\|^2),
\end{equation} 
\begin{equation}\label{R3}
R_{i_1i_2i_3}(\bm{t})=4(t_{i_3}\delta_{i_1,i_2}+t_{i_1}\delta_{i_2,i_3}+t_{i_2}\delta_{i_1,i_3})\rho^{(2)}(\|\bm{t}\|^2)+8t_{i_1}t_{i_2}t_{i_3}\rho^{(3)}(\|\bm{t}\|^2).    
\end{equation}

More calculation shows
\begin{equation}\label{R4}
    \begin{aligned}
    &~~~~R_{i_1i_2i_3i_4}(\bm{t})\\
    &=4\left(\delta_{i_1,i_2}\delta_{i_3,i_4}+\delta_{i_2,i_3}\delta_{i_1,i_4}+\delta_{i_1,i_3}\delta_{i_2,i_4}\right)\rho^{(2)}(\|\bm{t}\|^2)\\
    &~~~~+8\left(t_{i_3}t_{i_4}\delta_{i_1,i_2}+t_{i_1}t_{i_4}\delta_{i_2,i_3}+t_{i_2}t_{i_4}\delta_{i_1,i_3}+t_{i_2}t_{i_3}\delta_{i_1,i_4}+t_{i_1}t_{i_3}\delta_{i_2,i_4}+t_{i_1}t_{i_2}\delta_{i_3,i_4}\right)\rho^{(3)}(\|\bm{t}\|^2)\\
    &~~~~+16t_{i_1}t_{i_2}t_{i_3}t_{i_4}\rho^{(4)}(\|\bm{t}\|^2),
    \end{aligned}    
\end{equation}    
and
\begin{equation}\label{R6}
R_{i_1i_2i_3i_4i_5i_6}(\bm{t})
=\rho^{(3)}(\|\bm{t}\|^2)\sum_{(p_1\dots,p_5)\in\Pi_5}\delta_{i_{p_1},i_{p_2}}\delta_{i_{p_3},i_{p_4}}\delta_{i_{p_5},i_6}+f(\bm{t}),
\end{equation}
where $\Pi_n$ is the set of all the permutations on $\{1,\dots, n\}$, and $f(\bm{t})$ is a real-valued function of $\bm{t}\in\mathbb{R}^N$ such that $f(\bm{0})=0$. 

Combining them with Lemma \ref{Lem:MSD}, we have
$$-2\rho^{(1)}(0)=\Var[X_1(\bm{0})],$$
$$12\rho^{(2)}(0)\\=\Var[X_{11}(\bm{0})],$$
$$-120\rho^{(3)}(0)=\Var[X_{111}(\bm{0})],$$
hence 
\begin{equation}\label{ineq:rho1rho2}
    \rho^{(1)}(0)<0, \rho^{(2)}(0)>0\text{ and }\rho^{(3)}(0)<0. 
\end{equation}

Secondly, let $\alpha=\rho^{(1)}(0)^{-1}\rho^{(2)}(0)^2$ and $\beta=\rho^{(3)}(0)$. Then by Condition (4) in Definition \ref{Def:qualified} and the Cauchy–Schwarz inequality,
\begin{equation}\label{eq:covIne}
(\Cov[X_{111}(\bm{0}),X_{1}(\bm{0})])^2< \Var[X_{111}(\bm{0})] \Var[X_{1}(\bm{0})].
\end{equation}
As a result of Lemma \ref{Lem:MSD}, the above inequality (\ref{eq:covIne}) is equivalent to
$$\left(12\rho^{(2)}(0)\right)^2<\left(-120\rho^{(3)}(0)\right)\left(-2\rho^{(1)}(0)\right).$$
Then by $\rho^{(1)}(0)<0$, we have
\begin{equation}\label{Rel:1830}
    \alpha>\frac{5}{3}\beta.    
\end{equation}

Moreover, for any $r>0$ and $\bm{t}\in\mathbb{R}^N$, let $B(\bm{t},r)$ be the $N$-dimensional open ball centered at $\bm{t}$ with radius $r$. We have for any $1\le i\le N$ and $\bm{t}\in B(\bm{0}_N,\delta_{\rho})$ ($\delta_{\rho}$ as in Definition \ref{Def:qualified}),
$$
\begin{aligned}
\left|\Cov[X_i(\bm{t}),X_i(\bm{0})]\right|\le \sqrt{\Var[X_i(\bm{t})]\Var[X_i(\bm{0})]}.
\end{aligned}
$$
By Lemma \ref{Lem:MSD} and (\ref{R2}), this is equivalent to
$$
\left|-\rho^{(1)}(\|\bm{t}\|^2)-2t_i^2\rho^{(2)}(\|\bm{t}\|^2)\right|\le -\rho^{(1)}(0).
$$
Note that the equal sign in the above inequality holds if and only if $X_i(\bm{t})$ and $X_i(\bm{0})$ are linearly dependent, which is impossible for any $\bm{t}\in B(\bm{0}_N,\delta_{\rho})\setminus{\{\bm{0}_N\}}$ by Condition (4) in Definition \ref{Def:qualified}. Thus, for any $1\le i\le N$ and $\bm{t}\in B(\bm{0}_N,\delta_{\rho})\setminus{\{\bm{0}_N\}}$,
\begin{equation}\label{rho1ineq}
    \left|-\rho^{(1)}(\|\bm{t}\|^2)-2t_i^2\rho^{(2)}(\|\bm{t}\|^2)\right|< -\rho^{(1)}(0).    
\end{equation}
Then for any $x\in (0,\delta_{\rho}^2]$, by taking $\bm{t}:=(0,\dots,0,\sqrt{x})\in \mathbb{R}^N$, $N\geq 2$ and $i=1$ in (\ref{rho1ineq}), we have
\begin{equation}\label{absrho1}
    \left|\rho^{(1)}(x)\right|<-\rho^{(1)}(0).    
\end{equation}
 
Finally, by Proposition 3.3 in \cite{che}, we see
\begin{equation}\label{rho1rho2}
   1-\frac{\rho^{(1)}(0)^2}{3\rho^{(2)}(0)}>1-\frac{N+2}{3N}\ge 0. 
\end{equation}

Let $X$ be qualified and $L:=N(N+1)/2+2$. For any $\bm{t}\in\mathbb{R}^{N}\setminus\{\bm{0}_{N}\}$, by Lemma \ref{Lem:MSD} and the properties of Gaussian distributions, 
$(\triangledown^2 X(\bm{t}),X(\bm{t}),X(\bm{0}))$ conditional on $\triangledown X(\bm{t})=\triangledown X(\bm{0})=\bm{0}_{N}$ is a Gaussian $L$-vector. Let $\bm{\Sigma}(\bm{t})$ be its covariance matrix.
By Condition (4) in Definition \ref{Def:qualified}, $\bm{\Sigma}(\bm{t})$ is positive-definite for any $\bm{t}\in\mathbb{R}^{N}\setminus\{\bm{0}_{N}\}$.
Let
\begin{equation}\label{def:order}
\lambda_1(\bm{t})\ge\lambda_2(\bm{t})\ge\dots\ge\lambda_L(\bm{t})>0    
\end{equation}
be the ordered eigenvalues of $\bm{\Sigma}(\bm{t})$. 
Then an
eigen-decomposition of $\bm{\Sigma}(\bm{t})$ is given by
$$\bm{\Sigma}(\bm{t})=\bm{P}(\bm{t})\bm{\Lambda}(\bm{t})\bm{P}^T(\bm{t}),$$
where $\bm{P}(\bm{t})$ is an $L\times L$ orthogonal real matrix and $\bm{\Lambda}(\bm{t}):=\diag(\lambda_1(\bm{t}),\dots,\lambda_L(\bm{t}))$. 
One should note that $\bm{P}(\bm{t})$ in the above eigen-decomposition may be non-unique since the ordering in (\ref{def:order}) is not strict.

For $\bm{t}\in\mathbb{R}^{N}\setminus\{\bm{0}_{N}\}$, denote
$$\bm{A}(\bm{t}):=\bm{P}(\bm{t})\bm{\Lambda}^{1/2}(\bm{t})=\bm{P}(\bm{t})\diag\left(\lambda_1^{1/2}(\bm{t}),\dots,\lambda_L^{1/2}(\bm{t})\right).$$
Immediately, we have
$$\bm{\Sigma}(\bm{t})=\bm{A}(\bm{t})\bm{A}^T(\bm{t}),$$
and $\bm{A}(\bm{t})$ can be uniquely determined by a version of $\bm{P}(\bm{t})$.

Denote by $\mathbb{S}^{N-1}$ the unit $(N-1)$-sphere. Since $X$ is isotropic, we will focus on the behavior of $\bm{\Sigma}(r\bm{u})$ along a given direction $\bm{u}=(u_1,\dots,u_N)^T\in\mathbb{S}^{N-1}$ for $r> 0$. Thus, it would be more convenient to adopt the notations:
$$\bm{\Sigma}_{\bm{u}}(r):=\bm{\Sigma}(r\bm{u}),~~\bm{A}_{\bm{u}}(r):=\bm{A}(r\bm{u}),~~ \bm{P}_{\bm{u}}(r):=\bm{P}(r\bm{u}),\text{ and }\bm{\Lambda}_{\bm{u}}(r):=\bm{\Lambda}(r\bm{u}),$$
which emphasize on them being the matrix-valued functions of $r$.

Let $M_L(\mathbb{R})$ be the space of all $L\times L$ real symmetric matrices endowed with the Frobenius norm $\|\cdot\|_F$, i.e., for any $\bm{A}=(a_{ij})\in M_L(\mathbb{R})$, $\|\bm{A}\|_F=\sqrt{\sum_{i,j=1}^La_{ij}^2}.$
Recall that $\mathbb{R}^N$ is endowed with the
usual Euclidean norm. The following lemma describes the behavior of $\bm{\Sigma}_{\bm{u}}(r)$ as $r\to 0$.
\begin{lemma}\label{Lem:Sigma}
Let $X$ be qualified. 
Then for any direction $\bm{u}=(u_1,\dots,u_N)^T\in\mathbb{S}^{N-1}$, 
$\bm{\Sigma}_{\bm{u}}(0):=\lim_{r\downarrow 0}\bm{\Sigma}_{\bm{u}}(r)$ exists, and the function
$\bm{\Sigma}_{\bm{u}}(\cdot):[0,\infty)\to M_L(\mathbb{R})$ is continuous on $[0,\delta_{\rho}]$, where $\delta_{\rho}$ is as defined in Definition \ref{Def:qualified}. 
In addition, we have as $r\to 0$,
$$\bm{\Sigma}_{\bm{u}}(r)=\bm{\Sigma}_{\bm{u},0}+\bm{\Sigma}_{\bm{u},2}r^2+o(r^2),$$
where $\bm{\Sigma}_{\bm{u},0},\bm{\Sigma}_{\bm{u},2}\in M_L(\mathbb{R})$ satisfy
\begin{enumerate}
    \item $\bm{\Sigma}_{\bm{u},0}=\bm{\Sigma}_{\bm{u}}(0)$ is positive semi-definite;
    \item for any $1\le i_1\le j_1\le N$ and $1\le i_2\le j_2\le N$,
    \begin{equation}\label{eq:Sigma0_main}
    \begin{aligned}
    &~~~~\bm{\Sigma}_{\bm{u},0}[i_1+j_1(j_1-1)/2,i_2+j_2(j_2-1)/2]\\
    &=4\rho^{(2)}(0)(\delta_{i_2,j_1}\delta_{i_1,j_2}+\delta_{i_1,i_2}\delta_{j_1,j_2}-\delta_{j_1,j_2}u_{i_1}u_{i_2}-\delta_{i_1,j_2}u_{j_1}u_{i_2}\\
    &~~~~-\delta_{i_2,j_1}u_{i_1}u_{j_2}-\delta_{i_1,i_2}u_{j_1}u_{j_2}+2u_{i_1}u_{j_1}u_{i_2}u_{j_2})\\
    &~~~~+\frac{8}{3}\rho^{(2)}(0)(\delta_{i_1,j_1}-u_{i_1}u_{j_1})(\delta_{i_2,j_2}-u_{i_2}u_{j_2})\end{aligned}
    \end{equation}
    and
    \begin{equation}\label{eq:sigma2_main}
    \begin{aligned}
    &~~~~\bm{\Sigma}_{\bm{u},2}[i_1+j_1(j_1-1)/2,i_2+j_2(j_2-1)/2]\\
    &=\left(2\alpha-\frac{14}{9}\beta\right)\delta_{i_1,j_1}\delta_{i_2,j_2}
    +\left(4\alpha-\frac{52}{9}\beta\right)\delta_{i_2,j_2}u_{i_1}u_{j_1}
    +\left(4\alpha-\frac{52}{9}\beta\right)\delta_{i_1,j_1}u_{i_2}u_{j_2} \\   
    &~~~~
    + \left(2\alpha-6\beta\right)\delta_{j_1,j_2}u_{i_1}u_{i_2}
    + \left(2\alpha-6\beta\right)\delta_{i_1,j_2}u_{j_1}u_{i_2}
    + \left(2\alpha-6\beta\right)\delta_{i_2,j_1}u_{i_1}u_{j_2}\\
    &~~~~
    + \left(2\alpha-6\beta\right)\delta_{i_1,i_2}u_{j_1}u_{j_2}
    + \frac{64}{9}\beta u_{i_1}u_{j_1}u_{i_2}u_{j_2},
    \end{aligned}\end{equation}
    where $\alpha=\rho^{(1)}(0)^{-1}\rho^{(2)}(0)^2$ and $\beta=\rho^{(3)}(0)$;

    \item for any $1\le i_1\le j_1\le N$,
    \begin{equation}\label{eq:Sigma0_side}
    \begin{aligned}
    \bm{\Sigma}_{\bm{u},0}[i_1+j_1(j_1-1)/2,L]=\bm{\Sigma}_{\bm{u},0}[i_1+j_1(j_1-1)/2,L-1]=\frac{4}{3}\rho^{(1)}(0)(\delta_{i_1,j_1}-u_{i_1}u_{j_1})
    \end{aligned}
    \end{equation}
    and
    $$
    \begin{aligned}
    &~~~~\bm{\Sigma}_{\bm{u},2}[i_1+j_1(j_1-1)/2,L]=\bm{\Sigma}_{\bm{u},2}[i_1+j_1(j_1-1)/2,L-1]\\
    &=\left(\frac{1}{3}\alpha'-\frac{1}{9}\beta'\right)\delta_{i_1,j_1}+\left(\frac{2}{3}\alpha'-\frac{14}{9}\beta'\right) u_{i_1}u_{j_1},
    \end{aligned}
    $$
    where
    $\alpha'=\rho^{(2)}(0)$ and $\beta'=\rho^{(1)}(0)\rho^{(2)}(0)^{-1}\rho^{(3)}(0)$;

    \item 
    \begin{equation}\label{eq:Sigma0_corner}
    \begin{aligned}
    &~~~~\bm{\Sigma}_{\bm{u},0}[L-1,L-1]=\bm{\Sigma}_{\bm{u},0}[L,L-1]=\bm{\Sigma}_{\bm{u},0}[L-1,L]=\bm{\Sigma}_{\bm{u},0}[L,L]\\
    &=1-\frac{1}{3}\rho^{(1)}(0)^2\rho^{(2)}(0)^{-1}
    \end{aligned}
    \end{equation}
    and
    $$
    \begin{aligned}
    &~~~~\bm{\Sigma}_{\bm{u},2}[L-1,L-1]=\bm{\Sigma}_{\bm{u},2}[L,L-1]=\bm{\Sigma}_{\bm{u},2}[L-1,L]=\bm{\Sigma}_{\bm{u},2}[L,L]\\
    &=-\frac{1}{6}\rho^{(1)}(0)+\frac{5}{18}\rho^{(1)}(0)^2\rho^{(2)}(0)^{-2}\rho^{(3)}(0).
    \end{aligned}
    $$
\end{enumerate}
\end{lemma}

\begin{proof}
Since the proof is long and mainly consists of heavy computation, we include it in Appendix \ref{App:Analytic}.
\end{proof}

From Lemma \ref{Lem:Sigma}, we see that the elements of  $\bm{\Sigma}_{\bm{u}}(r)$ are at least two times continuously differentiable in $r$ at 0.
Intuitively, one may also expect that the diagonal matrix $\bm{\Lambda}_{\bm{u}}(r)$ of its ordered eigenvalues and the corresponding matrix $\bm{P}_{\bm{u}}(r)$ consisting of its eigenvectors to have the similar properties. In particular, when $\rho(x)$, $x\ge 0$ is real analytic on a neighborhood of 0, we can simply follow the proof of Lemma \ref{Lem:Sigma} to show that $\bm{\Sigma}_{\bm{u}}(r)$ is also real analytic on that neighborhood. Then by the main theorem in \cite{kri}, both the eigenvalues and the eigenvectors of $\bm{\Sigma}_{\bm{u}}(r)$ can be parameterized real analytically on a neighborhood of 0. As for the ordering of these eigenvalues, by noting the fact that zeros of a real analytic function indexed by $\mathbb{R}$ are isolated (see, for example, Corollary 1.2.5 in \cite{kra}), we can also show that $\bm{\Lambda}_{\bm{u}}(r)$ and $\bm{P}_{\bm{u}}(r)$ can be both real analytic on a neighborhood of 0. 

However, the above result relies on the assumption that $\rho$ is real analytic, which is not guaranteed even when $\rho$ is infinitely differentiable. Thus, the following condition can be regarded as a generalization of real analyticity of $\rho$ on a neighborhood of 0, such that both the ordered eigenvalues and the corresponding eigenvectors of  $\bm{\Sigma}_{\bm{u}}(r)$ can change smoothly enough under perturbation.

\begin{defn}\label{Def:PC}(Perturbation Condition)
Let $X$ be qualified, and let $\bm{\Sigma}(\bm{t})$, $\bm{P}(\bm{t})$ and $\bm{\Lambda}(\bm{t})$, $\bm{t}\in\mathbb{R}^N\setminus\{\bm{0}_N\}$ be the matrices as defined above. Then $X$ is said to be \textbf{qualified under perturbation} if there exists a version of $\bm{P}(\bm{t})$ and $\bm{\Lambda}(\bm{t})$, such that 
for any direction $\bm{u}=(u_1,\dots,u_N)^T\in\mathbb{S}^{N-1}$, 
\begin{enumerate}
    \item both $\bm{P}_{\bm{u}}(0):=\lim_{r\downarrow 0}\bm{P}_{\bm{u}}(r)$ and $\bm{\Lambda}_{\bm{u}}(0):=\lim_{r\downarrow 0}\bm{\Lambda}_{\bm{u}}(r)$ exist;
    \item there exists a constant $\delta_{pc}\in (0, \delta_{\rho}]$ such that $\bm{P}_{\bm{u}}(r)$ and $\bm{\Lambda}_{\bm{u}}(r)$ are both continuous on $r\in[0,\delta_{pc}]$;
    \item $\bm{P}_{\bm{u}}(r)=\bm{P}_{\bm{u},0}+\bm{P}_{\bm{u},1}r+o(r)$ as $r\to 0$ for some $\bm{P}_{\bm{u},0}$, $\bm{P}_{\bm{u},1}\in\mathbb{R}^{L\times L}$;
    \item $\bm{\Lambda}_{\bm{u}}(r)=\bm{\Lambda}_{\bm{u},0}+\bm{\Lambda}_{\bm{u},1}r+\bm{\Lambda}_{\bm{u},2}r^2+o(r^2)$ as $r\to 0$,
    where $\bm{\Lambda}_{\bm{u},j}$, $j=0,1,2$ are all $L\times L$ real-valued diagonal matrices. More specifically, this is equivalent to
    $$\lambda_{\bm{u},i}(r)=\lambda_{\bm{u},i,0}+\lambda_{\bm{u},i,1}r+\lambda_{\bm{u},i,2}r^2+o(r^2),$$
    where $\lambda_{\bm{u},i,j}\in\mathbb{R}$ for $i=1,\dots,L$ and $j=0,1,2$, such that
    $\bm{\Lambda}_{\bm{u},0}:=\diag(\lambda_{\bm{u},i,0},i=1,\dots,L)$, $\bm{\Lambda}_{\bm{u},1}:=\diag(\lambda_{\bm{u},i,1},i=1,\dots,L)$, and $\bm{\Lambda}_{\bm{u},2}:=\diag(\lambda_{\bm{u},i,2},i=1,\dots,L)$.
\end{enumerate}
\end{defn}
In the following sections,
if $X$ is qualified under perturbation, then $\bm{P}(\bm{t})$ and $\bm{\Lambda}(\bm{t})$ are selected to be a version satisfying all the four conditions in Definition \ref{Def:PC} by default.

\begin{remark}\label{Rem:PC}
Let $X$ be qualified under perturbation and $\bm{u}\in\mathbb{S}^{N-1}$.
Then by Definition \ref{Def:PC}, 
$\bm{A}_{\bm{u}}(r)=\bm{P}_{\bm{u}}(r)\bm{\Lambda}_{\bm{u}}^{1/2}(r)$, $r\ge 0$ is also continuous on $r\in [0,\delta_{pc}]$ and there exist $\bm{A}_{\bm{u},0},\bm{A}_{\bm{u},1/2},\bm{A}_{\bm{u},1}\in\mathbb{R}^{L\times L}$ such that
$$\bm{A}_{\bm{u}}(r)=\bm{A}_{\bm{u},0}+\bm{A}_{\bm{u},1/2}r^{1/2}+\bm{A}_{\bm{u},1}r+o(r)$$
as $r\to 0$. In particular, we have $\bm{A}_{\bm{u},0}=\bm{A}_{\bm{u}}(0)$.

In addition, by the continuities of $\bm{\Sigma}_{\bm{u}}(r)$, $\bm{A}_{\bm{u}}(r)$, $\bm{P}_{\bm{u}}(r)$ and $\bm{\Lambda}_{\bm{u}}(r)$ on $r\in [0,\delta_{pc}]$,
it is easy to see
\begin{enumerate}[label=(\roman*)]
    \item $\bm{P}_{\bm{u},0}$ is orthogonal;
    
    \item $\bm{\Sigma}_{\bm{u},0}$ has the eigen-decomposition
    $$\bm{\Sigma}_{\bm{u},0}=\bm{P}_{\bm{u},0}\bm{\Lambda}_{\bm{u},0}\bm{P}_{\bm{u},0}^T;$$
    \item 
    $\bm{A}_{\bm{u},0}=\bm{P}_{\bm{u},0}\bm{\Lambda}_{\bm{u},0}^{1/2}$.
\end{enumerate}
\end{remark}

\section{Covariance Structure}\label{proj2:sec:cs}
In this section, we have a detailed analysis on $\bm{\Sigma}_{\bm{u}}(r), \bm{A}_{\bm{u}}(r), \bm{P}_{\bm{u}}(r) \text{ and }\bm{\Lambda}_{\bm{u}}(r)$, especially when $r\to 0$.
\subsection{General Covariance Structure}\label{SubSec:GCS}

The following lemma collects some useful results from Lemma \ref{Lem:Sigma} and Definition \ref{Def:PC}.
\begin{lemma}\label{Lem:EigValDecSpd}
Let $X$ be qualified under perturbation and $\bm{u}\in\mathbb{S}^{N-1}$.
Then for any $1\le i\le L$, we have
\begin{enumerate}[label=(\roman*)]
    \item $(\bm{P}_{\bm{u},0}^{(i)})^T\bm{P}_{\bm{u}, 1}^{(i)}=0$;
    \item $\lambda_{\bm{u},i,1}=0$, i.e., $\bm{\Lambda}_{\bm{u},1}=\bm{0}_{L\times L}$; 
    \item if $\lambda_{\bm{u},i,0}=0$, then $\lambda_{\bm{u},i,2}=(\bm{P}_{\bm{u},0}^{(i)})^T\bm{\Sigma}_{\bm{u},2}\bm{P}_{\bm{u},0}^{(i)}$;
    \item $\bm{\Sigma}_{\bm{u},0}\bm{P}_{\bm{u},1}^{(i)}=\lambda_{\bm{u},i,0}\bm{P}_{\bm{u},1}^{(i)}$.
\end{enumerate}
\end{lemma}
\begin{proof}
By Definition \ref{Def:PC} and the orthogonality of $\bm{P}_{\bm{u}}(r)$ for any $r\ge 0$, we have
$$
\begin{aligned}
\bm{I}_{L}
&=\bm{P}_{\bm{u}}(r)^T\bm{P}_{\bm{u}}(r)\\
&=\left(\bm{P}_{\bm{u},0}+\bm{P}_{\bm{u},1}r+o(r)\right)^T\left(\bm{P}_{\bm{u},0}+\bm{P}_{\bm{u},1}r+o(r)\right)\\
&=\bm{I}_{L}+\left(\bm{P}_{\bm{u},0}^T\bm{P}_{\bm{u},1}+\bm{P}_{\bm{u},1}^T\bm{P}_{\bm{u},0}\right)r+o(r),
\end{aligned}
$$
which implies 
$$\bm{P}_{\bm{u},0}^T\bm{P}_{\bm{u},1}+\bm{P}_{\bm{u},1}^T\bm{P}_{\bm{u},0}=\bm{0}_{L\times L}.$$
Thus, for any $1\le i\le L$,
$$(\bm{P}_{\bm{u},0}^{(i)})^T\bm{P}_{\bm{u}, 1}^{(i)}=\frac{1}{2}\left(\bm{P}_{\bm{u},0}^T\bm{P}_{\bm{u},1}+\bm{P}_{\bm{u},1}^T\bm{P}_{\bm{u},0}\right)[i,i]=0.$$

For (ii), note that
$$
\begin{aligned}
\bm{\Sigma}_{\bm{u}}(r)\bm{P}_{\bm{u}}(r)^{(i)}
&=\left(\bm{\Sigma}_{\bm{u},0}+\bm{\Sigma}_{\bm{u},2}r^2+o(r^2)\right)\left(\bm{P}_{\bm{u},0}^{(i)}+\bm{P}_{\bm{u},1}^{(i)}r+o(r)\right)\\
&=\bm{\Sigma}_{\bm{u},0}\bm{P}_{\bm{u},0}^{(i)}+\left(\bm{\Sigma}_{\bm{u},0}\bm{P}_{\bm{u},1}^{(i)}\right)r+o(r)
\end{aligned}
$$
and
\begin{equation}\label{Eq:lambdaP}
    \begin{aligned}
        \lambda_{\bm{u},i}(r)\bm{P}_{\bm{u}}(r)^{(i)}
        &=\left(\lambda_{\bm{u},i,0}+\lambda_{\bm{u},i,1}r+\lambda_{\bm{u},i,2}r^2+o(r^2)\right)\left(\bm{P}_{\bm{u},0}^{(i)}+\bm{P}_{\bm{u},1}^{(i)}r+o(r)\right)\\
        &=\lambda_{\bm{u},i,0}\bm{P}_{\bm{u},0}^{(i)}+\left(\lambda_{\bm{u},i,0}\bm{P}_{\bm{u},1}^{(i)}+\lambda_{\bm{u},i,1}\bm{P}_{\bm{u},0}^{(i)}\right)r+o(r).
    \end{aligned}    
\end{equation}
Then by $\bm{\Sigma}_{\bm{u}}(r)\bm{P}_{\bm{u}}(r)^{(i)}=\lambda_{\bm{u},i}(r)\bm{P}_{\bm{u}}(r)^{(i)}$, we get
\begin{equation}\label{sp}
    \lambda_{\bm{u},i,0}\bm{P}_{\bm{u},1}^{(i)}+\lambda_{\bm{u},i,1}\bm{P}_{\bm{u},0}^{(i)}=\bm{\Sigma}_{\bm{u},0}\bm{P}_{\bm{u},1}^{(i)}.    
\end{equation}
By left-multiplying $(\bm{P}_{\bm{u},0}^{(i)})^T$ on the both sides of (\ref{sp}), (i) of this lemma, (i) and (ii) in Remark \ref{Rem:PC}, and the symmetry of $\bm{\Sigma}_{\bm{u},0}$, we have for any $1\le i\le L$,
$$\lambda_{\bm{u},i,1}=\left(\bm{P}_{\bm{u},0}^{(i)}\right)^T\bm{\Sigma}_{\bm{u},0}\bm{P}_{\bm{u},1}^{(i)}=\left(\bm{P}_{\bm{u},1}^{(i)}\right)^T\bm{\Sigma}_{\bm{u},0}\bm{P}_{\bm{u},0}^{(i)}=\lambda_{\bm{u},i,0}\left(\bm{P}_{\bm{u},1}^{(i)}\right)^T\bm{P}_{\bm{u},0}^{(i)}=0.$$

For (iii), by (ii) and $\lambda_{\bm{u},i,0}=0$, (\ref{Eq:lambdaP}) becomes
$$
\begin{aligned}
\lambda_{\bm{u},i}(r)\bm{P}_{\bm{u}}(r)^{(i)}
&=\left(\lambda_{\bm{u},i,2}r^2+o(r^2)\right)\left(\bm{P}_{\bm{u},0}^{(i)}+\bm{P}_{\bm{u},1}^{(i)}r+o(r)\right)\\
&=\lambda_{\bm{u},i,2}\bm{P}_{\bm{u},0}^{(i)}r^2+o(r^2).
\end{aligned}
$$
Similarly as in (ii), by left-multiplying both sides of $\bm{\Sigma}_{\bm{u}}(r)\bm{P}_{\bm{u}}(r)^{(i)}=\lambda_{\bm{u},i}(r)\bm{P}_{\bm{u}}(r)^{(i)}$ by $\left(\bm{P}_{\bm{u},0}^{(i)}\right)^T$ and comparing the coefficients for $r^2$, we have
$$\lambda_{\bm{u},i,2}=\left(\bm{P}_{\bm{u},0}^{(i)}\right)^T\bm{\Sigma}_{\bm{u},2}\bm{P}_{\bm{u},0}^{(i)}.$$

Finally, taking (ii) into  (\ref{sp}) yields (iv).
\end{proof}

\begin{remark}\label{Rem:PC2}
Let $X$ be qualified under perturbation and $\bm{u}\in\mathbb{S}^{N-1}$.
By (ii) of Lemma \ref{Lem:EigValDecSpd}, it is easy to see the matrix $\bm{A}_{\bm{u},1/2}$ in Remark \ref{Rem:PC} is $\bm{0}_{L\times L}$, and then we have
$$\bm{A}_{\bm{u}}(r)=\bm{A}_{\bm{u},0}+\bm{A}_{\bm{u},1}r+o(r)$$
as $r\to 0$.  
In addition, by Definition \ref{Def:PC} and (ii) of Lemma \ref{Lem:EigValDecSpd},
\begin{equation}\notag
\lambda_{\bm{u},i}^{1/2}(r)=\left\{
\begin{aligned}
    \lambda_{\bm{u},i,0}^{1/2}+\frac{\lambda_{\bm{u},i,2}}{2(\lambda_{\bm{u},i,0})^{1/2}}r^2+o\left(r^2\right), &&  \text{ for any $1\le i\le \Rank(\bm{\Sigma}_{\bm{u},0})$;}\\
    \lambda_{\bm{u},i,2}^{1/2}r+o(r), &&  \text{otherwise}.
\end{aligned}\right.
\end{equation}
Then by Definition \ref{Def:PC} and since $\bm{A}_{\bm{u}}(r)=\bm{P}_{\bm{u}}(r)\bm{\Lambda}_{\bm{u}}^{1/2}(r)$, we have
\begin{equation}\label{Eq:A1}
\bm{A}_{\bm{u},1}^{(i)}=\left\{
\begin{aligned}
    \lambda_{\bm{u},i,0}^{1/2}\bm{P}_{\bm{u},1}^{(i)}, &&  \text{ for any $1\le i\le \Rank(\bm{\Sigma}_{\bm{u},0})$;}\\
    \lambda_{\bm{u},i,2}^{1/2}\bm{P}_{\bm{u},0}^{(i)}, &&  \text{otherwise}.
\end{aligned}\right.
\end{equation}

\end{remark}

For any $1\le i\le j\le N$, $1\le k\le N$
and $\bm{u}=(u_1,\dots,u_N)^T\in\mathbb{S}^{N-1}$,
define $\bm{H}(\bm{u})=(h_{k,\ell}(\bm{u}))\in\mathbb{R}^{N\times L}$ by
\begin{equation}\label{Hdefn}
h_{k,i+j(j-1)/2}(\bm{u})=\delta_{j,k}u_i+(1-\delta_{j,k})\delta_{i,k}u_j\text{  and  }h_{k,L-1}(\bm{u})=h_{k,L}(\bm{u})=0.    
\end{equation}
For example, when $N=3$, we have
$$\bm{H}(\bm{u})=
\begin{pmatrix}
u_1 & u_2 & 0   &u_3 &0   &0  &0  &0\\
0   & u_1 & u_2 &0   &u_3 &0  &0  &0\\
0   & 0   & 0   &u_1 &u_2 &u_3&0  &0
\end{pmatrix}.
$$
The following lemma collects some useful results about $\bm{H}(\bm{u})$.
\begin{lemma}\label{Lem:H}
Let $X$ be qualified under perturbation and $\bm{u}\in\mathbb{S}^{N-1}$. Then
\begin{enumerate}[label=(\roman*)]
    \item if $\bm{M}\in\mathbb{R}^{N\times N}$ is the $N$-th order matriculation of a vector $\bm{a}\in\mathbb{R}^{n}$ ($n\ge N(N+1)/2$), then $$\bm{H}(\bm{u})\bm{a}=\bm{M}\bm{u};$$
    \item $\bm{\Sigma}_{\bm{u},0}\bm{H}^T(\bm{u})=\bm{0}_{L\times N}$;
    \item $\bm{H}(\bm{u})\bm{A}_{\bm{u}}(r)^{(i)}=\bm{0}_{N}+o(r)$ as $r\to 0$ for any $1\le i\le \Rank(\bm{\Sigma}_{\bm{u},0})$. 
\end{enumerate}
\end{lemma}

\begin{proof}
One can directly check (i) by (\ref{Hdefn}). As for (ii), note that the expressions in (\ref{eq:Sigma0_main}) and (\ref{eq:Sigma0_side}) are both symmetric in $i_1$ and $j_1$.
As a result, for any $1\le i\le L$, the elements in $\Matri\left((\bm{\Sigma}_{\bm{u},0})^{(i)}\right)$ for which the row number $i_1$ is greater than the column number $j_1$ will also be described by these expressions.
Consequently, one can varify that for any $1\le i\le L$,
$$\Matri\left((\bm{\Sigma}_{\bm{u},0})^{(i)}\right)\bm{u}=\bm{0}_N.$$ 
Then by (i) of this lemma and $\bm{\Sigma}_{\bm{u},0}=\bm{\Sigma}_{\bm{u},0}^T$,
$$(\bm{\Sigma}_{\bm{u},0})_{(i)}\bm{H}^T(\bm{u})=\left(\bm{H}(\bm{u})(\bm{\Sigma}_{\bm{u},0})^{(i)}\right)^T=\left(\Matri\left((\bm{\Sigma}_{\bm{u},0})^{(i)}\right)\bm{u}\right)^T=\bm{0}_N^T.$$
As for (iii), by (ii) of this lemma, we have 
\begin{equation}\label{Eq:HA0}
    \bm{H}(\bm{u})\bm{A}_{\bm{u},0}=\bm{0}_{N\times L}.    
\end{equation}
Then
by Remark \ref{Rem:PC2}, it suffices to show that for any $1\le i\le \Rank(\bm{\Sigma}_{\bm{u},0})$,
\begin{equation}\label{HA1}
    \bm{H}(\bm{u})\bm{A}_{\bm{u},1}^{(i)}=\bm{0}_{N}.
\end{equation}
Note that
$\lambda_{\bm{u},i,0}\neq 0$ for any $1\le i\le \Rank(\bm{\Sigma}_{\bm{u},0})$.
By (iii) in Remark \ref{Rem:PC} and (\ref{Eq:HA0}), we have for any $1\le i\le \Rank(\bm{\Sigma}_{\bm{u},0})$,
\begin{equation}\label{HP}
\bm{H}(\bm{u})\bm{P}_{\bm{u},0}^{(i)}=\lambda_{\bm{u},i,0}^{-1/2}\bm{H}(\bm{u})\bm{A}_{\bm{u},0}^{(i)}=\bm{0}_{N}. 
\end{equation}
By (\ref{Eq:A1}),
we have
$$
\bm{A}_{\bm{u},1}^{(i)}=\lambda_{\bm{u},i,0}^{1/2}\bm{P}_{\bm{u},1}^{(i)},
$$
which, together with (iv) of Lemma \ref{Lem:EigValDecSpd}, implies for any $1\le i\le \Rank(\bm{\Sigma}_{\bm{u},0})$,
$$\bm{A}_{\bm{u},1}^{(i)}\in \sp\left\{\bm{P}_{\bm{u}, 0}^{(1)},\dots,\bm{P}_{\bm{u},0}^{(\Rank(\bm{\Sigma}_{\bm{u},0}))}\right\}.$$
Then by (\ref{HP}), (\ref{HA1}) is immediate, and hence completes the proof.
\end{proof}

\subsection{Properties of the Covariance Matrix along a Coordinate Axis}\setcounter{MaxMatrixCols}{20}
In the last subsection, we have explored some properties of $\bm{\Sigma}_{\bm{u}}(r)$ for any $\bm{u}\in\mathbb{S}^{N-1}$.
Let $\bm{u}_0:=(0,\dots,0,1)^T\in\mathbb{R}^L$ which is the direction of the last coordinate axis. By Lemma \ref{Lem:Sigma}, $\bm{\Sigma}_{\bm{u}_0}(0)$ has a simple form. This would be helpful in solving problems that depend on $\bm{\Sigma}_{\bm{u}}(0)$ but are independent of the choice of $\bm{u}\in\mathbb{S}^{N-1}$.
In the following, we focus on the properties of the covariance matrix $\bm{\Sigma}_{\bm{u}_0}(r)$, $r\ge 0$. For conciseness, we will drop $\bm{u}_0$ from subscripts.

By Lemma \ref{Lem:Sigma}, it is easy to see the form of $\bm{\Sigma}_{0}$ can be very simple after swapping some of its rows and the corresponding columns. For example, for $N=4$ and $\rho(x)=\exp(-x)$, $x\ge 0$, i.e., the covariance function $R(\bm{t})=\rho(\|\bm{t}\|^2)=\exp(-\|\bm{t}\|^2)$, $\bm{t}\in\mathbb{R}^4$, we have
$$\bm{\Sigma}_{0}=
\begin{pmatrix}
\frac{32}{3}   &0  &\frac{8}{3}   & 0   & 0 &\frac{8}{3}    &0    &0    &0     &0  &-\frac{4}{3} &-\frac{4}{3}\\
0    &4  &0   & 0   & 0 &0    &0    &0    &0     &0  &0  &0\\
\frac{8}{3}    &0  &\frac{32}{3}  &  0  & 0 & \frac{8}{3} &  0    &0    &0     &0 &-\frac{4}{3} &-\frac{4}{3}\\
0    &0  &0   & 4   & 0 &0  &  0    &0    &0     &0  &0  &0\\
0    &0  &0   & 0   & 4 &0  &  0    &0    &0     &0  &0  &0\\
\frac{8}{3}    &0  &\frac{8}{3}   & 0   & 0 &\frac{32}{3} &   0   & 0   &0     &0 &-\frac{4}{3} &-\frac{4}{3}\\
0    &0  &0   & 0   & 0 & 0 &   0   & 0   &0     &0  &0  &0\\
0    &0  &0   & 0   & 0 & 0 &   0   & 0   &0     &0  &0  &0\\
0    &0  &0   & 0   & 0 & 0 &   0   & 0    &0    & 0  &0  &0\\
0    &0  &0   & 0   & 0 & 0 &   0   & 0    &0    & 0  &0  &0\\
-\frac{4}{3}   &0  &-\frac{4}{3}  &  0  & 0 &-\frac{4}{3} &   0   & 0    &0    & 0  &\frac{2}{3}  &\frac{2}{3}\\
-\frac{4}{3}   &0  &-\frac{4}{3}  &  0  & 0 &-\frac{4}{3} &   0   & 0    &0    & 0  &\frac{2}{3}  &\frac{2}{3}
\end{pmatrix}.
$$
After swapping some rows and the corresponding columns, it is turned into
$$\bm{\Sigma}'_0=
\begin{pmatrix}
\frac{32}{3}    &\frac{8}{3}   &\frac{8}{3} &-\frac{4}{3} &-\frac{4}{3} & 0 & 0 & 0 & 0 & 0 & 0 & 0 \\
\frac{8}{3}     &\frac{32}{3}  &\frac{8}{3} &-\frac{4}{3} &-\frac{4}{3} & 0 & 0 & 0 & 0 & 0 & 0 & 0 \\
\frac{8}{3}     &\frac{8}{3}   &\frac{32}{3}&-\frac{4}{3} &-\frac{4}{3} & 0 & 0 & 0 & 0 & 0 & 0 & 0 \\
-\frac{4}{3}    &-\frac{4}{3}  &-\frac{4}{3}&\frac{2}{3}  &\frac{2}{3}  & 0 & 0 & 0 & 0 & 0 & 0 & 0 \\
-\frac{4}{3}    &-\frac{4}{3}  &-\frac{4}{3}&\frac{2}{3}  &\frac{2}{3}  & 0 & 0 & 0 & 0 & 0 & 0 & 0 \\
0               &0             &0           &0            &0            & 4 & 0 & 0 & 0 & 0 & 0 & 0 \\
0               &0             &0           &0            &0            & 0 & 4 & 0 & 0 & 0 & 0 & 0 \\
0               &0             &0           &0            &0            & 0 & 0 & 4 & 0 & 0 & 0 & 0 \\
0               &0             &0           &0            &0            & 0 & 0 & 0 & 0 & 0 & 0 & 0 \\
0               &0             &0           &0            &0            & 0 & 0 & 0 & 0 & 0 & 0 & 0 \\
0               &0             &0           &0            &0            & 0 & 0 & 0 & 0 & 0 & 0 & 0 \\
0               &0             &0           &0            &0            & 0 & 0 & 0 & 0 & 0 & 0 & 0 \\
\end{pmatrix}.
$$
Indeed, one can easily check that $\bm{\Sigma}'_0$ is the limit of the covariance matrix of 
$$
\begin{aligned}
&(X_{11}(\bm{u}_0r),X_{22}(\bm{u}_0r),X_{33}(\bm{u}_0r),X(\bm{u}_0r),X(\bm{0}),X_{12}(\bm{u}_0r),X_{13}(\bm{u}_0r),X_{23}(\bm{u}_0r),\\
&~~~~X_{14}(\bm{u}_0r),X_{24}(\bm{u}_0r),X_{34}(\bm{u}_0r),X_{44}(\bm{u}_0r)|\triangledown X(\bm{u}_0r)=\triangledown X(\bm{0})=\bm{0}_N)
\end{aligned}
$$
as $r\to 0$.
In general, we can rearrange the elements of the random vector 
$$(\triangledown^2 X(\bm{t}),X(\bm{t}),X(\bm{0})|\triangledown X(\bm{t})=\triangledown X(\bm{0})=\bm{0}_N)$$
such that the limiting covariance matrix $\bm{\Sigma}'_0$ of the random vector after the rearrangement has the form
\begin{equation}\label{sigmap}
\bm{\Sigma}'_0=
\begin{pmatrix}
\bm{B}_0 & \bm{B}_2 &\bm{0} & \bm{0}\\
\bm{B}_2^T & \bm{B}_3 & \bm{0} & \bm{0}\\
\bm{0} & \bm{0} & \bm{B}_1 & \bm{0}\\
\bm{0} & \bm{0} & \bm{0} & \bm{0}_{N\times N} \\
\end{pmatrix},
\end{equation}
where $\bm{B}_0\in\mathbb{R}^{(N-1)\times (N-1)}$, $\bm{B}_1\in\mathbb{R}^{(N-1)(N-2)/2\times (N-1)(N-2)/2}$, $\bm{B}_2\in\mathbb{R}^{(N-1)\times 2}$, and $\bm{B}_3\in\mathbb{R}^{2\times 2}$ satisfy that
$$\bm{B}_0[i_0,j_0]=4\left(\frac{2}{3}+2\delta_{i_0,j_0}\right)\rho^{(2)}(0)\text{ for any $1\le i_0,j_0\le N-1$},$$
$$\bm{B}_1[i_1,j_1]=4\rho^{(2)}(0)\delta_{i_1,j_1}\text{ for any $1\le i_1,j_1\le (N-1)(N-2)/2$},$$
$$\bm{B}_2[i_2,j_2]=\frac{4}{3}\rho^{(1)}(0)\text{ for any $1\le i_2\le N-1$ and $1\le j_2\le 2$},$$
and
$$\bm{B}_3[i_3,j_3]=1-\frac{\rho^{(1)}(0)^2}{3\rho^{(2)}(0)}\text{ for any $1\le i_3,j_3\le 2$}.$$
Indeed, $\bm{B}_0$ corresponds to the elements in (\ref{eq:Sigma0_main}) with $1\le i_1,j_1,i_2,j_2\le N-1$, $i_1=j_1$ and $i_2=j_2$; $\bm{B}_1$ corresponds to the elements in (\ref{eq:Sigma0_main}) with $1\le i_1,j_1,i_2,j_2\le N-1$, $i_1< j_1$ and $i_2< j_2$; $\bm{B}_2$ corresponds to the elements in (\ref{eq:Sigma0_side}) with $1\le i_1=j_1\le N-1$; $\bm{B}_3$ corresponds to the elements in (\ref{eq:Sigma0_corner}).
It is noticeable that $\bm{\Sigma}_{0}$ and $\bm{\Sigma}'_0$ share the same eigenvalues (but with different eigenspaces).
The following lemma introduces some properties of the eigenvalues and eigenvectors of $\bm{\Sigma}_{0}$.

\begin{lemma}\label{Lem:EigenSigma}
Let
$$\bm{W}:=\begin{pmatrix}
a & b\\
c & d
\end{pmatrix}\in\mathbb{R}^{2\times 2},$$
where 
    $$a=\frac{1}{3}(32+8(N-2))\rho^{(2)}(0),~~b=\frac{8}{3}\rho^{(1)}(0),$$ 
    $$c=\frac{4}{3}(N-1)\rho^{(1)}(0)\text{ and } d=2\left(1-\frac{\rho^{(1)}(0)^2}{3\rho^{(2)}(0)}\right).$$
Then we have
\begin{enumerate}[label=(\roman*)]
    \item The eigenvalues
    $$\lambda_{+}=\frac{a+d+\sqrt{(a-d)^2+4bc}}{2}\text{ and }\lambda_{-}=\frac{a+d-\sqrt{(a-d)^2+4bc}}{2}$$ of $\bm{W}$ are also two different eigenvalues of $\bm{\Sigma}_{0}$, i.e., there exist integers $1\le l<s\le L$ such that
    $$\lambda_{l,0}=\lambda_{+} \text{ and }\lambda_{s,0}=\lambda_{-}.$$
    Moreover, we have 
    $$\lambda_{l,0}>8\rho^{(2)}(0)\text{ and }\lambda_{s,0}\neq 0.$$
    
    \item 0 is an eigenvalue of $\bm{\Sigma}_{0}$ with multiplicity  $N+1$ and its eigenvector, $\bm{p}_0$, must satisfy 
    $$\bm{p}_0[i+j(j-1)/2]=0\text{ for any }1\le i\le j\le N-1$$
    and
    $$\bm{p}_0[L-1]+\bm{p}_0[L]=0.$$

    \item $4\rho^{(2)}(0)$ is an eigenvalue of $\bm{\Sigma}_{0}$, and 
    if $\lambda_{s,0}\neq 4\rho^{(2)}(0)$, then the multiplicity of $4\rho^{(2)}(0)$ as an eigenvalue of $\bm{\Sigma}_{0}$ is
    $(N-1)(N-2)/2$.
    
    \item 
    $8\rho^{(2)}(0)$ is an eigenvalue of $\bm{\Sigma}_{0}$, and
    if $\lambda_{s,0}\neq 8\rho^{(2)}(0)$, then the multiplicity of $8\rho^{(2)}(0)$ as an eigenvalue of $\bm{\Sigma}_{0}$ is
    $N-2$.
    
    \item For any nonzero eigenvalue of $\bm{\Sigma}_{0}$, its eigenvector, $\bm{p}_{nz}$, must satisfy
    \begin{equation}\label{p1}
    \bm{p}_{nz}[L-1]=\bm{p}_{nz}[L]    
    \end{equation}
    and
    \begin{equation}\label{p2}
    \bm{p}_{nz}[i+N(N-1)/2]=0\text{ for any }1\le i\le N.   
    \end{equation}

    \item If $\lambda_{s,0}\notin\{4\rho^{(2)}(0), 8\rho^{(2)}(0)\}$, then any eigenvector, $\bm{p}_{*}$, of $4\rho^{(2)}(0)$ or $8\rho^{(2)}(0)$ must satisfy
    \begin{equation}\notag
        \bm{p}_{*}[L-1]=\bm{p}_{*}[L]=0.
    \end{equation}
    
    \item If $\lambda_{s,0}\notin\{4\rho^{(2)}(0), 8\rho^{(2)}(0)\}$, then for $\lambda_{l,0}$ and $\lambda_{s,0}$ as eigenvalues of $\bm{\Sigma}_{0}$, any eigenvector, $\bm{p}$, of them must have the form:
    $$\bm{p}[i+j(j-1)/2]=\delta_{i,j}x, \text{ for any }1\le i\le j\le N-1,$$
    $$\bm{p}[i+N(N-1)/2]=0\text{ for any }1\le i\le N,  \text{ and  }\bm{p}[L-1]=\bm{p}[L]=y,$$
    where $x$ and $y$ are both non-zero.

    \item There exists a constant $C>0$ such that 
    $$\tilde{\lambda}_{s,0}<4\tilde{\rho}^{(2)}(0),$$
    where $\tilde{\lambda}_{s,0}$ is the analog of $\lambda_{s,0}$ defined using the \textbf{rescaled} covariance function $$\tilde{\rho}(\|\bm{t}\|^2):=\rho(C\|\bm{t}\|^2)\text{ for any }\bm{t}\in\mathbb{R}^N\setminus\{\bm{0}_N\}.$$
\end{enumerate}
\end{lemma}

\begin{proof}
In (\ref{sigmap}), 
denote 
$$\widetilde{\bm{B}}:=
\begin{pmatrix}
\bm{B}_0 & \bm{B}_2\\
\bm{B}_2^T & \bm{B}_3
\end{pmatrix}\in\mathbb{R}^{(N+1)\times(N+1)}.
$$
Since $\det(\widetilde{\bm{B}}-\lambda\bm{I}_{N+1})=0$ for some $\lambda\in\mathbb{R}$ implies $\det(\bm{\Sigma}'_0-\lambda\bm{I}_L)=0$, for (i), it suffices to show that an eigenvalue of $\bm{W}$ must also be an eigenvalue of $\widetilde{\bm{B}}$. Indeed, assume $\Tilde{q}_1$ and $\Tilde{q}_N$ satisfy
$$(\bm{W}-\lambda\bm{I}_2)(\Tilde{q}_1,\Tilde{q}_N)^T=\bm{0}_2,$$
then it is easy to see that $\Tilde{\bm{q}}:=(\Tilde{q}_1,\dots,\Tilde{q}_{N+1})^T\in\mathbb{R}^{N+1}$ with $\Tilde{q}_1=\cdots=\Tilde{q}_{N-1}$ and $\Tilde{q}_{N}=\Tilde{q}_{N+1}$ satisfies
\begin{equation}\label{Eq:Bq}
(\widetilde{\bm{B}}-\lambda\bm{I}_{N+1})\Tilde{\bm{q}}=0. 
\end{equation}
%Let $\Tilde{\bm{q}}:=(\Tilde{q}_1,\dots,\Tilde{q}_{N+1})^T\in\mathbb{R}^{N+1}\setminus\{\bm{0}_{N+1}\}$ and $\lambda\in\mathbb{R}$ satisfy
%\begin{equation}\label{Eq:Bq}
%(\widetilde{\bm{B}}-\lambda\bm{I}_{N+1})\Tilde{\bm{q}}=0. 
%\end{equation}
%It is easy to see that (\ref{Eq:Bq}) implies 
%$$\Tilde{q}_1=\cdots=\Tilde{q}_{N-1}\text{ and }\Tilde{q}_{N}=\Tilde{q}_{N+1}.$$
%Then (\ref{Eq:Bq}) is equivalent to
%$$(\bm{W}-\lambda\bm{I}_2)(\Tilde{q}_1,\Tilde{q}_N)^T=\bm{0}_2.$$

Thus, $\lambda_{+}$ and $\lambda_{-}$ are both eigenvalues of $\bm{\Sigma}_{0}$.
Since $bc>0$ and $\rho^{(2)}(0)>0$ (see (\ref{ineq:rho1rho2})), we have
$$\lambda_{l,0}-\lambda_{s,0}\ge\sqrt{4bc}>0$$
and
$$\lambda_{l,0}=\frac{a+d+\sqrt{(a-d)^2+4bc}}{2}\ge \frac{a+d+|a-d|}{2}\ge \max(a,d)>8\rho^{(2)}(0).$$
Moreover, if $\lambda_{s,0}=0$, then $ad=bc$. After some calculation, this implies
$$\frac{\rho^{(2)}(0)}{\rho^{(1)}(0)^2 }\le \frac{N}{N+2}.$$
However, by Proposition 3.3 in \cite{che} and Condition (4) in Definition \ref{Def:qualified}, we can get
$$\frac{\rho^{(2)}(0)}{\rho^{(1)}(0)^2 }>\frac{N}{N+2},$$
which leads to a contradiction. Thus, $\lambda_{s,0}\neq 0$.

As for (ii),  we can observe from (\ref{sigmap}) that 
$$L-N-2\le\Rank(\bm{\Sigma}_{0})=\Rank(\bm{\Sigma}'_0)\le L-N-1,$$
and the only uncertainty of $\Rank(\bm{\Sigma}_{0})$ comes from 
$\widetilde{\bm{B}}$, where the last two rows (and columns) are the same, and hence, one of them can be dropped in the subsequent discussion. Thus,
\begin{center}
    $\Rank(\bm{\Sigma}_{0})= L-N-1$ if and only if $\widetilde{\bm{B}}[1:N,1:N]$ is non-degenerate.
\end{center}
By solving the equation $\widetilde{\bm{B}}[1:N,1:N]\bm{q}=\bm{0}_{N}$
for $\bm{q}=(q_1,\dots,q_N)^T\in\mathbb{R}^{N}$, we see that the non-degeneracy of $\widetilde{\bm{B}}[1:N,1:N]$ is equivalent to the equation
$\bm{W}(q_1,q_N)^T=\bm{0}_2$ not having any non-trivial solutions, i.e.,
$$\det(\bm{W})=\lambda_{l,0}\lambda_{s,0}\neq 0,$$
which is obvious by (i). Thus, $\Rank(\bm{\Sigma}_{0})= L-N-1$, and then the multiplicity of 0 is $N+1$.
By (\ref{sigmap}) and solving the equation $\bm{\Sigma}_{0}\bm{p}_0=\bm{0}_L$ for $\bm{p}_0\in\mathbb{R}^L$, we have
\begin{equation}\label{p0}
    \bm{p}_0[i+j(j-1)/2]=0\text{ for any }1\le i\le j\le N-1    
\end{equation}
and
\begin{equation}\label{p0'}
    \bm{p}_0[L-1]+\bm{p}_0[L]=0.    
\end{equation}

As for (iii), similarly, we can observe from (\ref{sigmap}) that 
$$\frac{1}{2}(N-1)(N-2)-1\le \Rank(\bm{\Sigma}_{0}-4\rho^{(2)}(0)\bm{I}_L)\le \frac{1}{2}(N-1)(N-2),$$
and the only uncertainty of $\Rank(\bm{\Sigma}_{0}-4\rho^{(2)}(0)\bm{I}_L)$ comes from 
$\Rank(\widetilde{\bm{B}}[1:N,1:N]-4\rho^{(2)}(0)\bm{I}_N)$, i.e.,
\begin{center}
   $\Rank(\bm{\Sigma}_{0}-4\rho^{(2)}(0)\bm{I}_L)=\frac{1}{2}(N-1)(N-2)$ if and only if $\widetilde{\bm{B}}[1:N,1:N]-4\rho^{(2)}(0)\bm{I}_N$ is non-degenerate. 
\end{center}
By solving the equation $(\widetilde{\bm{B}}[1:N,1:N]-4\rho^{(2)}(0)\bm{I}_N)\bm{q}'=\bm{0}_{N}$
for $\bm{q}'=(q'_1,\dots,q'_N)^T\in\mathbb{R}^{N}$, we see that the
non-degeneracy of $\widetilde{\bm{B}}[1:N,1:N]-4\rho^{(2)}(0)\bm{I}_N$ is equivalent to
$$
\det(\bm{W}-4\rho^{(2)}(0) \bm{I}_2)\neq 0.
$$
Then by (i), this is equivalent to $\lambda_{s,0}\neq 4\rho^{(2)}(0)$ as we desired.
Note that the proof for (iv) is only an analog of (iii) by replacing $4\rho^{(2)}(0)$ with $8\rho^{(2)}(0)$.

As for (v), note that $\bm{j}=(0,\dots,0,1,-1)^T\in\mathbb{R}^L$ is in the eigenspace of 0. Thus, by the orthogonality of eigenspaces, any eigenvector, $\bm{p}_{nz}$, of a nonzero eigenvalue satisfies
$\bm{p}_{nz}[L-1]=\bm{p}_{nz}[L]$  
as stated in (\ref{p1}). 
Then by (\ref{p1}), (\ref{p0}), (\ref{p0'}) and the orthogonality of eigenspaces, (\ref{p2}) is immediate.

As for (vi),
let $\bm{p}_{k,\ell}$, $1\le k<\ell\le N-1$ be $(N-1)(N-2)/2$  vectors in $\mathbb{R}^{L}$ such that
\begin{equation}\label{basis1part1}
\bm{p}_{k,\ell}[i+j(j-1)/2]=\delta_{i,k}\delta_{j,\ell}  \text{ for any }1\le i<j\le N-1, 
\end{equation}
$$\bm{p}_{k,\ell}[i+N(N-1)/2]=0\text{ for any }1\le i\le N,$$
and
\begin{equation}\label{basis1}
\bm{p}_{k,\ell}[L-1]=\bm{p}_{k,\ell}[L]=0.    
\end{equation}
Then by (iii) and (\ref{sigmap}), it is easy to check that $\bm{p}_{k,\ell}$, $1\le k<\ell\le N-1$ are linearly independent and form a basis of the eigenspace of $4\rho^{(2)}(0)$.

Similarly, let $\bm{p}'_{k}$, $1\le k\le N-2$ be $N-2$ linearly independent vectors in $\mathbb{R}^{L}$ such that for any $1\le k\le N-2$,
\begin{equation}\label{primeprime}
    \sum_{i=1}^{N-1}\bm{p}'_{k}[i+i(i-1)/2]=0,~~\bm{p}'_{k}[N+N(N-1)/2]=0,
\end{equation}
$$\bm{p}'_{k}[i+j(j-1)/2]=0\text{ for any }1\le i< j\le N,$$
and
\begin{equation}\label{basis2}
\bm{p}'_{k}[L-1]=\bm{p}''_{k}[L]=0.     
\end{equation}
Then by (iv) and (\ref{sigmap}), it is easy to check that $\bm{p}'_{k}$, $1\le k\le N-2$ form a basis of the eigenspace of $8\rho^{(2)}(0)$.
Then combining (\ref{basis1}), (\ref{basis2}) %and the orthogonality of eigenspaces 
yields (vi). 

%By the orthogonality of eigenspaces and (\ref{p2}), this implies for any eigenvalue of $\bm{\Sigma}_{0}$ not equal to $4\rho^{(2)}(0)$, its eigenvector, $\bm{p}'$, must satisfy
%\begin{equation}\label{p3}
%\bm{p}'[i+j(j-1)/2]=0\text{ for }1\le i< j\le N.    
%\end{equation}

As for (vii),
by the orthogonality of eigenspaces %, (vi), (\ref{p2}), (\ref{p3}) and (\ref{primeprime}), 
and (\ref{basis1part1}),
we have for any nonzero eigenvalue of $\bm{\Sigma}_0$ not equal to $4\rho^{(2)}(0)$ or $8\rho^{(2)}(0)$, its eigenvector, $\bm{p}$ must have
\begin{equation}\label{p4}
    \bm{p}[i+j(j-1)/2]=\delta_{i,j}x,
\end{equation}
for a constant $x\in \mathbb{R}$ and any $1\le i\le j\le N-1$. By (\ref{p1}), (\ref{p2}) and (\ref{p4}), 
the only thing left is to show $xy\neq 0$. Let $\bm{p}_l,\bm{p}_s\in\mathbb{R}^L$ be eigenvectors of $\lambda_{l,0}$ and $\lambda_{s,0}$ respectively, such that
$$\bm{p}_l[i+j(j-1)/2]=\delta_{i,j}x_l, \text{ for any }1\le i\le j\le N-1, $$
$$\bm{p}_s[i+j(j-1)/2]=\delta_{i,j}x_s, \text{ for any }1\le i\le j\le N-1, $$
$$\bm{p}_l[L-1]=\bm{p}_l[L]=y_l\text{ and }\bm{p}_s[L-1]=\bm{p}_s[L]=y_s.$$
Then it suffices to show $x_ly_lx_sy_s\neq 0$.
Suppose $x_ly_lx_sy_s=0$. By the orthogonality of eigenspaces, 
\begin{equation}\notag
(N-1)x_lx_s+2y_ly_s=0.    
\end{equation}
Then we must have
$$x_lx_s=y_ly_s=0.$$
Without loss of generality, suppose $x_l=0$. Then $y_l\neq 0$ since $\bm{p}_l$ is nonzero. By checking the first rows of both sides of the equation
$(\bm{\Sigma}_{0}-\lambda_{l,0}\bm{I}_L)\bm{p}_l=\bm{0}_L$, (\ref{eq:Sigma0_side}) and (i), we have
$$\frac{8}{3}\rho^{(1)}(0)y_l=(\bm{\Sigma}_{0}[1,L-1]+\bm{\Sigma}_{0}[1,L])y_l=0.$$
This implies $\rho^{(1)}(0)=0$, which leads to a contradiction with (\ref{ineq:rho1rho2}). Therefore, $x_ly_lx_sy_s\neq 0$ as we desired.

As for (viii), note that for any $C>0$ 
$$\tilde{\lambda}_{s,0}=\frac{\tilde{a}+\tilde{d}-\sqrt{(\tilde{a}-\tilde{d})^2+4\tilde{b}\tilde{c}}}{2},$$
where 
$$\tilde{a}=\frac{1}{3}(32+8(N-2))C^2\rho^{(2)}(0),~~\tilde{b}=\frac{8}{3}C\rho^{(1)}(0),$$ 
    $$\tilde{c}=\frac{4}{3}(N-1)C\rho^{(1)}(0)\text{ and } \tilde{d}=2\left(1-\frac{\rho^{(1)}(0)^2}{3\rho^{(2)}(0)}\right).$$
Then we can define
$$k_1:=\tilde{a}C^{-2},~~k_2:=\tilde{b}C^{-1},~~k_3:=\tilde{c}C^{-1}, \text{ and }k_4:=\tilde{d},$$
and by (\ref{ineq:rho1rho2}) and (\ref{rho1rho2}), we have $k_i\neq 0$ for $i=1,2,3,4$.
In particular, we have
$$k_1=\frac{1}{3}(32+8(N-2))\rho^{(2)}(0)>8\rho^{(2)}(0).$$
Then by (i), the inequality
$\tilde{\lambda}_{s,0}<4\rho^{(2)}(0)C^2$ holds if and only if
$$k_1C^2+k_4-\sqrt{(k_1C^2-k_4)^2+4k_2k_3C^2}<8\rho^{(2)}(0) C^2.$$
Thus, it suffices to have
$$f(C):=((k_1-8\rho^{(2)}(0))C^2+k_4)^2-(k_1C^2-k_4)^2-4k_2k_3C^2<0.$$
Note that $f$ is a polynomial of $C$ with degree four and its coefficient of $C^4$ is 
$$(k_1-8\rho^{(2)}(0))^2-k_1^2=-16\rho^{(2)}(0) k_1+64\rho^{(2)}(0)^2=16\rho^{(2)}(0)(4\rho^{(2)}(0)-k_1)<0.$$
Thus, there exists a constant $C_0>0$ such that 
$f(C)<0$ for any $C>C_0$. This implies $\tilde{\lambda}_{s,0}<4\rho^{(2)}(0)C^2=4\tilde{\rho}^{(2)}(0)$ for any $C>C_0$.
\end{proof}

\begin{remark}\label{Rem:scale}
Recall that $L=N(N+1)/2+2$.
By Lemma \ref{Lem:EigenSigma}, the sum of multiplicities of $0$, $4\rho^{(2)}(0)$, $8\rho^{(2)}(0)$, $\lambda_{s,0}$, and $\lambda_{t,0}$ is equal to $L$, which implies they are the only eigenvalues of $\bm{\Sigma}_0$. 
The condition in (vii) of Lemma \ref{Lem:EigenSigma}, i.e.,
$$\lambda_{s,0}\notin\{4\rho^{(2)}(0), 8\rho^{(2)}(0)\},$$
ensures that these eigenvalues are distinct.

However, 
if the problem of interest is independent of the choice of $C$ in the covariance function $R(\bm{t})=\rho(C\|\bm{t}\|^2)$, $\bm{t}\in\mathbb{R}^N$, then by (viii) of Lemma \ref{Lem:EigenSigma}, we can assume that this condition always holds,
since it can be achieved by a suitable rescaling.
\end{remark}

\begin{lemma}\label{Lem:speed}
Let $X$ be qualified under perturbation. Then
\begin{enumerate}[label=(\roman*)]
    \item there exists an integer $L-N\le i\le L$ such that $\lambda_{i,2}> 0$;
    \item there exists an integer $L-N\le i\le L$ such that $\lambda_{i,2}= 0$.
\end{enumerate}
\end{lemma}

\begin{proof}
By (ii) of Lemma \ref{Lem:EigValDecSpd} and (ii) of Lemma \ref{Lem:EigenSigma}, $\lambda_{i,0}=\lambda_{i,1}=0$ for any $L-N\le i\le L$. Then by Definition \ref{Def:PC},
$$\lambda_{i,2}=\lim_{r\to 0}\frac{\lambda(r)-\lambda_{i,0}-\lambda_{i,1}r}{r^2}=\lim_{r\to 0}\frac{\lambda(r)}{r^2}\ge 0.$$
Thus, for (i), it suffices to show $\lambda_{i,2}\neq 0$ for some $L-N\le i\le L$.
Let $\bm{p}\in\mathbb{R}^{L}$ satisfy $$\bm{p}[k]=\delta_{k,N+N(N-1)/2} \text{ for }k=1,\dots,L.$$
Then $\bm{p}^T=\bm{H}(\bm{u}_0)_{(N)}$, and by (ii) of Lemma \ref{Lem:H},
$$\bm{p}^T\bm{\Sigma}_{0}\bm{p}=0.$$
Recall that in Lemma \ref{Lem:Sigma} and (\ref{Rel:1830}),
$\alpha=\rho^{(1)}(0)^{-1}\rho^{(2)}(0)^2$ and $\beta=\rho^{(3)}(0).$ 
Then
$$
\bm{p}^T\bm{\Sigma}_{2}\bm{p}=\bm{\Sigma}_{2}[N+N(N-1)/2,N+N(N-1)/2]
=18\alpha-30\beta\neq 0,
$$
which implies 
\begin{equation}\label{sim}
\bm{p}^T\bm{\Sigma}(r)\bm{p}=\Theta(r^2) \text{ as $r\to 0$}.    
\end{equation}
By $\bm{p}^T=\bm{H}(\bm{u}_0)_{(N)}$, (iii) of Lemma \ref{Lem:H}, and (ii) of Lemma \ref{Lem:EigenSigma}, we can also get 
$$
\begin{aligned}
\bm{p}^T\bm{\Sigma}(r)\bm{p}
&=\left(\bm{p}^T\left(\bm{A}(r)[1:L,1:(L-N-1)];\bm{A}(r)[1:L,(L-N):L]\right)\right)\\
&~~~~\left(\bm{p}^T\left(\bm{A}(r)[1:L,1:(L-N-1)];\bm{A}(r)[1:L,(L-N):L]\right)\right)^T\\
&=\left(\bm{p}^T(\bm{A}(r)[1:L,(L-N):L])\right)\left(\bm{p}^T(\bm{A}(r)[1:L,(L-N):L])\right)^T+o(r^2)\\
&=\left(\bm{p}^T\bm{P}(r)\bm{\Lambda}^{1/2}(r)[1:L,(L-N):L]\right)\left(\bm{p}^T\bm{P}(r)\bm{\Lambda}^{1/2}(r)[1:L,(L-N):L]\right)^T+o(r^2)\\
&=\left(\bm{p}^T\bm{P}(r)\right)\diag\left(0,\dots,0,\lambda_{L-N}(r),\dots,\lambda_{L}(r)\right)\left(\bm{p}^T\bm{P}(r)\right)^T+o(r^2)\\
&=\left(\bm{p}^T\bm{P}(r)\right)\diag\left(0,\dots,0,\lambda_{L-N,2},\dots,\lambda_{L,2}\right)\left(\bm{p}^T\bm{P}(r)\right)^Tr^2+o(r^2),
\end{aligned}
$$
where the semicolons represent juxtaposition operations on matrices.
If $\lambda_{i,2}=0$ for all integer $L-N\le i\le L$,
then from the above equation, we see 
$\bm{p}^T\bm{\Sigma}(r)\bm{p}=o(r^2)$,  which contradicts with (\ref{sim}).

As for (ii), 
by Lemma \ref{Lem:Sigma}, we see
$$\bm{\Sigma}_{0}^{(i+N(N-1)/2)}=\bm{0}_{L}\text{ for any }1\le i\le N,$$ 
which contributes $O(r^{2N})$ in $\det(\bm{\Sigma}(r))$, and
$$\bm{\Sigma}_{0}^{(L-1)}=\bm{\Sigma}_{0}^{(L)}\text{ and }\bm{\Sigma}_{2}^{(L-1)}=\bm{\Sigma}_{2}^{(L)},$$
which contribute $o(r^2)$ in $\det(\bm{\Sigma}(r))$.
Hence we have
$$\det(\bm{\Sigma}(r))=o\left(r^{2N+2}\right)\text{ as $r\to 0$}.$$
Thus, there must be an integer $L-N\le i\le L$ such that $\lambda_i(r)=o(r^2)$, which implies $\lambda_{i,2}=0$ as desired.
\end{proof}

\begin{lemma}\label{Lem:PLJ}
Let $X$ be qualified under perturbation. Denote $\bm{j}:=(0,\dots,0,1,-1)^T\in\mathbb{R}^{L}$.
Then
$$\lambda_{i,2}>0\text{ for any }L-N\le i\le L-1,\text{ and }\lambda_{L,2}=0.$$
Moreover, $\bm{P}_0^{(L)}$ and $\bm{j}$ are linearly dependent.
\end{lemma}
\begin{proof}
Note that by (ii) of Lemma \ref{Lem:speed}, there exists an integer $L-N\le i_*\le L$ such that $\lambda_{i_*,2}= 0$, 
and by (\ref{eq:sigma2_main}),
$$
\begin{aligned}
&~~~~\bm{\Sigma}_2[(N(N-1)/2+1):N(N+1)/2,(N(N-1)/2+1):N(N+1)/2]\\
&=\diag(2\alpha-6\beta,\dots,2\alpha-6\beta,18\alpha-30\beta)_{N\times N}.
\end{aligned}
$$
Then by (iii) of Lemma \ref{Lem:EigValDecSpd} and (ii) of Lemma \ref{Lem:EigenSigma}, we have
\begin{equation}\label{ps2}
\begin{aligned}
0&=\left(\bm{P}_0^{(i_*)}\right)^T\bm{\Sigma}_{2}\bm{P}_0^{(i_*)}\\
&=\left(2\alpha-6\beta\right)\sum_{i=1}^{N-1}\bm{P}_0[i+N(N-1)/2,i_*]^2+(18\alpha-30\beta)\bm{P}_0[N+N(N-1)/2,i_*]^2.
\end{aligned}   
\end{equation}
By (\ref{Rel:1830}), we have $18\alpha-30\beta>0$ and $\beta<0$, which implies $2\alpha-6\beta>0$. Then
by (\ref{ps2}) and (ii) of Lemma \ref{Lem:EigenSigma}, we have
$$\bm{P}_0[i,i_*]=0\text{ for any $1\le i\le N(N+1)/2$},$$
and
$$\bm{P}_0[L-1,i_*]+\bm{P}_0[L,i_*]=0.$$
Since $\bm{P}_0$ is non-degenerate, we have
\begin{center}
$\bm{P}_0^{(i_*)}$ and $\bm{j}$ are linearly dependent.    
\end{center}
Note that this property holds for any column $L-N\le i\le L$ satisfying $\lambda_{i,2}= 0$. However,
by the non-degeneracy of $\bm{P}_0$, there can only be one column of $\bm{P}_0$ satisfying this property. This means that there exists a unique integer $L-N\le i_*\le L$ such that $\lambda_{i_*,2}= 0$. By the continuity of $\bm{P}(r)$ at $r=0$ and (\ref{def:order}), we have $i^*=L$.
\end{proof}

\section{Asymptotic Behavior as \texorpdfstring{$r\to 0$}{the Distance Tends to 0}}\label{proj2:sec:mr1}

\subsection{Main result 1}\label{Sec:M1}
We start from the following Rice's formula as in \cite{adl}. Recall that $\lambda^n$ refers to the $n$-dimensional Lebesgue measure.

\begin{lemma}(Corollary 11.2.2, \cite{adl})\label{Lem:RiceFormula}
Let $T$ be a compact subset of $\mathbb R^N$ with $\lambda^{N-1}(\partial T)<\infty$, $X$ be a centered Gaussian random field defined on $T$. Let $\bm{v}\in \mathbb R^N$, $B$ be an open set in $\mathbb{R}^k$, $k=N(N+1)/2+1$ such that $\partial B$ has Hausdorff dimension $k-1$. Denote by $N_{\bm{v},B}(T)$ the number of points $\bm{t}\in T$ which satisfy $\triangledown X(\bm{t})=\bm{v}$ and $(\triangledown^2X(\bm{t}), X(\bm{t}))\in B$.

Suppose that the distribution of $(\triangledown^2X(\bm{t}),\triangledown X(\bm{t}),X(\bm{t}))$ is non-degenerate for any $\bm{t}\in T$. For any $i,j\in\{1,\dots, N\}$ and $\bm{s},\bm{t}\in T$, denote
$$C_{ij}(\bm{s},\bm{t}):=\Cov\left[X_{ij}(\bm{s}),X_{ij}(\bm{t})\right].$$
Suppose 
\begin{equation}\label{cfi2}
\max_{1\le i,j\le N}\left|C_{ij}(\bm{t},\bm{t})+C_{ij}(\bm{s},\bm{s})-2C_{ij}(\bm{s},\bm{t})\right|\le K|\log(\|\bm{t}-\bm{s}\|)|^{-(1+\alpha)}
\end{equation}
for some finite $K>0$, some $\alpha>0$, and all $\bm{s},\bm{t}\in T$ such that $\|\bm{t}-\bm{s}\|$ is small enough. Then we have $$\E\left[N_{\bm{v},B}(T)\right]=\int_{\bm{t}\in T}\int_{(\bm{x}'',x)\in B}\left|\det(\bm{x}'')\right|p_{\bm{t}}(\bm{x}'',\bm{v},x)d\bm{x}''dxd\bm{t},$$
where $p_{\bm{t}}(\bm{x}'',\bm{x}',x)$ is the density of the Gaussian vector $(\triangledown^2X(\bm{t}),\triangledown X(\bm{t}),X(\bm{t})).$
\end{lemma}

There is a second condition similar to (\ref{cfi2}) on the continuity of $\bm{\Sigma}(\bm{t})$, but it is automatically satisfied in our setting by Lemma \ref{Lem:Sigma}, hence omitted here. 

Let $X$ be qualified under perturbation. Then for any $z\in\mathbb{R}$ and compact set $T\subset \mathbb{R}^N\setminus\{\bm{0}_N\}$ with $\lambda^{N-1}(\partial T)<\infty$, $X(\bm{t})$ conditional on $\triangledown X(\bm{0})=\bm{0}_{N}$ and $X(\bm{0})=z$ is still a Gaussian random field on $\bm{t}\in T$. Our goal is to apply Lemma \ref{Lem:RiceFormula} to it. The following result guarantees that (\ref{cfi2}) holds for this conditional random field.
\begin{lemma}\label{Lem:sixorder}
Let $\{X(\bm{t}),\bm{t}\in \mathbb{R}^N\}$ be a stationary Gaussian random field possessing up to second-order almost sure derivatives. 
Let $r(\bm{t})$, $\bm{t}\in\mathbb{R}^N$ be the covariance function of $X$. Suppose that all of the sixth-order partial derivatives of $r(\bm{t})$ exist at $\bm{t}=\bm{0}_N$.
Then Condition (\ref{cfi2}) holds for all $\bm{s},\bm{t}\in \mathbb{R}^N$ such that $\|\bm{t}-\bm{s}\|$ is small enough.
\end{lemma}
\begin{proof}
See Appendix \ref{subsec:sixorder}.
\end{proof}

There is still one last condition which prevents us from applying Lemma \ref{Lem:RiceFormula} to the conditional random field: after conditioning, the random field may no longer be centered. However, one can easily check that the mean function of this new random field is continuously twice differentiable. As discussed in page 268 of \cite{adl}, (\ref{cfi2}) is only used to ensure that Condition (1) in Theorem 11.2.1 in \cite{adl} holds, which is a moduli of continuity condition. One can easily show that this condition still holds if the random field is shifted by a continuously twice differentiable function. Consequently, we can indeed apply Lemma \ref{Lem:RiceFormula} to $X(\bm{t})$ conditional on $\triangledown X(\bm{0})=\bm{0}_{N}$ and $X(\bm{0})=z$.

Finally, note that by a simple limiting argument, we can replace the condition in Lemma \ref{Lem:RiceFormula} that $T$ be compact by $T\subset\mathbb{R}^N\setminus\{\bm{0}_N\}$ be such that $T\cup \{\bm{0}_N\}$ is compact.

To study the number of critical points of different types using Lemma \ref{Lem:RiceFormula}, we need to set up some notations.

For any $0\le k\le N$, define 
\begin{equation}\notag
    D_k:=\left\{\bm{x}\in\mathbb{R}^{N(N+1)/2}:\Matri_N(\bm{x})\text{ is non-degenerate and has exactly $k$ negative eigenvalues}\right\}.
\end{equation}
Let $p(z)$, $z\in\mathbb{R}$ be the density of $X(\bm{0})$,
let $p(\bm{z})$, $\bm{z}\in\mathbb{R}^N$ be the density of $\triangledown X(\bm{0})$, and let $p(\bm{z},z)$, $(\bm{z},z)\in\mathbb{R}^N\times\mathbb{R}$ be the density of $(\triangledown X(\bm{0}),X(\bm{0}))$. Then by Lemma \ref{Lem:MSD}, $(\triangledown X(\bm{0}),X(\bm{0}))$ are Gaussian with
$$\Cov[X_i(\bm{0}),X(\bm{0})]=0$$
for any $1\le i\le N$, and thus, for any $(\bm{z},z)\in\mathbb{R}^N\times\mathbb{R}$,
$$p(\bm{z},z)=p(\bm{z})p(z).$$
For any $\bm{t}\in\mathbb{R}^N\setminus\{\bm{0}_N\}$,
let $p_{\bm{t}}(\bm{x},\bm{z})$, $\bm{x},\bm{z}\in\mathbb{R}^N$ be the density of $(\triangledown X(\bm{t}),\triangledown X(\bm{0}))$, and let $p_{\bm{t}}(\bm{x}'',\bm{x}',x,\bm{z},z)$, $(\bm{x}'',\bm{x}',x,\bm{z},z)\in\mathbb{R}^{N(N+1)/2}\times\mathbb{R}^{N}\times\mathbb{R}\times\mathbb{R}^{N}\times\mathbb{R}$ be the density of $(\triangledown^2 X(\bm{t}),\triangledown X(\bm{t}),X(\bm{t}),\triangledown X(\bm{0}),X(\bm{0}))$.
For any $\bm{z}\in\mathbb{R}^N$ and $z\in\mathbb{R}$, let $p_{\bm{t}}(\bm{x}'',\bm{x}',x|\bm{z},z)$, $(\bm{x}'',\bm{x}',x)\in\mathbb{R}^{N(N+1)/2}\times\mathbb{R}^{N}\times\mathbb{R}$ be the density of $(\triangledown^2X(\bm{t}),\triangledown X(\bm{t}),X(\bm{t}))$ conditional on $\triangledown X(\bm{0})=\bm{z}$ and $X(\bm{0})=z$. 

Similarly, for any $\bm{t}\in\mathbb{R}^N\setminus\{\bm{0}_N\}$, $\bm{x}',\bm{z}'\in\mathbb{R}^N$, let $p_{\bm{t}}(\bm{x}'',x,z|\bm{x}',\bm{z}')$, $(\bm{x}'',x,z)\in\mathbb{R}^{N(N+1)/2}\times\mathbb{R}\times\mathbb{R}$ be the density of  $(\triangledown^2X(\bm{t}), X(\bm{t}), X(\bm{0}))$ conditional on $\triangledown X(\bm{t})=\bm{x}'$ and $\triangledown X(\bm{0})=\bm{z}'$. In particular, the covariance matrix of the Gaussian density $p_{\bm{t}}(\bm{x}'',x,z|\bm{0},\bm{0})$ is
$\bm{\Sigma}(\bm{t})$ which we carefully studied in the previous sections. In addition, it is clear by symmetry that the mean vector of $p_{\bm{t}}(\bm{x}'',x,z|\bm{0},\bm{0})$ is $\bm{0}_{L}$. Thus, we have
\begin{equation}\label{eq:pt}
p_{\bm{t}}(\bm{x}'',x,z|\bm{0},\bm{0})=\frac{1}{\sqrt{(2\pi)^{L}\det(\bm{\Sigma}(\bm{t}))}}\exp\left(-\frac{1}{2}(\bm{x}'',x,z)\bm{\Sigma}(\bm{t})^{-1}(\bm{x}'',x,z)^T\right). 
\end{equation}

By Lemma \ref{Lem:RiceFormula}, we have for any $0\le k\le N$ and $u\in\mathbb{R}$,
the density, $f_{u,k}(\bm{t})$ of the mean measure of the non-degenerate critical points of $X$ above $u$ with index $k$, conditional on $\triangledown X(\bm{0})=\bm{0}_{N}$ and $X(\bm{0})>u$, is given by
\begin{equation}\label{expression}
    \begin{aligned}
        f_{u,k}(\bm{t})&=P(X(\bm{0})>u)^{-1}\int_{x,z>u}\int_{D_k}\left|\det\left(\Matri_N(\bm{x}'')\right)\right|p_{\bm{t}}(\bm{x}'',\bm{0},x|\bm{0},z)p(z)d\bm{x}''dxdz\\
        &=P(X(\bm{0})>u)^{-1}\int_{x,z>u}\int_{D_k}\left|\det\left(\Matri_N(\bm{x}'')\right)\right|p_{\bm{t}}(\bm{x}'',\bm{0},x,\bm{0},z)p(\bm{0},z)^{-1}p(z)d\bm{x}''dxdz\\
        &=P(X(\bm{0})>u)^{-1}p(\bm{0})^{-1}p_{\bm{t}}(\bm{0},\bm{0})\int_{x,z>u}\int_{D_k}\left|\det\left(\Matri_N(\bm{x}'')\right)\right|p_{\bm{t}}(\bm{x}'',x,z|\bm{0},\bm{0})d\bm{x}''dxdz.
    \end{aligned}
\end{equation}

We can further define 
$$f_{u,+}(\bm{t}):=\sum_{\text{$k$ even}}f_{u,k}(\bm{t})$$
and
$$f_{u,-}(\bm{t}):=\sum_{\text{$k$ odd}}f_{u,k}(\bm{t}),$$
i.e., to replace $D_k$ in (\ref{expression}) with $\{\bm{x}''\in\mathbb{R}^{N(N+1)/2}:\det(\Matri_N(\bm{x}''))>0\}$ and $\{\bm{x}''\in\mathbb{R}^{N(N+1)/2}:\det(\Matri_N(\bm{x}''))<0\}$, respectively. $f_{u,\pm}(\bm{t})$ are clearly continuous functions of $\bm{t}\in\mathbb{R}^N\setminus\{\bm{0}_N\}$.

Moreover, by Lemma 11.2.12 in \cite{adl} with $\bm{f}(\bm{t})=\triangledown X(\bm{t})$ and $\bm{g}(\bm{t})=(\triangledown^2 X(\bm{t}),X(\bm{t}))$ for $\bm{t}\in T$, we have 
\begin{equation}\notag
P\left[\#\left\{\bm{t}\in T: X(\bm{t})>u,\triangledown X(\bm{t})=\bm{0},\det(\triangledown^2X(\bm{t}))=0\right\}=0\right]=1 .
\end{equation}
In other words, with probability one, all the critical points in the excursion set $A_u(X,T)$ are non-degenerate. Hence (\ref{expression}) also holds for the (not necessarily non-degenerate) critical points in general.

The following is our first main result.
\begin{theorem}\label{Theo:MR1}
Let $X$ be qualified under perturbation. Then for any $u\in\mathbb{R}$,
$$\lim_{\|\bm{t}\|\to 0}\frac{f_{u,+}(\bm{t})}{f_{u,-}(\bm{t})}=1.$$
\end{theorem}

Intuitively, when a critical point is very close to another critical point with unknown type, the determinant of its Hessian is equally likely to be positive or negative.

%===========================================================================================================
%===========================================================================================================

\subsection{Preparatory results for the proof of Theorem \ref{Theo:MR1}}\label{Sec:S}

Let $X$ be qualified under perturbation.
In this subsection, we only consider $\bm{t}=\bm{u}_0r$, $r> 0$.
Based on the knowledge of the covariance $\bm{\Sigma}(r)$ in the previous sections, we can establish two lemmas suggesting the asymptotic symmetry between the domains
$$D_{u,+}(r):=\left\{\bm{y}\in\mathbb{R}^{L}:\bm{A}(r)_{(L-1)}\bm{y}>u,\bm{A}(r)_{(L)}\bm{y}>u,\frac{1}{r}\det\left(\Matri_N(\bm{A}(r)\bm{y}) \right)>0\right\}$$
and
$$D_{u,-}(r):=\left\{\bm{y}\in\mathbb{R}^{L}:\bm{A}(r)_{(L-1)}\bm{y}>u,\bm{A}(r)_{(L)}\bm{y}>u,\frac{1}{r}\det\left(\Matri_N(\bm{A}(r)\bm{y}) \right)<0\right\}$$
as $r\to 0$.
This asymptotic symmetry plays an essential role in the proof of Theorem \ref{Theo:MR1}.

We first introduce an important matrix.
For any $r> 0$, let $a_{kl}(r)$ be the element of $\bm{A}(r)$ on the $k$-th row and $l$-th column for any $1\le k,l\le L$.
Define the map $\tau:\{1,\dots,N\}^2\to\{1,\dots,N(N+1)/2\}$, where for any $1\le i\le j\le N$,
$$\tau(j,i)\equiv\tau(i,j):=i+j(j-1)/2.$$
Then for any $\bm{v}=(v_1,\dots,v_N)^T\in\{1,\dots,L\}^N$ and $r>0$, we can define $\bm{B}^{\bm{v}}(r):=(b^{\bm{v}}_{i,j}(r))_{1\le i,j\le N}$, where for any $1\le i,j\le N$,
$$b^{\bm{v}}_{i,j}(r)=a_{\tau(i,j), v_i}(r).$$
For example, when $N=3$, 
$$
\bm{B}^{\bm{v}}(r):=
\begin{pmatrix}
a_{1,v_1}(r) & a_{2,v_1}(r) & a_{4,v_1}(r) \\
a_{2,v_2}(r) & a_{3,v_2}(r) & a_{5,v_2}(r) \\
a_{4,v_3}(r) & a_{5,v_3}(r) & a_{6,v_3}(r) 
\end{pmatrix}.
$$
In particular, taking $v_1=\cdots=v_N=k$ for some $1\le k \le L$, we have $$\bm{B}^{\bm{v}}(r)=\Matri_N\left(\bm{A}(r)^{(k)}\right).$$
In general, for any $r>0$, $\bm{B}^{\bm{v}}(r)$ can be written in the form:
\begin{equation}\label{rowexpression}
\bm{B}^{\bm{v}}(r)=
\begin{pmatrix}
\Matri_N(\bm{A}(r)^{(v_1)})_{(1)}\\\vdots\\ \Matri_N(\bm{A}(r)^{(v_N)})_{(N)}
\end{pmatrix}.    
\end{equation}

Recall that $\Pi_N$ is the set of permutations on $\{1,\dots,N\}$. For any $\sigma\in\Pi_N$, define map $\hat{\sigma}:\{1,\dots,L\}^N\to\{1,\dots,L\}^N$, where for any $\bm{v}\in\{1,\dots,L\}^N$,
$$\hat{\sigma}(\bm{v})=(v_{\sigma(1)},\dots,v_{\sigma(N)}).$$
With slight abuse of notation, we still write $\hat{\sigma}$ as $\sigma$, but one can easily distinguish them by the object it works on. 

Now we can explain why $\bm{B}^{\bm{v}}(r)$ is important in the asymptotic symmetry between $D_{u,+}(r)$ and $D_{u,-}(r)$.
Note that for any $\bm{y}\in\mathbb{R}^{L}$ and integers $1\le i,j\le N$,
$$\Matri_N(\bm{A}(r)\bm{y})[i,j]=\sum_{k=1}^La_{\tau(i,j),k}(r)y_k.$$
Then by the definition of determinant, we have
\begin{equation}\label{whyB}
\begin{aligned}
\det\left(\Matri_N(\bm{A}(r)\bm{y})\right)
&=\sum_{\sigma\in\Pi_N}(-1)^{\sign(\sigma)}\prod_{1\le n\le N}\left(\sum_{k=1}^La_{\tau(n,\sigma(n)),k}(r)y_k\right)\\
&=\sum_{\bm{v}\in\{1,\dots,L\}^N}\sum_{\sigma\in\Pi_N}(-1)^{\sign(\sigma)}a_{\tau(1,\sigma(1)),v_1}(r)\cdots a_{\tau(N,\sigma(N)),v_N}(r)y_{v_1}\cdots y_{v_N}\\
&=\sum_{\bm{v}\in\{1,\dots,L\}^N}\sum_{\sigma\in\Pi_N}(-1)^{\sign(\sigma)}b^{\bm{v}}_{1,\sigma(1)}(r)\cdots b^{\bm{v}}_{N,\sigma(N)}(r)y_{v_1}\cdots y_{v_N}\\
&=\sum_{\bm{v}\in\{1,\dots,L\}^N}\det(\bm{B}^{\bm{v}}(r)) y_{v_1}\cdots y_{v_N}\\
&=\frac{1}{N!}\sum_{\bm{v}\in\{1,\dots,L\}^N}\left(\sum_{\sigma\in\Pi_N}\det(\bm{B}^{\sigma(\bm{v})}(r))\right) y_{v_1}\cdots y_{v_N}.
\end{aligned}   
\end{equation}
This implies the asymptotic symmetry between $D_{u,+}(r)$ and $D_{u,-}(r)$ is determined by the asymptotic behavior of the coefficients, $\sum_{\sigma\in\Pi_N}\det(\bm{B}^{\sigma(\bm{v})}(r))$, of $y_{v_1}\cdots y_{v_N}$, $\bm{v}\in \{1,\dots,L\}^N$ as $r\to 0$. To study asymptotic properties of 
$\sum_{\sigma\in\Pi_N}\det(\bm{B}^{\sigma(\bm{v})}(r))$, the perturbation condition in Definition \ref{Def:PC} will be intensively used, and
we also need the following notations:
$$
\begin{aligned}
V_N
&:=\{(v_1,\dots,v_N)\in \{1,\dots,L\}^N:L-N\le v_1\le L,\\
&~~~~1\le v_i\le L-N-1\text{ for any } 2\le i\le N\}
\end{aligned}
$$
and
$$\widetilde{V}_N:=\{\bm{v}\in \{1,\dots,L\}^N:\sigma(\bm{v})\in V_N\text{ for some }\sigma\in\Pi_N\}.$$

\begin{lemma}\label{Lem:max1}
Let $X$ be qualified under perturbation. Then
$$\max_{\bm{v}\in \{1,\dots,L\}^N\setminus\widetilde{V}_N}\left|\sum_{\sigma\in\Pi_N}\det(\bm{B}^{\sigma(\bm{v})}(r))\right|=o(r)\text{ as $r\to0$.}$$
\end{lemma}
\begin{proof}
For any $\bm{v}=(v_1,\dots,v_N)^T\in \{1,\dots,L\}^N\setminus\widetilde{V}_N$, there are only two possible situations:
\begin{enumerate}[label={(\arabic*)}]
    \item $1\le v_i\le L-N-1$ for $i=1,\dots,N$;
    \item there exist distinct integers $1\le i,j\le N$ such that $L-N\le v_i,v_j\le L$.
\end{enumerate}
It suffices to show that as $r\to 0$,
$$\det(\bm{B}^{\bm{v}}(r))=o(r)$$
for any $\bm{v}$ in these two situations.

For Situation (1),
note that by (ii) of Lemma \ref{Lem:EigenSigma}, 
$$\Rank(\bm{\Sigma}_{0})=L-N-1.$$
By (i) and (iii) of Lemma \ref{Lem:H}, for any $1\le i\le N$,
$$\Matri\left(\bm{A}(r)^{(v_i)}\right)_{(i)}\bm{u}_0=\left(H(\bm{u}_0)\bm{A}(r)^{(v_i)}\right)_{(i)}=\bm{0}_{N}^T+o(r)\text{ as $r\to 0$}.$$
Then by (\ref{rowexpression}), 
\begin{equation}\notag
\bm{B}^{\bm{v}}(r)^{(N)}=\bm{B}^{\bm{v}}(r)\bm{u}_0=\bm{0}_{N}+o(r)\text{ as $r\to 0$.}
\end{equation}
This implies
$\det(\bm{B}^{\bm{v}}(r))=o(r)$ as $r\to 0$.

As for Situation (2), by (ii) of Lemma \ref{Lem:EigValDecSpd}, (\ref{rowexpression}) and that $\bm{A}(r)=\bm{P}(r)\bm{\Lambda}^{1/2}(r)$, there are two rows of $\bm{B}^{\bm{v}}(r)$ converging to $\bm{0}_N$ with the speed $O(r)$ as $r\to 0$. This also implies $\det(\bm{B}^{\bm{v}}(r))=o(r)$ as $r\to 0$.
\end{proof}

\begin{lemma}\label{Lem:max2}
Let $X$ be qualified under perturbation.
Then
\begin{equation}\label{interest}
\max_{\bm{v}\in \widetilde{V}_N}\left|\sum_{\sigma\in\Pi_N}\det\left(\bm{B}^{\sigma(\bm{v})}(r)\right)\right|=\Theta(r)\text{ as }r\to 0.
\end{equation}
\end{lemma}

\begin{proof}
The basic idea of this proof is to select a suitable $\bm{v}\in\widetilde{V}_N$ such that 
$$\sum_{\sigma\in\Pi_N}\det(\bm{B}^{\sigma(\bm{v})}(r))=\Theta(r)\text{ as }r\to 0.$$
For any $1\le l\le L$, define
$$\bm{M}^{l}(0):=\Matri_N\left(\bm{A}_0^{(l)}\right),~~ \bm{M}^{l}(r):=\Matri_N\left(\bm{A}(r)^{(l)}\right),$$
$$\widetilde{\bm{M}}^{l}(0):=\Matri_N\left(\bm{P}_0^{(l)}\right), \text{  and }\widetilde{\bm{M}}^{l}(r):=\Matri_N\left(\bm{P}(r)^{(l)}\right).$$
Indeed, $\bm{M}^{l}(0)$ and $\widetilde{\bm{M}}^{l}(0)$ are the left limits of  $\bm{M}^{l}(r)$ and $\widetilde{\bm{M}}^{l}(r)$ as $r\downarrow 0$, respectively. 
According to Remark \ref{Rem:scale}, since the problem is independent of the rescaling of the covariance function, it is safe to use (vii) of Lemma \ref{Lem:EigenSigma}.
Let $1\le v^*\le L-N-1$ correspond to an eigenvalue in (vii) of Lemma \ref{Lem:EigenSigma}, then we have
$$\bm{P}_0[i+j(j-1)/2,v^*]=\delta_{i,j}x_{v^*} \text{ for any }1\le i\le j\le N-1,$$
$$\bm{P}_0[i+N(N-1)/2,v^*]=0\text{ for any }1\le i\le N,\text{  and  }\bm{P}_0[L-1,v^*]=\bm{P}_0[L,v^*]=y_{v^*},$$
where $x_{v^*}$ and $y_{v^*}$ are both non-zero. 
Hence 
\begin{equation}\label{Mstar}
\bm{M}^{v^*}(0)=\Matri_N\left(\bm{A}_0^{(v^*)}\right)=
\diag\left(x_{v^*}\lambda^{1/2}_{v^*,0},\dots,x_{v^*}\lambda^{1/2}_{v^*,0},0\right).   
\end{equation}

Our next step is to calculate  $\sum_{\sigma\in\Pi_N}\det(\bm{B}^{\sigma(\bm{v})}(r))$ for a general $\bm{v}\in\widetilde{V}_N$ using the above notations. 
By (\ref{rowexpression}),
for any $\bm{v}\in\{1,\dots,L\}^N$,
$$
\begin{aligned}
\sum_{\sigma\in\Pi_N}\det\left(\bm{B}^{\sigma(\bm{v})}(r)\right)
&=\sum_{\sigma\in\Pi_N}\det\begin{pmatrix}
\bm{M}^{v_{\sigma(1)}}(r)_{(1)}\\\vdots\\ \bm{M}^{v_{\sigma(N)}}(r)_{(N)}
\end{pmatrix}\\
&=\sum_{\sigma\in\Pi_N}(-1)^{\sign(\sigma^{-1})}\det\begin{pmatrix}
\bm{M}^{v_1}(r)_{(\sigma^{-1}(1))}\\\vdots\\ \bm{M}^{v_N}(r)_{(\sigma^{-1}(N))}
\end{pmatrix}\\
&=\sum_{\sigma\in\Pi_N}(-1)^{\sign(\sigma)}\det\begin{pmatrix}
\bm{M}^{v_1}(r)_{(\sigma(1))}\\\vdots\\ \bm{M}^{v_N}(r)_{(\sigma(N))}
\end{pmatrix}.
\end{aligned}
$$
Note that for any matrix $\bm{A}=(a_{i,j})\in\mathbb{R}^{N\times N}$,
$$\det(\bm{A}):=\sum_{\sigma\in\Pi_N}(-1)^{\sign(\sigma)}a_{1,\sigma(1)}\cdots a_{N,\sigma(N)}.$$
Similarly, we can define an operation $f$ on any symmetric matrices
$\bm{A}_1,\dots,\bm{A}_N\in\mathbb{R}^{N\times N}$ by
\begin{equation}\label{foperation}
    \begin{aligned}
    f(\bm{A}_1,\cdots,\bm{A}_N)
    :&=\sum_{\sigma\in\Pi_N}(-1)^{\sign(\sigma)}\det
    \begin{pmatrix}
    (\bm{A}_1)_{(\sigma(1))}\\
    \vdots\\
    (\bm{A}_N)_{(\sigma(N))}\\
    \end{pmatrix}\\
    &\equiv\sum_{\sigma\in\Pi_N}(-1)^{\sign(\sigma)}\det\left(
    (\bm{A}_1)^{(\sigma(1))},
    \cdots,
    (\bm{A}_N)^{(\sigma(N))}
    \right).
    \end{aligned}
\end{equation}
Then in our context, we have for any $r>0$,
$$\sum_{\sigma\in\Pi_N}\det\left(\bm{B}^{\sigma(\bm{v})}(r)\right)=f\left(\bm{M}^{v_{1}}(r),\cdots,\bm{M}^{v_{N}}(r)\right).$$
Note that by Remark \ref{Rem:PC2}, we have
\begin{equation}\notag
\bm{M}^{i}(r)=\left\{
\begin{aligned}
\bm{M}^{i}(0)+O(r), && \text{for any $1\le i\le L-N-1$};\\
\lambda_{i,2}^{1/2}r\widetilde{\bm{M}}^{i}(r)+o(r), && \text{otherwise}.
\end{aligned}\right.
\end{equation}
For any $\bm{v}\in \widetilde{V}_N$, there exists an integer $1\le i\le N$ such that $L-N\le v_i\le L$. Then
$$
\begin{aligned}
&~~~~\sum_{\sigma\in\Pi_N}\det\left(\bm{B}^{\sigma(\bm{v})}(r)\right)\\
&=\lambda_{v_i,2}^{1/2}r f\left(\bm{M}^{v_1}(r),\cdots,\bm{M}^{v_{i-1}}(r),\widetilde{\bm{M}}^{v_i}(r),\bm{M}^{v_{i+1}}(r),\cdots,\bm{M}^{v_{N}}(r)\right)+o(r)\\
&=(-1)^{N-i}\lambda_{v_i,2}^{1/2}r f\left(\bm{M}^{v_1}(r),\cdots,\widehat{\bm{M}^{v_i}(r)},\cdots,\bm{M}^{v_{N}}(r),\widetilde{\bm{M}}^{v_i}(r)\right)+o(r)\\
&=(-1)^{N-i}\lambda_{v_i,2}^{1/2}r f\left(\bm{M}^{v_1}(0),\cdots,\widehat{\bm{M}^{v_i}(0)},\cdots,\bm{M}^{v_{N}}(0),\widetilde{\bm{M}}^{v_i}(0)\right)+o(r),
\end{aligned}
$$
where the hat over a component of a vector means that component is discarded.

From the above calculation, we see that it suffices to show $$\lambda_{v_i,2}^{1/2}f\left(\bm{M}^{v_1}(0),\cdots,\widehat{\bm{M}^{v_i}(0)},\cdots,\bm{M}^{v_{N}}(0),\widetilde{\bm{M}}^{v_i}(0)\right)\neq 0$$
for some $\bm{v}\in \widetilde{V}_N$ with $L-N\le v_i\le L$.
To take advantage of (\ref{Mstar}) and by noting that $\lambda_{k,2}> 0$ for any $L-N\le k\le L-1$ and $\lambda_{L,2}=0$ (see Lemma \ref{Lem:PLJ}), we select $\bm{v}$ having the form: $\bm{v}=(v^*,\dots,v^*,v_0)$ for some $L-N\le v_0\le L-1$. Then it suffices to show that there exists an integer $L-N\le v_0\le L-1$ such that
$$f\left(\bm{M}^{v^*}(0),\cdots,\bm{M}^{v^*}(0),\widetilde{\bm{M}}^{v_0}(0)\right)\neq 0.$$
To this end, we first observe that
$$
\begin{aligned}
&~~~~f\left(\bm{M}^{v^*}(0),\cdots,\bm{M}^{v^*}(0),\widetilde{\bm{M}}^{v_0}(0)\right)\\
&=\sum_{\sigma\in\Pi_N}(-1)^{\sign(\sigma)}\det\left(\bm{M}^{v^*}(0)^{(\sigma(1))},\cdots,\bm{M}^{v^*}(0)^{(\sigma(N-1))},\widetilde{\bm{M}}^{v_0}(0)^{(\sigma(N))}\right)\\
&=\sum_{\sigma\in\Pi_N}(-1)^{\sign(\sigma)}\widetilde{\bm{M}}^{v_0}(0)[N,\sigma(N)]\\
&~~~~\det\left(\left(\bm{M}^{v^*}(0)^{(\sigma(1))},\cdots,\bm{M}^{v^*}(0)^{(\sigma(N-1))}\right)[1:(N-1),1:(N-1)]\right)\\
&=(N-1)!\sum_{j=1}^N(-1)^{N-j}\widetilde{\bm{M}}^{v_0}(0)[N,j]\\
&~~~~\det
\left(\left(\bm{M}^{v^*}(0)^{(1)},\cdots,\widehat{\bm{M}^{v^*}(0)}^{(j)},\cdots,\bm{M}^{v^*}(0)^{(N)}\right)[1:(N-1),1:(N-1)]\right)\\
&=(N-1)!\det\begin{pmatrix}
\bm{M}^{v^*}(0)_{(1)} \\
\vdots\\
\bm{M}^{v^*}(0)_{(N-1)} \\
\widetilde{\bm{M}}^{v_0}(0)_{(N)} \\
\end{pmatrix},
\end{aligned}
$$
where the second equality follows from the fact that $\bm{M}^{v^*}(0)_{(N)}=\bm{0}_N$ (see (\ref{Mstar})),
and the last equality is given by Laplace's expansion.

Now suppose
$f(\bm{M}^{v^*}(0),\cdots,\bm{M}^{v^*}(0),\widetilde{\bm{M}}^{v_0}(0))=0$
for any integer $L-N\le v_0\le L-1$,
then 
$$\det\begin{pmatrix}
\bm{M}^{v^*}(0)_{(1)} \\
\vdots\\
\bm{M}^{v^*}(0)_{(N-1)} \\
\widetilde{\bm{M}}^{v_0}(0)_{(N)} \\
\end{pmatrix}=0,$$
which, together with (\ref{Mstar}), implies
$$\widetilde{\bm{M}}^{v_0}(0)[N,N]=\bm{u}_0^T\widetilde{\bm{M}}^{v_0}(0)\bm{u}_0= 0.$$
Then by (i) of Lemma \ref{Lem:H}, we have for any $L-N\le v_0\le L-1$,
$$\bm{u}_0^T\bm{H}(\bm{u}_0)\bm{P}_0^{(v_0)}=\bm{u}_0^T\widetilde{\bm{M}}^{v_0}(0)\bm{u}_0=0.$$
Note that $(\bm{u}_0^T\bm{H}(\bm{u}_0))^T=(\bm{H}(\bm{u}_0)_{(N)})^T$ is non-zero, and by (ii) of Lemma \ref{Lem:H}, it is in the zero space of $\bm{\Sigma}_{0}$, which is expanded by $\bm{P}_0^{(L-N)},\dots, \bm{P}_0^{(L)}$. Thus, the above equality implies $(\bm{H}(\bm{u}_0)_{(N)})^T$ and $\bm{P}_0^{(L)}$ must be linearly dependent. Then
by Lemma \ref{Lem:PLJ}, we get $\bm{H}(\bm{u}_0)_{(N)}$ and $(0,\dots,0,1,-1)\in\mathbb{R}^{L}$ are linearly dependent,
which contradicts  (\ref{Hdefn}). This completes the proof.
\end{proof}

Let $X$ be qualified under perturbation. Observe that by (\ref{whyB}), Lemmas \ref{Lem:max1} and \ref{Lem:max2},
for any $\bm{y}=(y_1,\dots,y_L)^T\in\mathbb{R}^{L}$,
\begin{equation}\label{eq:LimPoly}
    \lim_{r\to 0}r^{-1}\det\left(\Matri_N(\bm{A}(r)\bm{y})\right)=\sum_{i=L-N}^L y_{i}K_{i}(y_1,\dots,y_{L-N-1}),
\end{equation}
where each of $K_{i}(y_1,\dots,y_{L-N-1})$, $i=L-N,\dots,L$ is either a homogeneous polynomial of $y_1,\dots,y_{L-N-1}$ with degree $N-1$ or a zero function, and at least one of $K_i$, $L-N\le i\le L$ is not a zero function.
Thus, $\sum_{i=L-N}^L y_{i}K_{i}(y_1,\dots,y_{L-N-1})$ is a homogeneous polynomial of $y_1,\dots,y_{L}$ with degree $N$. For any $r>0$ and $\bm{y}\in\mathbb{R}^L$, define
\begin{equation}\label{def:hr}
    h_r(\bm{y}):=r^{-1}\det\left(\Matri_N(\bm{A}(r)\bm{y})\right)
\end{equation}
and
\begin{equation}\label{def:h0}
h_0(\bm{y}):= \sum_{i=L-N}^L y_{i}K_{i}(y_1,\dots,y_{L-N-1}).
\end{equation}
Then $h_r(\bm{y})$, $r\ge 0$ are all polynomials of $y_1,\dots,y_{L}$, and (\ref{eq:LimPoly}) can be rewritten as
\begin{equation}\label{eq:LimPolyh}
    h_0(\bm{y})=\lim_{r\to 0}h_r(\bm{y}).
\end{equation}
Since $\bm{A}(r)$ is continuous on $r\in[0,\delta_{pc}]$ (Definition \ref{Def:PC}), we have for any $\bm{y}\in\mathbb{R}^L$, $h_r(\bm{y})$ is a continuous function of $r\in [0,\delta_{pc}]$. Moreover, since $h_r$ are polynomials, the convergence in (\ref{eq:LimPolyh}) is uniform on any compact set of $\bm{y}$.

For $u\in\mathbb{R}$, $r\ge 0$ and $k=L-1,L$, define
\begin{equation}\label{def:Hr}
    H_{k,u}(r):=\left\{\bm{y}\in\mathbb{R}^L:\bm{A}(r)_{(k)}\bm{y}>u\right\}.
\end{equation}
For any $r\ge 0$, we also define
\begin{equation}\notag
    G_{+}(r):=\left\{\bm{y}\in\mathbb{R}^L:h_r(\bm{y})>0\right\}\text{ and }
    G_{-}(r):=\left\{\bm{y}\in\mathbb{R}^L:h_r(\bm{y})<0\right\}.
\end{equation}
With these notations, the domains $D_{u,\pm}(r)$ mentioned at the beginning of this section can be rewritten as
\begin{equation}\label{def:Dur}
    D_{u,\pm}(r):=H_{L-1,u}(r)\cap H_{L,u}(r)\cap G_{\pm}(r).    
\end{equation}
In particular, by $\bm{A}(0)=\bm{P}(0)\bm{\Lambda}^{1/2}(0)$, (ii) and (v) of Lemma \ref{Lem:EigenSigma},
\begin{equation}\label{eq:last2rows}
    \bm{A}_{(L-1)}(0)=\bm{A}_{(L)}(0).
\end{equation}
Then we have
\begin{equation}\notag
    H_{L-1,u}(0)=H_{L,u}(0),
\end{equation}
which implies
\begin{equation}\label{eq:Du0}
    D_{u,\pm}(0)=H_{L,u}(0)\cap G_{\pm}(0).    
\end{equation}

The following lemma reveals the key idea of this section:
the asymptotic symmetry between $D_{u,\pm}(r)$ as $r\to 0$. 

\begin{lemma}\label{Lem:AS1}
Let $X$ be qualified under perturbation.
Then for any $u\in\mathbb{R}$, we have
\begin{enumerate}[label=(\roman*)]
    \item if $(y_1,\dots,y_L)\in D_{u,\pm}(0)$, then $(y_1,\dots,y_{L-N-1},-y_{L-N},\dots,-y_L)\in D_{u,\mp}(0)$;
    \item $D_{u,\pm}(0)$ are both non-empty open sets; 
    \item for any $r_0\in[0,\delta_{pc}]$, we have
    \begin{equation}\notag
        \lim_{r\to r_0}I_{D_{u,\pm}(r)}(\bm{y})=I_{D_{u,\pm}(r_0)}(\bm{y})~~\text{for almost all $\bm{y}\in\mathbb{R}^L$},
    \end{equation}
    where $I$ stands for the indicator function of a set.
    
\end{enumerate}

\end{lemma}

\begin{proof}
For (i), let $\bm{y}=(y_1,\dots,y_L)\in D_{u,+}(0)$ and $\bm{y}'=(y_1,\dots,y_{L-N-1},-y_{L-N},\dots,-y_L)$. Then by (\ref{eq:Du0}),
$$\bm{A}_{(L)}(0)\bm{y}>u\text{ and }h_0(\bm{y})>0.$$
Note that by $\bm{A}(0)=\bm{P}(0)\bm{\Lambda}^{1/2}(0)$ and (ii) of Lemma \ref{Lem:EigenSigma}, $\bm{A}(0)^{(i)}=0$ for any $L-N\le i\le L$,
which implies $\bm{A}_{(L)}(0)\bm{y}'>u$.  
In addition, by (\ref{def:h0}), 
$$h_0(\bm{y}')=\sum_{L-N}^L(-y_i)K(y_{1},\dots,y_{L-N-1})=-h_0(\bm{y}).$$
Thus, $\bm{y}'\in D_{u,-}(0)$. The proof of the other part is similar.

As for (ii),
it is easy to see $D_{u,\pm}(0)$ are both open.  By (i), $D_{u,\pm}(0)$ are either both empty or both non-empty.
Since 
$$D_{u,+}(0)\cup D_{u,-}(0)=H_{L,u}(0)\setminus\{\bm{y}\in\mathbb{R}^L:h_0(\bm{y})= 0\},$$
their union is non-empty. Thus, $D_{u,\pm}(0)$ are both non-empty open sets, as required.

As for (iii), fix $r_0\in[0,\delta_{pc}]$.
By the continuity of $\bm{A}(r)$ on $r\in[0,\delta_{pc}]$,
it is easy to see 
\begin{equation}\notag
    H_{i,u}(r_0)\subset\liminf_{r\to r_0}H_{i,u}(r)\subset\limsup_{r\to r_0}H_{i,u}(r)\subset H_{i,u}(r_0)\cup \partial H_{i,u}(r_0)
\end{equation}
for $i=L-1,L$, and
\begin{equation}\notag
    G_{\pm}(r_0)\subset\liminf_{r\to r_0}G_{\pm}(r)\subset\limsup_{r\to r_0}G_{\pm}(r)\subset G_{\pm}(r_0)\cup \partial G_{\pm}(r_0),
\end{equation}
where $\partial$ stands for the boundary of a set, and the $\liminf$ and $\limsup$ of the sets are defined by $$\liminf_{r\to r_0}A_r:=\bigcup_{\delta>0}\bigcap_{|r-r_0|<\delta} A_r\text{ and }\limsup_{r\to r_0}A_r:=\bigcap_{\delta>0}\bigcup_{|r-r_0|<\delta} A_r.$$

This implies
\begin{equation}\label{eq:liminfDr}
    D_{u,\pm}(r_0)\subset \liminf_{r\to r_0}D_{u,\pm}(r)\subset \limsup_{r\to r_0}D_{u,\pm}(r)\subset D_{u,\pm}(r_0)\cup Q_0,
\end{equation}
where $$Q_0:=\partial H_{L-1,u}(r_0)\cup\partial H_{L,u}(r_0)\cup \partial G_{\pm}(r_0),$$
and it is easy to see $\lambda_L(Q_0)=0$.
Note that for any $\bm{y}\in (\liminf_{r\to r_0} D_{u,\pm}(r))\cap D_{u,\pm}(r_0)$,  $$\lim_{r\to r_0}I_{D_{u,\pm}(r)}(\bm{y})=I_{D_{u,\pm}(r_0)}(\bm{y})=1,$$
while for any $\bm{y}\in (\liminf_{r\to r_0} D_{u,\pm}^c(r))\cap D_{u,\pm}^c(r_0)=(\limsup_{r\to r_0} D_{u,\pm}(r))^c\cap D_{u,\pm}^c(r_0)$, 
$$\lim_{r\to r_0}I_{D_{u,\pm}(r)}(\bm{y})=I_{D_{u,\pm}(r_0)}(\bm{y})=0,$$
where by (\ref{eq:liminfDr}),
$$\left(\liminf_{r\to r_0} D_{u,\pm}(r)\right)\cap D_{u,\pm}(r_0)=D_{u,\pm}(r_0)$$
and
$$\left(\limsup_{r\to r_0} D_{u,\pm}(r)\right)^c\cap D_{u,\pm}^c(r_0)=\left(\limsup_{r\to r_0} D_{u,\pm}(r)\right)^c\supset D_{u,\pm}^c(r_0)\cap Q_0^c.$$
Combining all of the above, we have for any $\bm{y}\in Q_0^c$, $\lim_{r\to r_0}I_{D_{u,\pm}(r)}(\bm{y})=I_{D_{u,\pm}(r_0)}(\bm{y})$.
This completes the proof of the lemma.

\iffalse
Then
\begin{equation}\notag
    \begin{aligned}
    &~~~~\limsup_{n\to \infty}\left(D_{u,\pm}(r_n)\triangle D_{u,\pm}(r_0)\right)\\
    &=\limsup_{n\to \infty}\left(D_{u,\pm}(r_n)\setminus D_{u,\pm}(r_0)\right)\cup\limsup_{n\to \infty}\left(D_{u,\pm}(r_0)\setminus D_{u,\pm}(r_n)\right)\\
    &\subset \left(\left(\limsup_{n\to \infty}D_{u,\pm}(r_n)\right)\setminus D_{u,\pm}(r_0)\right)\cup \left(D_{u,\pm}(r_0)\setminus\left(\liminf_{n\to \infty}D_{u,\pm}(r_n)\right)\right)\\
    &\subset \partial H_{L-1,u}(r_0)\cup\partial H_{L,u}(r_0)\cup \partial G_{\pm}(r_0).
    \end{aligned}
\end{equation}
Then 
$$\lambda_L\left(\limsup_{n\to \infty}\left(D_{u,\pm}(r_n)\triangle D_{u,\pm}(r_0)\right)\right)=\lambda_L\left(\partial H_{L-1,u}(r_0)\cup\partial H_{L,u}(r_0)\cup \partial G_{\pm}(r_0)\right)=0.$$
This completes the proof of the lemma.
\fi

\end{proof}

\begin{defn}
For any $n\ge 1$, a function $g:\mathbb{R}^{n}\to\mathbb{R}$ is said to be \textbf{regular} if it can be written in the form:
$$g(\bm{y})=|\alpha(\bm{y})|\exp(-\beta(\bm{y})),$$
where $\alpha(\bm{y})$ is a non-zero polynomial of $y_1,\dots,y_n$ and $\beta(\bm{y})=\bm{y}^T\bm{\Sigma}\bm{y}$ for some positive-definite matrix $\bm{\Sigma}\in\mathbb{R}^{n\times n}$.
\end{defn}
\begin{remark}\label{Rem:regular}
It is easy to check that a regular function on $\mathbb{R}^n$ must be non-negative, bounded and integrable. Moreover, the integral of a regular function on a non-empty open set must be positive.
\end{remark}

Let $X$ be qualified under perturbation. 
For any $r\ge 0$ and $\bm{y}\in\mathbb{R}^{L}$,
define 
\begin{equation}\label{def:gr}
g_r(\bm{y}):=\left|h_r(\bm{y})\right|p_L(\bm{y}),    
\end{equation}
where $h_r(\bm{y})$ is as defined in (\ref{eq:LimPoly})-(\ref{eq:LimPolyh}) and
\begin{equation}\label{def:pL}
    p_L(\bm{y}):=\frac{1}{(2\pi)^{L/2}}\exp\left(-\frac{1}{2}\bm{y}^T\bm y\right).    
\end{equation}
For any $r\ge 0$, since $h_r(\bm{y})$ is a nonzero polynomial of $y_1,\dots,y_L$, $g_r(\bm{y})$ is a regular function of $\bm{y}\in\mathbb{R}^L$. In addition, since $\bm{A}(r)$ is continuous on $r\in[0,\delta_{pc}]$ (where $\delta_{pc}$ is as defined in Definition \ref{Def:PC}), all the coefficients of the polynomial $h_r$ are continuous, and then uniformly bounded on $r\in[0,\delta_{pc}]$.
Then by the triangle inequality, there exists a constant $C>0$ such that for any $r\in[0,\delta_{pc}]$,
\begin{equation}\label{Rel:constdom}
    h_r(\bm{y})\le C\sum_{v_1,\dots,v_N\in\{1,\dots,L\}}\left|y_{v_1}\cdots y_{v_N}\right|.
\end{equation}
Thus, for any $r\in[0,\delta_{pc}]$ and $\bm{y}\in\mathbb{R}^L$,
\begin{equation}\label{Rel:dominated}
    g_r(\bm{y})\le C\sum_{v_1,\dots,v_N\in\{1,\dots,L\}}\left|y_{v_1}\cdots y_{v_N}\right|p_L(\bm{y}),   
\end{equation}
where the right-hand side is a finite sum of regular functions.

%===========================================================================================================
%===========================================================================================================

\subsection{Proof of Theorem \ref{Theo:MR1}}
\begin{proof} Fix $u\in\mathbb{R}$.
Since $X$ is isotropic, we only need to show
$$\lim_{r\to 0}\frac{f_{u,+}(\bm{u}_0r)}{f_{u,-}(\bm{u}_0r)}=1.$$
For any $r>0$, since $\bm{\Sigma(r)}$ is positive-definite and $\bm{\Sigma(r)}=\bm{A}(r)\bm{A}^T(r)$, $\bm{A}(r)$ is invertible.
By the change of variable $\bm{y}=\bm{A}^{-1}(r)(\bm{x}'',x,z)^T$, (\ref{eq:pt}) and (\ref{def:pL}), we have
\begin{equation}\notag
\begin{aligned}
p_{\bm{t}}(\bm{x}'',x,z|\bm{0},\bm{0})
&=\frac{1}{\sqrt{(2\pi)^{L}\det(\bm{\Sigma}(\bm{t}))}}\exp\left(-\frac{1}{2}(\bm{x}'',x,z)\bm{\Sigma}(\bm{t})^{-1}(\bm{x}'',x,z)^T\right)\\
&=\frac{1}{\sqrt{(2\pi)^{L}\det(\bm{\Sigma}(\bm{t}))}}\exp\left(-\frac{1}{2}\bm{y}^T\bm{y}\right)\\
&=\frac{1}{\sqrt{\det(\bm{\Sigma}(\bm{t}))}}p_L(\bm{y}).
\end{aligned}
\end{equation} 
Then the ratio becomes
\begin{equation}\label{newintegrand}
\begin{aligned}
\frac{f_{u,+}(\bm{u}_0r)}{f_{u,-}(\bm{u}_0r)}
&=\frac{\int_{x,z>u}\int_{\det\left(\Matri_N(\bm{x}'')\right)>0}\left|\det\left(\Matri_N(\bm{x}'')\right)\right|p_{\bm{t}}(\bm{x}'',x,z|\bm{0},\bm{0})d\bm{x}''dxdz}{\int_{x,z>u}\int_{\det\left(\Matri_N(\bm{x}'')\right)<0}\left|\det\left(\Matri_N(\bm{x}'')\right)\right|p_{\bm{t}}(\bm{x}'',x,z|\bm{0},\bm{0})d\bm{x}''dxdz}\\
&=\frac
{\int_{D_{u,+}(r)}\left|\det\left(\Matri_N(\bm{A}(r)\bm{y}) \right)\right|p_L(\bm{y})d\bm{y}}
{\int_{D_{u,-}(r)}\left|\det\left(\Matri_N(\bm{A}(r)\bm{y}) \right)\right|p_L(\bm{y})d\bm{y}}\\
&=\frac
{\int_{D_{u,+}(r)}\left|r^{-1}\det\left(\Matri_N(\bm{A}(r)\bm{y}) \right)\right|p_L(\bm{y})d\bm{y}}
{\int_{D_{u,-}(r)}\left|r^{-1}\det\left(\Matri_N(\bm{A}(r)\bm{y}) \right)\right|p_L(\bm{y})d\bm{y}}\\
&=\frac
{\int_{D_{u,+}(r)}|h_r(\bm{y})|p_L(\bm{y})d\bm{y}}
{\int_{D_{u,-}(r)}|h_r(\bm{y})|p_L(\bm{y})d\bm{y}}\\
&=\frac
{\int_{D_{u,+}(r)}g_r(\bm{y})d\bm{y}}
{\int_{D_{u,-}(r)}g_r(\bm{y})d\bm{y}}\\
&=\frac
{\int_{\mathbb{R}^{L}}g_r(\bm{y})I_{D_{u,+}(r)}(\bm{y})d\bm{y}}
{\int_{\mathbb{R}^{L}}g_r(\bm{y})I_{D_{u,-}(r)}(\bm{y})d\bm{y}},
\end{aligned}\end{equation}
where $D_{u,\pm}(r)$, $h_r$, $g_r$ and $p_L$ are as defined in (\ref{def:Dur}), (\ref{def:hr}), (\ref{def:gr}), and (\ref{def:pL}), respectively.

Note that by (iii) of Lemma \ref{Lem:AS1}, for any $r_0\in [0,\delta_{pc}]$,
$g_{r}(\bm{y})I_{D_{u,\pm}(r)}(\bm{y})$ converges almost everywhere on $\mathbb{R}^L$ to $g_{r_0}(\bm{y})I_{D_{u,\pm}(r_0)}(\bm{y})$ as $r\to r_0$, and $g_{r}(\bm{y})I_{D_{u,\pm}(r)}(\bm{y})$ is dominated by the right-hand side of (\ref{Rel:dominated}) which does not depend on $r$ and is integrable over $\mathbb{R}^L$. Then by the dominated convergence theorem, we have
$$\lim_{r\to r_0}\int_{\mathbb{R}^{L}}g_{r}(\bm{y})I_{D_{u,\pm}(r)}(\bm{y})d\bm{y}=\int_{\mathbb{R}^{L}}g_{r_0}(\bm{y})I_{D_{u,\pm}(r_0)}(\bm{y})d\bm{y}=\int_{D_{u,\pm}(r_0)}g_{r_0}(\bm{y})d\bm{y}.$$

Moreover,
by (i) and (ii) of Lemma \ref{Lem:AS1} and Remark \ref{Rem:regular}, we have
$$
\int_{D_{u,+}(0)}g_{0}(\bm{y})d\bm{y}=\int_{D_{u,-}(0)}g_{0}(\bm{y})d\bm{y}>0.$$
Then
$$
\lim_{r\to 0}\frac{f_{u,+}(\bm{u}_0r)}{f_{u,-}(\bm{u}_0r)}
=\frac{\lim_{r\to 0}\int_{D_{u,+}(r)}g_{r}(\bm{y})d\bm{y}}{\lim_{r\to 0}\int_{D_{u,-}(r)}g_{r}(\bm{y})d\bm{y}}\\
=\frac{\int_{D_{u,+}(0)}g_{0}(\bm{y})d\bm{y}}{\int_{D_{u,-}(0)}g_{0}(\bm{y})d\bm{y}}=1.
$$
\end{proof}

\begin{remark}
In fact, the result in Theorem \ref{Theo:MR1} also holds for $N=1$.
Following the proof of Lemma \ref{Lem:Sigma}, one can easily check that the expression of $\bm{\Sigma}_0$ and $\bm{\Sigma}_2$ in Lemma \ref{Lem:Sigma} also holds for $N=1$. Hence
$$\bm{\Sigma}_0=\begin{pmatrix}
0 & 0 & 0\\
0 & a & a\\
0 & a & a\\
\end{pmatrix},$$
where $a:=1-\frac{1}{3}\rho^{(1)}(0)^2\rho^{(2)}(0)^{-1}$, and by Proposition 3.3 in \cite{che}, we have $a>0$.
By solving the equation $\det(\bm{\Sigma}_{0}-\lambda\bm{I}_3)=0$ for any $\lambda\in\mathbb{R}$, we have
$$\bm{\Lambda}_{0}=\diag(2a,0,0),$$
which implies $\bm{A}_0^{(2)}=\bm{A}_0^{(3)}=\bm{0}_3$. Then by $\bm{\Sigma}_{0}=\bm{A}_0\bm{A}_0^T$, we have
$$\bm{A}_{0}=\begin{pmatrix}
0 & 0 & 0\\
\sqrt{a} & 0 & 0\\
\sqrt{a} & 0 & 0
\end{pmatrix}.$$
Thus,
$$\bm{A}(r)_{(1)}=\bm{A}_1r+o(r).$$
Moreover, based on Lemma \ref{Lem:Sigma}, one can verify that for $N=1$, all the results in Section \ref{SubSec:GCS} still hold. As $N=1$, $\Matri_N(\bm{A}(r)\bm{y})=\bm{A}(r)_{(1)}\bm{y}$. 
Then by Remark \ref{Rem:PC2},
\begin{equation}\notag
    \lim_{r\to 0}r^{-1}\det\left(\Matri_N(\bm{A}(r)\bm{y})\right)=\lim_{r\to 0}\frac{1}{r}\bm{A}(r)_{(1)}\bm{y}=(\bm{A}_{1})_{(1)}\bm{y}.
\end{equation}
Note that by Remark \ref{Rem:PC2},
\begin{equation}\notag
\bm{A}_{1}^{(i)}=\left\{
\begin{aligned}
    \lambda_{i,0}^{1/2}\bm{P}_{1}^{(i)}, &&  \text{ $i=1$;}\\
    \lambda_{i,2}^{1/2}\bm{P}_{0}^{(i)}, &&  \text{$i=2,3$}.
\end{aligned}\right.
\end{equation}
Thus,
\begin{equation}\label{Eq:rAyN1}
    \lim_{r\to 0}r^{-1}\det\left(\Matri_N(\bm{A}(r)\bm{y})\right)=\lambda_{1,0}^{1/2}\bm{P}_{1}[1,1]y_1+\lambda_{2,2}^{1/2}\bm{P}_{0}[1,2]y_2+\lambda_{3,2}^{1/2}\bm{P}_{0}[1,3]y_3.
\end{equation}
Note that the dimension of the eigenspace of $2a$, as an eigenvalue of $\bm{\Sigma}_0$, is one. Thus, an eigenvector of $2a$ must have the form $k\bm{P}_{0}^{(1)}$ for some $k\neq 0$. 
By (iv) of Lemma \ref{Lem:EigValDecSpd}, $\bm{P}_{1}^{(1)}$ is either $\bm{0}_3$ or an eigenvector of $2a$. Then by (i) of Lemma \ref{Lem:EigValDecSpd},
we have $$\bm{P}_{1}^{(1)}=\bm{0}_3.$$
Note that 0 is an eigenvalue of $\bm{\Lambda}_{0}$ with multiplicity $N+1$ for $N=1$. In addition, since $\bm{\Sigma}_{0}^{(1)}=\bm{0}_{3}$, $\bm{\Sigma}_{0}^{(2)}=\bm{\Sigma}_{0}^{(3)}$ and 
$\bm{\Sigma}_{2}^{(2)}=\bm{\Sigma}_{2}^{(3)}$,
we can also get 
$$\det(\bm{\Sigma}(r))=o\left(r^{2N+2}\right)$$
for $N=1$. Then 
we can follow the proof of Lemma \ref{Lem:speed} to show that it also holds for $N=1$. This implies $\lambda_{2,2}>0$ and $\lambda_{3,2}=0$.

Therefore, (\ref{Eq:rAyN1}) becomes
\begin{equation}\notag
    \lim_{r\to 0}r^{-1}\det\left(\Matri_N(\bm{A}(r)\bm{y})\right)=\lambda_{2,2}^{1/2}\bm{P}_{0}[1,2]y_2.
\end{equation}
Note that by Lemma \ref{Lem:Sigma},
$$\bm{\Sigma}_2=\begin{pmatrix}
18\alpha-30\beta &\alpha'-\frac{5}{3}\beta' &\alpha'-\frac{5}{3}\beta'\\
\alpha'-\frac{5}{3}\beta' &-b &-b\\
\alpha'-\frac{5}{3}\beta' &-b &-b\\
\end{pmatrix},$$
where $b:=\frac{1}{6}\rho^{(1)}(0)-\frac{5}{18}\rho^{(1)}(0)^2\rho^{(2)}(0)^{-2}\rho^{(3)}(0)$, and by (\ref{Rel:1830}), we have $b>0$. Since $\bm{\Sigma}_0\bm{P}_0^{(3)}=0$, we have $$\bm{P}_0[2,3]+\bm{P}_0[3,3]=0.$$
Then by (iii) of Lemma \ref{Lem:EigValDecSpd},
$$
\begin{aligned}
0=\lambda_{3,2}
&=\left(\bm{P}_0^{(3)}\right)^T\bm{\Sigma}_2\bm{P}_0^{(3)}\\
&=(18\alpha-30\beta)\bm{P}_0[1,3]^2+2\left(\alpha'-\frac{5}{3}\beta'\right)\bm{P}_0[1,3](\bm{P}_0[2,3]+\bm{P}_0[3,3])\\
&~~~~-b(\bm{P}_0[2,3]+\bm{P}_0[3,3])^2\\
&=(18\alpha-30\beta)\bm{P}_0[1,3]^2.
\end{aligned}
$$
Since $18\alpha-30\beta>0$, we have $\bm{P}_0[1,3]=0$ (and thus, Lemma \ref{Lem:PLJ} also  holds for $N=1$). By $(\bm{\Sigma}_0-2a\bm{I})\bm{P}_0^{(1)}=0$ and $a\neq 0$, we have $\bm{P}_0[1,1]=0$. Thus, if $\bm{P}_0[1,2]=0$, then from the above, we can get $(\bm{P}_0)_{(1)}=(0,0,0)$. This implies $\bm{P}_0$ is degenerate, resulting in a contradiction. Thus, we have $$\bm{P}_0[1,2]\neq 0.$$

Therefore, for $N=1$, we can define $$h_0(\bm{y}):=\lambda_{2,2}^{1/2}\bm{P}_{0}[1,2]y_2$$
which is a non-zero polynomial of $\bm{y}$ with degree one.
Then we can also define $D_{u,\pm}(r)$ as in (\ref{def:Dur}).
Since $(y_1,y_2,y_3)\in D_{u,\pm}(0)$ is equivalent to $(y_1,-y_2,-y_3)\in D_{u,\mp}(0)$, one can follow the proof of Lemma \ref{Lem:AS1} and show that it also holds for $N=1$.
The remaining proof is the same as that of Theorem \ref{Theo:MR1}.
\end{remark}

\section{Asymptotic Behavior as \texorpdfstring{$u\to \infty$}{the Threshold Tends to Infinity}}\label{proj2:sec:mr2}

Theorem \ref{Theo:MR1} states that if two critical points are very close, then the determinant of their Hessian should have opposite signs. However, note this result holds for any threshold $u$. Consequently, it is not surprising that Theorem \ref{Theo:MR1} cannot further specify what exact types do the two critical points belong to: is one of them a local maxima or a critical point with index $N-2$? Is the other a saddle point with index $N-1$ or a critical point with a lower index? In this section, we investigate these questions for the asymptotic case where $u\to\infty$. It is shown that in this case, only two types of critical points remain: local maxima and the saddle points with index $N-1$.

\subsection{Main Result 2}\label{Sec:M2}
Let $X$ be qualified under perturbation. Recall that Condition (3) in Definition \ref{Def:qualified} implies $\rho(x)$ is four times continuously differentiable on $[0,\delta_{\rho}^2]$, which makes $\bm{\Sigma}(r)$ continuous on $r\in[0,\delta_{\rho}]$ (see Lemma \ref{Lem:Sigma}).
Indeed, for any $\tilde{\delta}_{\rho}>0$ such that 
\begin{equation}\label{Con:GC1}
    \text{$\rho(x)$ is four times continuously differentiable on $\left[0,\tilde{\delta}_{\rho}^2\right]$},
\end{equation}
we can show that $\bm{\Sigma}(r)$ is continuous on $r\in[0,\tilde{\delta}_{\rho}]$ by a similar proof of Lemma \ref{Lem:Sigma}. 

Assume that (\ref{Con:GC1}) holds for some $\tilde{\delta}_{\rho}>0$.
For any $r\ge 0$, let $\widetilde{\bm{A}}(r)$ be the non-negative square root of $\bm{\Sigma}(r)$. Recall that this means $\widetilde{\bm{A}}(r)$ is the unique positive semi-definite matrix such that 
\begin{equation}\notag
    \bm{\Sigma}(r)=\widetilde{\bm{A}}(r)\widetilde{\bm{A}}^T(r).
\end{equation}
Since the non-negative square root is continuous (see, for example, pages 405 and 411 in \cite{hor}),  $\widetilde{\bm{A}}(r)$ is continuous on $r\in(0,\tilde{\delta}_{\rho}]$. 
In addition, for any $r\ge 0$, we can get
\begin{equation}\label{eq:TA0A0}
   \widetilde{\bm{A}}(r)=\bm{A}(r)\bm{P}^T(r).
\end{equation}
Then by Definition \ref{Def:PC}, 
$$\lim_{r\to 0}\widetilde{\bm{A}}(r)=\widetilde{\bm{A}}(0).$$
Thus, $\widetilde{\bm{A}}(r)$ is continuous on $r\in[0,\tilde{\delta}_{\rho}]$. 

To establish our second main result, we also need
$$\bm{\Sigma}(r)[k+k(k-1)/2,i]<0$$ 
for any $1\le k\le N-1$, $r\ge 0$ and $i=L-1,L$.
Note that by Lemma \ref{Lem:Sigma}, $\bm{\Sigma}(0)[k+k(k-1)/2,i]=\frac{4}{3}\rho^{(1)}(0)<0$ for any $1\le k\le N-1$ and $i=L-1,L$. Thus, it is equivalent to requiring that the above inequality to hold for any $r>0$.
By the proof of Lemma \ref{Lem:Sigma} (see (\ref{k1lessthan1}), (\ref{kstarlessthan1}), (\ref{SPL1}) and (\ref{SPL})), this is equivalent to 
$$k_1(\bm{u}_0r)>0\text{ and }1-k_1(\bm{u}_0r)k_*(\bm{u}_0r)-k_2^2(\bm{u}_0r)r^4>0$$
for any $r\in (0,\tilde{\delta}_{\rho}]$,
where for any $\bm{t}\in\mathbb{R}^N$,
$$k_1(\bm{t})=\frac{\rho^{(1)}(\|\bm{t}\|^2)} {\rho^{(1)}(0)},~~k_2(\bm{t})=\frac{2\rho^{(2)}(\|\bm{t}\|^2)}{\rho^{(1)}(0)},\text{ and }k_*(\bm{t})=k_1(\bm{t})+k_2(\bm{t})\|\bm{t}\|^2,$$
i.e., for any $x\in (0,\tilde{\delta}_{\rho}^2]$, 
\begin{equation}\label{Con:GC2}
\rho^{(1)}(x)<0\text{ and }(\rho^{(1)}(x))^2+2\rho^{(1)}(x)\rho^{(2)}(x)x+4(\rho^{(2)}(x))^2x^2<(\rho^{(1)}(0))^2.
\end{equation}

\begin{example}
For any $a>0$, we can show that condition (\ref{Con:GC2}) is satisfied when $\rho(x)=e^{-ax}$, $x\ge 0$.
Since condition (\ref{Con:GC2}) is invariant under rescaling, it is equivalent to check this condition when $a=1$.
Note that $\rho^{(1)}(x)=-e^{-x}<0$ and $\rho^{(2)}(x)=e^{-x}$, $x\ge 0$. Thus, it suffices to check for any $x>0$,
$$e^{2x}>1-2x+4x^2,$$
which is obvious since $e^{t}>1-t+t^2$ for any $t>0$.
\end{example}

Recall that by (\ref{expression}),
$$
f_{u,k}(\bm{t}):=P(X(\bm{0})>u)^{-1}p(\bm{0})^{-1}\int_{x,z>u}\int_{D_k}|\det(\bm{x}'')|p_{\bm{t}}(\bm{x}'',x,z|\bm{0},\bm{0})p_{\bm{t}}(\bm{0},\bm{0})d\bm{x}''dxdz,
$$
where 
$$
\begin{aligned}
D_k:=\Big\{\bm{x}\in\mathbb{R}^{N(N+1)/2}:&\Matri_N(\bm{x})\text{ is non-degenerate} \\ &\text{and has exactly $k$ negative eigenvalues}\Big\}.
\end{aligned}
$$
For any $0\le k\le N$, $f_{u,k}(\bm{t})$ is positive and  continuous (by the dominated convergence theorem) on $\bm{t}\in\mathbb{R}^N\setminus\{\bm{0}_N\}$.

For any $u>0$ and $\bm{t}\in\mathbb{R}^N\setminus{\{\bm{0}_N\}}$, define
$$\Psi_u(\bm{t}):=\frac{\sum_{k=0}^{N-2} f_{u,k}(\bm{t})}{f_{u,N-1}(\bm{t})+f_{u,N}(\bm{t})}.$$
Then $\Psi_u(\bm{t})$ is also continuous on $\bm{t}\in\mathbb{R}^N\setminus\{\bm{0}_N\}$.
The following theorem describes the limiting behavior of $\Psi_u(\bm{t})$ as $\|\bm{t}\|\to 0$ or $u\to \infty$. In particular, it shows that as $u$ tends to infinity, the other types of critical points become negligible compared to the critical points with indices $N$ or $N-1$ uniformly in $\bm{t}$ in a neighborhood of $\bm{0}$.

\begin{theorem}\label{Theo:MR2}
Let $X$ be qualified under perturbation. Assume that $X$ also satisfies (\ref{Con:GC1}) and (\ref{Con:GC2}) for some $\tilde{\delta}_{\rho}>0$. 
Then we have
\begin{enumerate}[label=(\roman*)]
    \item For any $u>0$, the limit $\Psi_u(\bm{0}):=\lim_{\|\bm{t}\|\to 0}\Psi_u(\bm{t})$ exists (and thus, $\Psi_u(\bm{t})$ is well-defined and continuous on $\bm{t}\in\mathbb{R}^N$).
    \item As $u\to\infty$,
    $$\max_{\bm{t}\in \overline{ B(\bm{0}_N,\tilde{\delta}_{\rho})}}\Psi_u(\bm{t})\to 0.$$
    
\end{enumerate}

\end{theorem}

%===========================================================================================================
%===========================================================================================================

\subsection{Preparatory results for the proof of Theorem \ref{Theo:MR2}}\label{Sec:PMR2}

By replacing $\bm{A}(r)$ in (\ref{def:Hr}) with $\widetilde{\bm{A}}(r)$, we can similarly define 
\begin{equation}\label{def:Hr_tilde}
    \widetilde{H}_{i,u}(r):=\left\{\bm{y}\in\mathbb{R}^L:\widetilde{\bm{A}}(r)_{(i)}\bm{y}>u\right\},
\end{equation}
for any $r\ge 0$, $u>0$ and $i=L-1,L$.
It can be noted that 
\begin{itemize}
    \item by (\ref{eq:last2rows}) and (\ref{eq:TA0A0}), % and (v) of Lemma \ref{Lem:EigenSigma},
    $\widetilde{H}_{L-1,u}(0)$ and $\widetilde{H}_{L,u}(0)$ coincide;
    \item by the continuity of $\bm{A}(r)$ on $r\in[0,\tilde{\delta}_{\rho}]$, for $i=L-1,L$,
    \begin{equation}\label{Rel:Hur_tilde}
    \widetilde{H}_{i,u}(0)\subset\liminf_{r\to 0}\widetilde{H}_{i,u}(r)\subset\limsup_{r\to 0}\widetilde{H}_{i,u}(r)\subset \widetilde{H}_{i,u}(0)\cup \partial \widetilde{H}_{i,u}(0),
    \end{equation}
    \item for any $r>0$, since $\bm{\Sigma}(r)$ is non-degenerate, the two $(L-1)$-dimensional hyper-planes $\partial\widetilde{H}_{L-1,u}(r)$ and $\partial\widetilde{H}_{L,u}(r)$ cannot be parallel, and thus, $\partial\widetilde{H}_{L-1,u}(r)\cap \partial\widetilde{H}_{L,u}(r)\neq\emptyset$.
\end{itemize}
For any $r\ge 0$ and $u>0$, let $V_u(r):=\widetilde{H}_{L-1,u}(r)\cap \widetilde{H}_{L,u}(r)$.
Since $\overline{V}_u(r)$ is a convex set, the point in $\overline{V}_u(r)$ which minimizes the distance between the origin and a point in $\overline{V}_u(r)$ is unique. Thus, we can define
\begin{equation}\label{yurhat}
    \hat{\bm{y}}_{u}(r):=\argmin_{\bm{y}\in \overline{V}_u(r)}\|\bm{y}\|.
\end{equation}
Let $\hat{\bm{y}}_{i,u}(r)$, $i=L-1,L$ be the projection of the origin on the $(L-1)$-dimensional hyper-plane $\partial\widetilde{H}_{i,u}(r)$, and let
$\hat{\bm{y}}_{L-1,L,u}(r)$ be the projection of the origin on the hyper-plane $\partial\widetilde{H}_{L-1,u}(r)\cap \partial\widetilde{H}_{L,u}(r)$ (when $r>0$, this hyper-plane is $(L-2)$-dimensional, where $L=N(N+1)/2+2>2$). 
Obviously, for any $r\ge 0$ and $u>0$, we have
\begin{equation}\label{eq:Projection1}
    \hat{\bm{y}}_{u}(r)\in\left\{\hat{\bm{y}}_{L-1,u}(r),\hat{\bm{y}}_{L,u}(r),\hat{\bm{y}}_{L-1,L,u}(r)\right\},
\end{equation}
and in particular, 
\begin{equation}\label{eq:Projection2}
    \hat{\bm{y}}_{u}(0)=\hat{\bm{y}}_{L-1,u}(0)=\hat{\bm{y}}_{L,u}(0).
\end{equation}

Recall that the index of a critical point is defined to be the number of negative eigenvalues of its Hessian matrix.
\begin{lemma}\label{Lem:projectionneg}
Let $X$ be qualified. Assume that $X$ also satisfies (\ref{Con:GC1}) and (\ref{Con:GC2}) for some $\tilde{\delta}_{\rho}>0$. Let $\hat{\bm{y}}_{u}(r)$, $\hat{\bm{y}}_{L-1,u}(r)$, $\hat{\bm{y}}_{L,u}(r)$ and $\hat{\bm{y}}_{L-1,L,u}(r)$ be as defined above.
Then for any $u>0$,
\begin{enumerate}[label=(\roman*)]
    \item $\Matri_N(\widetilde{\bm{A}}(r)\hat{\bm{y}}_{i,u}(r))$, $i=L-1,L$ has at least $N-1$ negative eigenvalues for any $r\in[0,\tilde{\delta}_{\rho}]$;
    \item $\Matri_N(\widetilde{\bm{A}}(r)\hat{\bm{y}}_{L-1,L,u}(r))$ has at least $N-1$ negative eigenvalues for any $r\in(0,\tilde{\delta}_{\rho}]$.
\end{enumerate}
These, together with (\ref{eq:Projection1}) and (\ref{eq:Projection2}), imply that
$\Matri_N(\widetilde{\bm{A}}(r)\hat{\bm{y}}_{u}(r))$ has at least $N-1$ negative eigenvalues for any $r\in[0,\tilde{\delta}_{\rho}]$.
\end{lemma}

\begin{proof}

For (i), fix $u>0$ and $r\in[0,\tilde{\delta}_{\rho}]$.
Since $\hat{\bm{y}}_{i,u}(r)$ is the projection of the origin on the hyper-plane $\partial\widetilde{H}_{i,u}(r)$ for $i=L-1,L$, 
it is easy to see
$$\hat{\bm{y}}_{i,u}(r)=\beta_{i,u}(r)\left(\widetilde{\bm{A}}(r)_{(i)}\right)^T$$ for some real number $\beta_{i,u}(r)\neq 0$. Then 
$$0<u=\widetilde{\bm{A}}(r)_{(i)}\hat{\bm{y}}_{i,u}(r)=\beta_{i,u}(r)\left\|\widetilde{\bm{A}}(r)_{(i)}\right\|^2=\beta_{i,u}(r)\bm{\Sigma}(r)[i,i].$$
Since $\bm{\Sigma}(r)$ is positive semi-definite, we have $\beta_{i,u}(r)>0$ for $i=L-1,L$, and then
$$
\begin{aligned}
\Matri_N\left(\widetilde{\bm{A}}(r)\hat{\bm{y}}_{i,u}(r)\right)
&=\beta_{i,u}(r)\Matri_N\left(\widetilde{\bm{A}}(r)\left(\widetilde{\bm{A}}(r)_{(i)}\right)^T\right)\\
&=\beta_{i,u}(r)\Matri_N\left(\bm{\Sigma}(r)^{(i)}\right).
\end{aligned}
$$
Note that by (i) and (ii) in Appendix \ref{FSP}, (\ref{Side1}) and (\ref{Side2}), it is easy to check that  $\Matri_N(\bm{\Sigma}(r)^{(i)})$ is diagonal with the first $N-1$ diagonal elements equal to 
$\bm{\Sigma}(r)[1,i]$ for $i=L-1,L$. In addition, by Condition (\ref{Con:GC2}), we have
$\bm{\Sigma}(r)[1,i]<0$ for $i=L-1,L$. Thus,
$\Matri_N(\widetilde{\bm{A}}(r)\hat{\bm{y}}_{i,u}(r))$, $i=L-1,L$ has at least $N-1$ negative eigenvalues.

For (ii), fix $u>0$ and $r\in(0,\tilde{\delta}_{\rho}]$. It is easy to see
$$\widetilde{\bm{A}}(r)_{(L-1)}\hat{\bm{y}}_{L-1,L,u}(r)=\widetilde{\bm{A}}(r)_{(L)}\hat{\bm{y}}_{L-1,L,u}(r)=u,$$
and for any $\bm{x}\in\mathbb{R}^L$ such that $\widetilde{\bm{A}}(r)_{(L-1)}\bm{x}=\widetilde{\bm{A}}(r)_{(L)}\bm{x}=u$,
$$\hat{\bm{y}}^T_{L-1,L,u}(r)(\bm{x}-\hat{\bm{y}}_{L-1,L,u}(r))=0.$$
This implies 
$$\hat{\bm{y}}_{L-1,L,u}(r)=\beta_{L-1,u}'(r)\left(\widetilde{\bm{A}}(r)_{(L-1)}\right)^T+\beta_{L,u}'(r)\left(\widetilde{\bm{A}}(r)_{(L)}\right)^T,$$
for some constants $\beta_{L-1,u}'(r)$  and $\beta_{L,u}'(r)$.
Then for $i=L-1,L$,
$$
\begin{aligned}
u
&=\widetilde{\bm{A}}(r)_{(i)}\hat{\bm{y}}_{L-1,L,u}(r)\\
&=\widetilde{\bm{A}}(r)_{(i)}\left(\beta_{L-1,u}'(r)\left(\widetilde{\bm{A}}(r)_{(L-1)}\right)^T+\beta_{L,u}'(r)\left(\widetilde{\bm{A}}(r)_{(L)}\right)^T\right)\\
&=\beta_{L-1,u}'(r)\bm{\Sigma}_{}(r)[i,L-1]+\beta_{L,u}'(r)\bm{\Sigma}(r)[i,L],
\end{aligned}
$$
i.e.,
\begin{equation}\label{eq:beta}
    \begin{pmatrix}
    \beta_{L-1,u}'(r)\\ \beta_{L,u}'(r)
    \end{pmatrix}=
    \left(\bm{\Sigma}(r)[(L-1):L,(L-1):L]\right)^{-1}
    \begin{pmatrix}
    u\\ u
    \end{pmatrix}.
\end{equation}
Denote
$$a(r):=\bm{\Sigma}(r)[L-1,L-1]=\bm{\Sigma}(r)[L,L],$$
where the equality comes from the symmetry between $X(\bm{0})$ and $X(u_0r)$ in the definition of $\bm{\Sigma}(r)$,
and
$$b(r):=\bm{\Sigma}(r)[L-1,L]=\bm{\Sigma}(r)[L,L-1].$$
Since $\bm{\Sigma}(r)[(L-1):L,(L-1):L]$
is positive-definite, we have
$$a(r)>0\text{ and }a^2(r)-b^2(r)>0.$$
Then by (\ref{eq:beta}), we have
$$\beta_{L-1,u}'(r)=\beta_{L,u}'(r)=\frac{u}{a^2(r)-b^2(r)}(a(r)-b(r))=\frac{u}{a(r)+b(r)}>0.$$
Thus,  
$$
\begin{aligned}
&~~~~\Matri_N\left(\widetilde{\bm{A}}(r)\hat{\bm{y}}_{L-1,L,u}(r)\right)\\
&=\frac{u}{a(r)+b(r)}\Matri_N\left(\widetilde{\bm{A}}(r)\left(\widetilde{\bm{A}}(r)_{(L-1)}+\widetilde{\bm{A}}(r)_{(L)}\right)^T\right)\\
&=\frac{u}{a(r)+b(r)}\left(\Matri_N\left(\bm{\Sigma}(r)^{(L-1)}\right)+\Matri_N\left(\bm{\Sigma}(r)^{(L)}\right)\right),
%&=\frac{u}{a(r)+b(r)}\left(\beta_{L-1,u}^{-1}(r)\Matri_N(\widetilde{\bm{A}}(r)\hat{\bm{y}}_{L-1,u}(r))+\beta_{L,u}^{-1}(r)\Matri_N(\widetilde{\bm{A}}(r)\hat{\bm{y}}_{L,u}(r))\right),
\end{aligned}
$$
which is diagonal and has at least $N-1$ negative eigenvalues by the proof of (i).
\end{proof}

Recall that in Section \ref{Sec:M1}, we have defined for any $0\le k\le N$,
\begin{equation}\notag
    \begin{aligned}
    D_k:=\Big\{\bm{x}\in\mathbb{R}^{N(N+1)/2}:&\Matri_N(\bm{x})\text{ is non-degenerate} \\ &\text{and has exactly $k$ negative eigenvalues}\Big\}.
    \end{aligned}
\end{equation}
For any $0\le k\le N$ and $r>0$, define
$$\widetilde{G}_k(r):=\left\{\bm{y}\in\mathbb{R}^{L}:\widetilde{\bm{A}}(r)\bm{y}\in \bigcup_{i=k}^N D_i\right\},$$
and let
$$\widetilde{G}_k(0):=\liminf_{r\to 0}\widetilde{G}_k(r).$$
Then for any $u>0$, $0\le k\le N$ and $r\ge 0$, we can further define
\begin{equation}\label{def:Dur_tilde}
    \widetilde{D}_{u,k}(r):=\widetilde{H}_{L-1,u}(r)\cap \widetilde{H}_{L,u}(r)\cap\widetilde{G}_k(r).    
\end{equation}

The following lemma describes the behavior of $\widetilde{D}_{u,k}(r)$ as $r\to r_0$ for any $r_0\in[0,\tilde{\delta}_{\rho}]$, which is important for the proof of Theorem \ref{Theo:MR2}.
\begin{lemma}\label{Lem:AS2}
Let $X$ be qualified under perturbation and satisfy (\ref{Con:GC2}) for some $\tilde{\delta}_{\rho}>0$.
Then for any $u>0$, we have
\begin{enumerate}[label=(\roman*)]
    \item $\widetilde{D}_{u,k}(0)$ has a non-empty interior for any $0\le k\le N-1$;
    \item for any $r_0\in[0,\tilde{\delta}_{\rho}]$ and $0\le k\le N$,
    \begin{equation}\notag
        \lim_{r\to r_0}I_{\widetilde{D}_{u,k}(r)}(\bm{y})=I_{\widetilde{D}_{u,k}(r_0)}(\bm{y})\text{ for almost all $\bm{y}\in\mathbb{R}^L$}.
    \end{equation}
\end{enumerate}
\end{lemma}

\begin{proof}
Fix $u>0$. %Without loss of generality, we can assume $\tilde{\delta}_{\rho}\le \delta_{\rho}$.
For (i),
since $\widetilde{D}_{u,j}(0)\subset \widetilde{D}_{u,i}(0)$ for any $0\le i\le j\le N-1$,
it suffices to show that $\widetilde{D}_{u,N-1}(0)$ contains a non-empty open set.
By (i) of Lemma \ref{Lem:projectionneg}, $\Matri_N(\widetilde{\bm{A}}(0)\hat{\bm{y}}_{L,u}(0))$ has at least $N-1$ negative eigenvalues. Note that by Theorem 5.1 of \cite{kat} and the continuity of $\widetilde{\bm{A}}(r)$ on $r\in [0,\tilde{\delta}_{\rho}]$, the eigenvalues of $\Matri_N(\widetilde{\bm{A}}(r)\bm{y})$ are continuous on $(r,\bm{y})\in [0,\tilde{\delta}_{\rho}]\times \mathbb{R}^L$. Thus,
there exist constants $\delta>0$ and $\gamma>0$, such that  $\Matri_N(\widetilde{\bm{A}}(r)\bm{y})$ has at least $N-1$ negative eigenvalues for any $r\in [0,\delta)$ and $\bm{y}\in B(\hat{\bm{y}}_{L,u}(0),\gamma)$, the $L$-dimensional open ball centered at $\hat{\bm{y}}_{L,u}(0)$ with radius $\gamma$.
Note that for any $\bm{y}'\in \mathbb{R}^L$,
if there exists $r_n'\downarrow 0$ such that
$$\det\left(\Matri_N\left(\widetilde{\bm{A}}\left(r'_n\right)\bm{y}'\right)\right)=0\text{ for any $n\ge 1$},$$
then by (\ref{eq:LimPoly})-(\ref{def:h0}), the uniform convergence on compact sets of $\bm{y}$ given in (\ref{eq:LimPolyh}), (\ref{eq:TA0A0}) and the continuity of $\bm{P}(r)$ on $r\in [0,\delta_{pc}]$,
$$h_{0}\left(\bm{P}^T(0)\bm{y}'\right)=\lim_{n\to\infty}h_{r_n'}\left(\bm{P}^T(r'_n)\bm{y}'\right)=\lim_{n\to\infty}\frac{1}{r_n'}\det\left(\Matri_N\left(\bm{A}\left(r'_n\right)\bm{P}^T(r'_n)\bm{y}'\right)\right)=0.$$
Then by the orthogonality of $\bm{P}(0)$, we have
\begin{equation}\notag
    \bm{y}'\in Q_{h}:=\left\{\bm{P}(0)\bm{y}:\bm{y}\in \mathbb{R}^L,h_0(\bm{y})=0\right\}.    
\end{equation}
Combining all of the above, we have
$$B(\hat{\bm{y}}_{L,u}(0),\gamma)\setminus Q_{h}\subset \widetilde{G}_{N-1}(0).$$
Then% by (\ref{Rel:Hur_tilde}),
$$
\begin{aligned}
\widetilde{D}_{u,N-1}(0)&=\widetilde{H}_{L,u}(0)\cap \widetilde{G}_{N-1}(0)\\
%&=\liminf_{r\to 0}\widetilde{D}_{u,N-1}(r)\\
%&=\liminf_{r\to 0}\widetilde{H}_{L-1,u}(r)\cap \liminf_{r\to 0}\widetilde{H}_{L,u}(r)\cap \widetilde{G}_{N-1}(0)\\
&\supset \left(\widetilde{H}_{L,u}(0)\cap B(\hat{\bm{y}}_{L,u}(0),\gamma)\right)\setminus Q_{h}.
\end{aligned}
$$
Since $\widetilde{H}_{L,u}(0)$ is open and $\hat{\bm{y}}_{L,u}(0)\in\partial\widetilde{H}_{L,u}(0)$, we have
$\widetilde{H}_{L,u}(0)\cap B(\hat{\bm{y}}_{L,u}(0),\gamma)$ is a non-empty open set. Since $h_0(\bm{y})$ is a non-degenerate homogeneous polynomial of $y_1,\dots,y_L$ with degree $N$, we have $Q_{h}$ is closed and $\lambda_L(Q_{h})=0$. Thus, $(\widetilde{H}_{L,u}(0)\cap B(\hat{\bm{y}}_{L,u}(0),\gamma))\setminus Q_{h}$ is also a non-empty open set,
and hence (i) is proved.

As for (ii), fix $0\le k\le N$.
Since the eigenvalues of $\Matri_N(\widetilde{\bm{A}}(r)\bm{y})$ are all continuous on $(r,\bm{y})\in [0,\tilde{\delta}_{\rho}]\times \mathbb{R}^L$,
we have
\begin{equation}\label{Rel:Dur_tilde}
    \widetilde{D}_{u,k}(0)\subset \liminf_{r\to 0}\widetilde{D}_{u,k}(r)\subset \limsup_{r\to 0}\widetilde{D}_{u,k}(r)\subset \widetilde{D}_{u,k}(0)\cup \partial \widetilde{H}_{L,u}(0)\cup Q_{h}.
\end{equation}
The rest of the proof is analogous to the proof of (iii) in Lemma \ref{Lem:AS1}.
\end{proof}

The following is the last preparation for Theorem \ref{Theo:MR2}.
\begin{lemma}\label{Lem:Compare}
For any $a,b>0$
\begin{equation}\notag
    \sup_{k_1>a,k_2-k_1>b}e^{\frac{1}{2}k_1^2u^2}\int_{\|\bm{y}'\|\ge k_2u}\exp\left(-\frac{1}{2}\bm{y}'^T\bm{y}'\right)d\bm{y}'
    \to 0\text{ as }u\to\infty.
\end{equation}
\end{lemma}
\begin{proof}
By the change of variable for spherical coordinates, we get
\begin{equation}\notag
    \int_{\|\bm{y}'\|\ge k_2u}\exp\left(-\frac{1}{2}\bm{y}'^T\bm{y}'\right)d\bm{y}'
    =C\left(\int_{k_2u}^{\infty}e^{-\frac{1}{2}r^2}r^{L-1}dr\right),  
\end{equation}
where $C$ is positive and independent of $k_2$.
Thus, we have
$$
\begin{aligned}
e^{\frac{1}{2}k_1^2u^2}\int_{\|\bm{y}'\|\ge k_2u}\exp\left(-\frac{1}{2}\bm{y}'^T\bm{y}'\right)d\bm{y}' & = C\left(\int_{k_2u}^{\infty}e^{-\frac{1}{2}(r^2-k_1^2u^2)}r^{L-1}dr\right)\\
& = C\left(\int_{k_2u}^{\infty}e^{-\frac{1}{2}(r+k_1u)(r-k_1u)}r^{L-1}dr\right)\\
& \leq C\left(\int_{k_2u}^{\infty}e^{-\frac{1}{2}(r+k_1u)b u}r^{L-1}dr\right)\\
& = Ce^{-\frac{1}{2}k_1b u^2}\left(\int_{k_2u}^{\infty}e^{-\frac{1}{2}b ur}r^{L-1}dr\right)\\
& \leq Ce^{-\frac{1}{2}ab u^2}\left(\int_{(a+b)u}^{\infty}e^{-\frac{1}{2}b ur}r^{L-1}dr\right),
\end{aligned}
$$
which is independent of $k_1, k_2$, and converges to 0 as $u$ tends to infinity.

\end{proof}

%===========================================================================================================
%===========================================================================================================

\subsection{Proof of Theorem \ref{Theo:MR2}}

\begin{proof}
The proof of (i) is similar to that of Theorem \ref{Theo:MR1}.
Fix $u>0$. Since $X$ is isotropic, it suffices to show that
$\Psi_u(\bm{u}_0r)$ converges as $r\to 0$.
By the change of variable 
$$\bm{y}=\widetilde{\bm{A}}^{-1}(r)(\bm{x}'',x,z)^T,$$
(\ref{expression}) and (\ref{eq:TA0A0}), we have for any $r>0$,
\begin{equation}\label{Eq:Phi_g_tilde}
    \begin{aligned}
    \Psi_u(\bm{u}_0r)
    &=\frac{\sum_{k=0}^{N-2} f_{u,k}(\bm{u}_0r)}{f_{u,N-1}(\bm{u}_0r)+f_{u,N}(\bm{u}_0r)}\\
    &=\frac{\int_{x,z>u}\int_{\cup_{k=0}^{N-2}D_k}|\det(\bm{x}'')|p_{\bm{t}}(\bm{x}'',x,z|\bm{0},\bm{0})d\bm{x}''dxdz,}{\int_{x,z>u}\int_{D_{N-1}\cup D_{N}}|\det(\bm{x}'')|p_{\bm{t}}(\bm{x}'',x,z|\bm{0},\bm{0})d\bm{x}''dxdz}\\
    &=\frac
    {\int_{\widetilde{D}_{u,0}(r)\setminus \widetilde{D}_{u,N-1}(r)}\left|r^{-1}\det\left(\Matri_N\left(\widetilde{\bm{A}}(r)\bm{y}\right)\right)\right|p_L(\bm{y})d\bm{y}}
    {\int_{\widetilde{D}_{u,N-1}(r)}\left|r^{-1}\det\left(\Matri_N\left(\widetilde{\bm{A}}(r)\bm{y}\right)\right)\right|p_L(\bm{y})d\bm{y}}\\
    &=\frac
    {\int_{\widetilde{D}_{u,0}(r)\setminus \widetilde{D}_{u,N-1}(r)}\left|r^{-1}\det\left(\Matri_N\left(\bm{A}(r)\bm{P}^T(r)\bm{y}\right)\right)\right|p_L\left(\bm{P}^T(r)\bm{y}\right)d\bm{y}}
    {\int_{\widetilde{D}_{u,N-1}(r)}\left|r^{-1}\det\left(\Matri_N\left(\bm{A}(r)\bm{P}^T(r)\bm{y}\right)\right)\right|p_L\left(\bm{P}^T(r)\bm{y}\right)d\bm{y}}\\
    &=\frac
    {\int_{\widetilde{D}_{u,0}(r)\setminus \widetilde{D}_{u,N-1}(r)}g_r\left(\bm{P}^T(r)\bm{y}\right)d\bm{y}}
    {\int_{\widetilde{D}_{u,N-1}(r)}g_r\left(\bm{P}^T(r)\bm{y}\right)d\bm{y}}\\
    &=\frac
    {\int_{\widetilde{D}_{u,0}(r)\setminus \widetilde{D}_{u,N-1}(r)}\tilde{g}_r(\bm{y})d\bm{y}}
    {\int_{\widetilde{D}_{u,N-1}(r)}\tilde{g}_r(\bm{y})d\bm{y}},
    \end{aligned}
\end{equation}
where $\widetilde{D}_{u,k}(r)$, $0\le k\le N-1$ are defined in (\ref{def:Dur_tilde}) and discussed
in Lemma \ref{Lem:AS2},  $g_r$ and $p_L$ are as defined in (\ref{def:gr}) and (\ref{def:pL}), and for any $\bm{y}\in\mathbb{R}^L$ and $r\ge 0$, 
\begin{equation}\label{def:gr_tilde}
    \tilde{g}_r(\bm{y}):=g_r\left(\bm{P}^T(r)\bm{y}\right).   
\end{equation}
The fourth equality follows from the fact that $p_L(y)=p_L(\bm{P}^T(r)y)$.
Then we can use (ii) of Lemma \ref{Lem:AS2}, which plays the role of (iii) of Lemma \ref{Lem:AS1} in the proof of Theorem \ref{Theo:MR1}, to show that
$\int_{\widetilde{D}_{u,k}(r)}\tilde{g}_r(\bm{y})d\bm{y}$, $0\le k\le N-1$ are continuous functions of $r\in[0,\tilde{\delta}_{\rho}]$. 
In addition, by
Remark \ref{Rem:regular} and (i) of Lemma \ref{Lem:AS2}, we have for any $0\le k\le N-1$,
$$
\int_{\widetilde{D}_{u,k}(0)}\tilde{g}_0(\bm{y})d\bm{y}>0.$$
Then 
$$
\begin{aligned}
\lim_{r\to 0}\Psi_u(\bm{u}_0r)
&=\frac
{\lim_{r\to 0}\int_{\widetilde{D}_{u,0}(r)\setminus \widetilde{D}_{u,N-1}(r)}\tilde{g}_{r}(\bm{y})d\bm{y}}
{\lim_{r\to 0}\int_{\widetilde{D}_{u,N-1}(r)}\tilde{g}_{r}(\bm{y})d\bm{y}}
&=\frac
{\int_{\widetilde{D}_{u,0}(0)\setminus \widetilde{D}_{u,N-1}(0)}\tilde{g}_{0}(\bm{y})d\bm{y}}
{\int_{\widetilde{D}_{u,N-1}(0)}\tilde{g}_{0}(\bm{y})d\bm{y}}.
\end{aligned}
$$

As for (ii), it suffices to show that for any $\varepsilon>0$, there exists a constant $U>0$ such that for any $r\in[0,\tilde{\delta}_{\rho}]$ and $u>U$,
$$\Psi_u(\bm{u}_0r)<\varepsilon.$$
To this end, we need to introduce some notations and concepts.

We start from $\hat{\bm{y}}_{u}(r)$, $u>0$ and $r\ge 0$ as defined in (\ref{yurhat}).
It is easy to see $\hat{\bm{y}}_{u}(r)=u\hat{\bm{y}}_{1}(r)$ for any $u>0$ and $r\ge 0$.
By Lemma \ref{Lem:projectionneg}, there exists a positive function $\gamma(r)$ of $r\in[0,\tilde{\delta}_{\rho}]$ such that
for any $\bm{y}\in B(\hat{\bm{y}}_{1}(r),\gamma(r))$, $\Matri_N(\widetilde{\bm{A}}(r)\bm{y})$ has at least $N-1$ negative eigenvalues.
Since $\widetilde{\bm{A}}(r)$ is continuous on $r\in [0,\tilde{\delta}_{\rho}]$,  $\{\hat{\bm{y}}_{1}(r),r\in[0,\tilde{\delta}_{\rho}]\}$ is compact and covered by 
$\{B(\hat{\bm{y}}_{1}(r),\gamma(r)),r\in[0,\tilde{\delta}_{\rho}]\}$. Then by the Heine–Borel theorem, there exists a finite open subcover $\{B(\hat{\bm{y}}_{1}(r),\gamma(r)),r\in\{r_1,\dots,r_n\}\}$  of $\{\hat{\bm{y}}_{1}(r):r\in[0,\tilde{\delta}_{\rho}]\}$ for some positive integer $n$ and $r_1,\dots,r_n\in[0,\tilde{\delta}_{\rho}]$. Let $\gamma'$ be the distance between the two compact sets $\{\hat{\bm{y}}_{1}(r),r\in[0,\tilde{\delta}_{\rho}]\}$ and $\partial(\bigcup_{k=1}^n B(\hat{\bm{y}}_{1}(r_k),\gamma(r_k)))$, i.e.,
\begin{equation}\label{def:gammap}
    \gamma':=\min\left\{\|\hat{\bm{y}}_{1}(r)-\bm{y}\|:\bm{y}\in \partial\left(\bigcup_{k=1}^n B(\hat{\bm{y}}_{1}(r_k),\gamma(r_k))\right)\text{ and }r\in[0,\tilde{\delta}_{\rho}]\right\}.    
\end{equation}
It is not hard to see $\gamma'>0$.
Then for any $r\in[0,\tilde{\delta}_{\rho}]$ and $\bm{y}\in B(\hat{\bm{y}}_{1}(r),\gamma')$, 
$\Matri_N(\widetilde{\bm{A}}(r)\bm{y})$ has at least $N-1$ negative eigenvalues.
Therefore, for any $u>0$,  $\bm{y}\in B(\hat{\bm{y}}_{u}(r),\gamma'u)$ and $r\in[0,\tilde{\delta}_{\rho}]$,
$\Matri_N(\widetilde{\bm{A}}(r)\bm{y})$ also has at least $N-1$ negative eigenvalues.

The next step is to consider some distances.
For any $u>0$ and $r\in[0,\tilde{\delta}_{\rho}]$, let $\gamma_{2,u}(r)$ be the distance between the origin and the compact set $\partial B(\hat{\bm{y}}_{u}(r),\gamma'u)\cap \partial(\widetilde{H}_{L-1,u}(r)\cap \widetilde{H}_{L,u}(r))$, let
$$\gamma_{0,u}(r):=\|\hat{\bm{y}}_{u}(r)\|,$$
and let
$$\gamma_{1,u}(r):=\frac{1}{2}(\gamma_{0,u}(r)+\gamma_{2,u}(r)).$$
It is easy to see for any $u>0$, $\gamma_{i,u}(r)$, $i=0,1,2$ 
are all positive and continuous on $r\in[0,\tilde{\delta}_{\rho}]$ with $\gamma_{i,u}(r)=\gamma_{i,1}(r)u$. Thus, for any $u>0$ and $i=0,1,2$, we have
\begin{equation}\label{eq:PositiveMinGamma}
    \min_{r\in[0,\tilde{\delta}_{\rho}]}\gamma_{i,u}(r)>0.    
\end{equation}
Illustrations of these distances with $\widetilde{P}_{L-1,u}(r)$, $\widetilde{P}_{L,u}(r)$, $B(\hat{\bm{y}}_{u}(r),\gamma'u)$ and $B(\bm{0}_L,\gamma_{2,u}(r))$ are provided in Figure \ref{Illgamma2ur}.

\begin{figure}
\centering
    \includegraphics[width=0.45\textwidth]{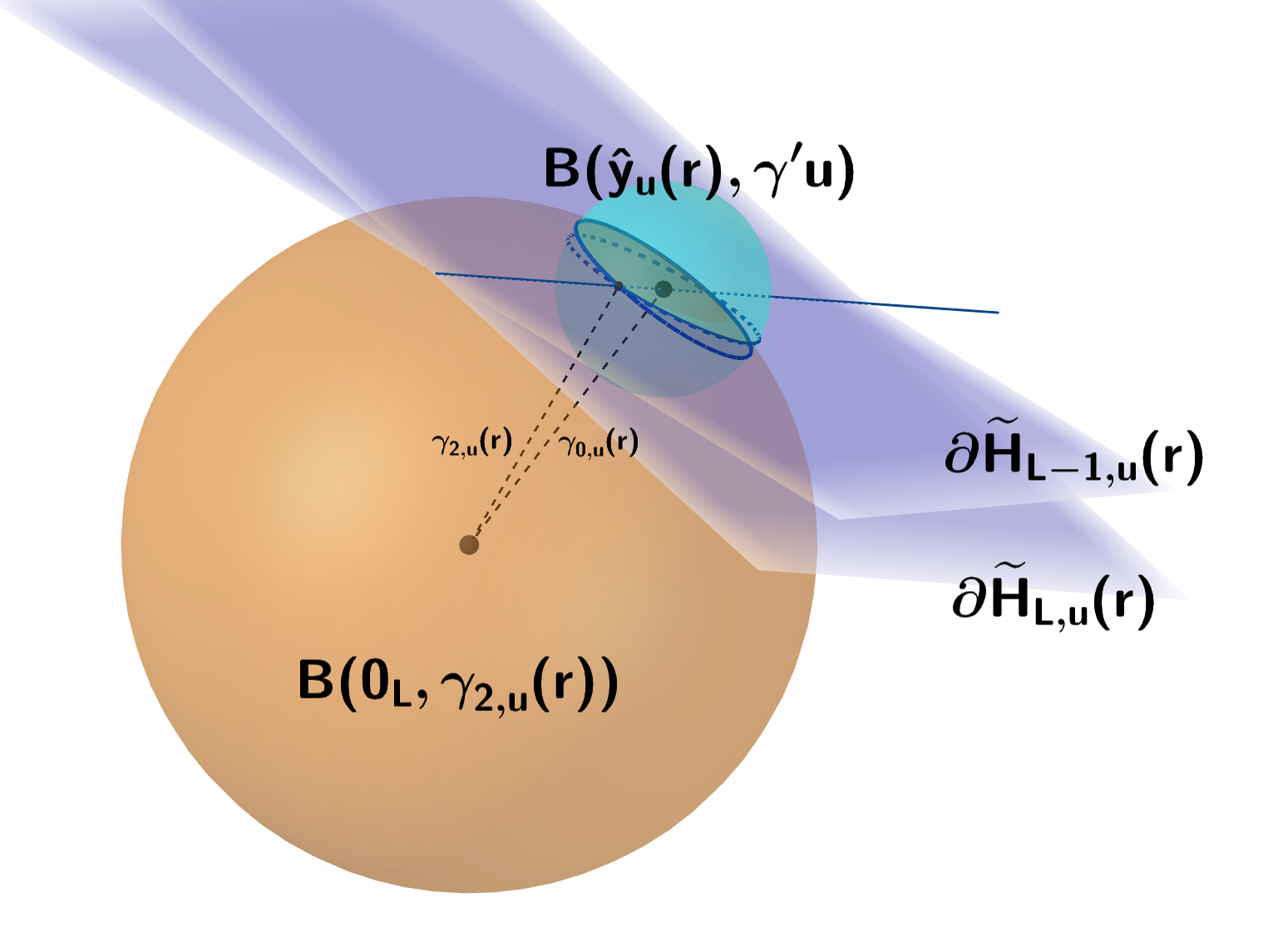}
    \hfill
    \includegraphics[width=0.45\textwidth]{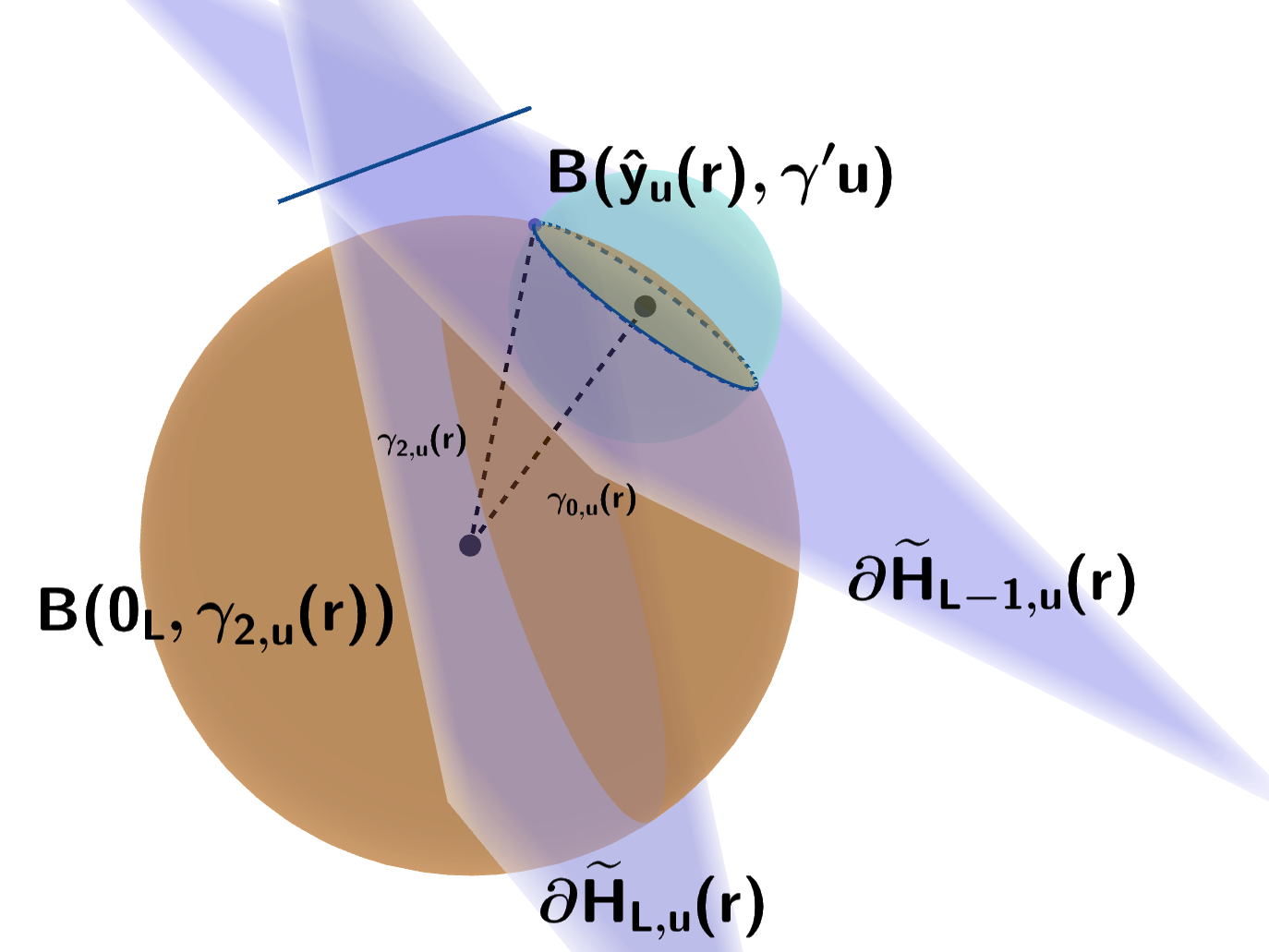}
    \caption[Illustrations of $\partial\widetilde{H}_{L-1,u}(r)$, $\partial\widetilde{H}_{L,u}(r)$, $B(\hat{\bm{y}}_{u}(r),\gamma'u)$ and $B(\bm{0}_L,\gamma_{2,u}(r))$]{Illustrations of $\partial\widetilde{H}_{L-1,u}(r)$, $\partial\widetilde{H}_{L,u}(r)$, $B(\hat{\bm{y}}_{u}(r),\gamma'u)$ and $B(\bm{0}_L,\gamma_{2,u}(r))$ when $\hat{\bm{y}}_{u}(r)=\hat{\bm{y}}_{L-1,L,u}(r)$ (left) and $\hat{\bm{y}}_{u}(r)=\hat{\bm{y}}_{L-1,u}(r)$ (right).} 
    \label{Illgamma2ur}
\end{figure}

Fix $r\in[0,\tilde{\delta}_{\rho}]$ and $u>0$. We have some results about these distances.
Firstly, note that
\begin{equation}\label{eq:gamma2gamma1}
\gamma_{2,u}(r)-\gamma_{1,u}(r)=\frac{1}{2}\left(\gamma_{2,u}(r)-\gamma_{0,u}(r)\right)= \frac{\gamma'^2u^2}{2(\gamma_{2,u}(r)+\gamma_{0,u}(r))}>\frac{\gamma'^2u^2}{4\max_{r\in[0,\tilde{\delta}_{\rho}]}\gamma_{2,1}(r)}>0,
\end{equation}
where the second equality comes from the observation
\begin{equation}\notag
    \gamma_{2,u}^2(r)= \gamma_{0,u}^2(r)+\gamma'^2u^2,    
\end{equation}
which can be easily seen geometrically. In addition, for any $\bm{y}\in\mathbb{R}^L$ such that $\widetilde{\bm{A}}(\bm{u}_0r)_{(i)}\bm{y}>u$ for $i=L-1,L$ and $\|\bm{y}\|<\gamma_{2,u}(r)$, suppose  $\|\bm{y}-\hat{\bm{y}}_{u}(r)\|\ge \gamma'u$, then 
\begin{equation}\notag
    \|\bm{y}\|^2<\gamma_{2,u}^2(r)=\gamma_{0,u}^2(r)+\gamma'^2u^2\le \|\hat{\bm{y}}_u(r)\|^2 + \|\bm{y}-\hat{\bm{y}}_{u}(r)\|^2.
\end{equation}
By discussing each case in (\ref{eq:Projection1}),
this means that the origin and $\bm{y}$ are on the same side of the $(L-1)$-dimensional hyper-plane $\partial\widetilde{H}_{L-1,u}(r)$ or $\partial\widetilde{H}_{L,u}(r)$. Then we get $\widetilde{\bm{A}}(r)_{(L-1)}\bm{y}<u$ or $\widetilde{\bm{A}}(r)_{(L)}\bm{y}<u$, resulting in a contradiction.
Thus, we must have $\|\bm{y}-\hat{\bm{y}}_{u}(r)\|< \gamma'u$, which implies
\begin{equation}\label{eq:HHB}
    \widetilde{H}_{L-1,u}(r)\cap \widetilde{H}_{L,u}(r)\cap B(\hat{\bm{y}}_{u}(r),\gamma'u)\supset \widetilde{H}_{L-1,u}(r)\cap \widetilde{H}_{L,u}(r)\cap B(\bm{0}_L,\gamma_{2,u}(r)),
\end{equation}
or equivalently,
\begin{equation}\notag
    \widetilde{H}_{L-1,u}(r)\cap \widetilde{H}_{L,u}(r)\cap B^c(\hat{\bm{y}}_{u}(r),\gamma'u)\subset \widetilde{H}_{L-1,u}(r)\cap \widetilde{H}_{L,u}(r)\cap B^c(\bm{0}_L,\gamma_{2,u}(r)).
\end{equation}
Then by (\ref{def:gammap}), 
\begin{equation}\label{num_dom_mag}
    \widetilde{D}_{u,0}(r)\setminus \widetilde{D}_{u,N-1}(r)
    \subset \widetilde{H}_{L-1,u}(r)\cap \widetilde{H}_{L,u}(r)\cap B^c(\hat{\bm{y}}_{u}(r),\gamma'u)
    \subset B^c(\bm{0}_L,\gamma_{2,u}(r)),
\end{equation}
where recall $\widetilde{D}_{u,k}(r)$, $0\le k\le N-1$ was defined in (\ref{def:Dur_tilde}).

Now we further define some useful subsets of $\mathbb{R}^L$.
For any $r\in[0,\tilde{\delta}_{\rho}]$ and $u>0$, define
\begin{equation}\label{def:Cur}
C_u(r):=\widetilde{H}_{L-1,u}(r)\cap \widetilde{H}_{L,u}(r)\cap B(\bm{0}_L,\gamma_{1,u}(r)).
\end{equation}
It is easy to see $C_u(r)=uC_1(r)$, $\lambda_L(C_1(r))>0$, 
and
$\lambda_L(C_1(r))$ is a continuous function of $r\in [0,\tilde{\delta}_{\rho}]$. Then we have
\begin{equation}\label{eq:C}
    C_{1,\tilde{\delta}_{\rho}}:=\min_{r\in[0,\tilde{\delta}_{\rho}]}\lambda_L(C_1(r))>0.    
\end{equation}
For any $r\ge 0$ and $\bm{y}\in\mathbb{R}^L$, let
$$\tilde{h}_r(\bm{y}):=h_r\left(\bm{P}^T(r)\bm{y}\right),$$
where $h_r$ is as defined in (\ref{def:hr}) and (\ref{def:h0}). Indeed, by (\ref{def:gr}) and (\ref{def:gr_tilde}), we have
\begin{equation}\label{eq:h_g_tilde}
    \tilde{g}_r(\bm{y})=\left|\tilde{h}_r(\bm{y})\right|p_L(\bm{y})=r^{-1}\left|\det\left(\Matri_N\left(\widetilde{\bm{A}}(r)\bm{y}\right)\right)\right|p_L(\bm{y}).    
\end{equation}
For any $r\ge 0$, define
$$Q_r:=\left\{\bm{y}\in\mathbb{R}^L:\tilde{h}_r(\bm{y})=0\right\}$$
Since $h_r(\bm{y})$ is a polynomial of $\bm{y}$, we have
\begin{equation}\label{eq:Qr}
    \lambda_L\left(Q_r\right)=0.
\end{equation}
By the continuity of $\widetilde{\bm{A}}(r)$ on $r\in[0,\tilde{\delta}_{\rho}]$, $\tilde{h}_r(\bm{y})$ is continuous on $r\in[0,\tilde{\delta}_{\rho}]$ for any $\bm{y}\in\mathbb{R}^L$ (not just continuous on $[0,\delta_{pc}]$ as $h_r(\bm{y})$).
By (\ref{eq:gamma2gamma1}) and (\ref{eq:HHB}),  we also have
\begin{equation}\label{eq:CuHHB}
    \begin{aligned}
        C_u(r)
        &\subset \widetilde{H}_{L-1,u}(r)\cap \widetilde{H}_{L,u}(r)\cap B(\bm{0}_L,\gamma_{2,u}(r))\\
        &\subset \widetilde{H}_{L-1,u}(r)\cap \widetilde{H}_{L,u}(r)\cap B(\hat{\bm{y}}_{u}(r),\gamma'u)\\
        &\subset \widetilde{D}_{u,N-1}(r)\cup Q_r.
    \end{aligned}
\end{equation}
For any $\varepsilon>0$, denote
\begin{equation}\label{def:Ieps}
    I_{\varepsilon}:=\left\{r\in[0,\tilde{\delta}_{\rho}] :\lambda_L\left(C_1(r)\cap \left\{\bm{y}\in\mathbb{R}^L:\left|\tilde{h}_r(\bm{y})\right|\ge \varepsilon\right\}\right)/\lambda_L(C_1(r))\le 0.5\right\}.
\end{equation}
It is easy to see that for any given $\varepsilon$, $\lambda_L(C_1(r)\cap \{\bm{y}\in\mathbb{R}^L:|\tilde{h}_r(\bm{y})|\ge \varepsilon\})/\lambda_L(C_1(r))$ is a continuous function of $r\in [0,\tilde{\delta}_{\rho}]$. Moreover, for each $r\in [0,\tilde{\delta}_{\rho}]$, $\lambda_L(C_1(r)\cap \{\bm{y}\in\mathbb{R}^L:|\tilde{h}_r(\bm{y})|\ge \varepsilon\})/\lambda_L(C_1(r))$ converges monotonically to 1 as $\varepsilon$ decreases to 0. Hence by Dini's theorem, this convergence is uniform. Consequently, 
%Then $I_{\varepsilon}$ is compact for any $\varepsilon>0$.
%Suppose $I_{\varepsilon}\neq\emptyset$ for any $\varepsilon>0$, then $\{I_{1/n},n\ge 1\}$ forms a %decreasing sequence of non-empty compact set. By Cantor's intersection theorem, this implies 
%$$\bigcap_{n=1}^{\infty} I_{1/n}\neq \emptyset,$$ 
%i.e., there exists a constant $r_0\in[0,\tilde{\delta}_{\rho}]$ such that for any $n\ge 1$,
%$$
%    \lambda_L\left(C_1(r_0)\cap \{\bm{y}\in\mathbb{R}^L:\left|\tilde{h}_{r_0}(\bm{y})\right|\ge %1/n\}\right)\le 0.5\lambda_L(C_1(r_0)).
%$$
%By taking limits on the both sides of the above inequality as $n\to\infty$, we have
%\begin{equation}\label{limc1r0}
%    \lambda_L\left(C_1(r_0)\cap \{\bm{y}\in\mathbb{R}^L:\left|\tilde{h}_{r_0}(\bm{y})\right|>0\}\right)\le 0.5\lambda_L(C_1(r_0)).
%\end{equation}
%However, by (\ref{eq:C}), (\ref{eq:Qr}) and (\ref{limc1r0}), we have
%\begin{equation}\notag
%    \begin{aligned}
%    0<C_{1,\tilde{\delta}_{\rho}}\le\lambda_L(C_1(r_0))
%    &\le\lambda_L\left(C_1(r_0)\cap \left\{\bm{y}\in\mathbb{R}^L:\left|\tilde{h}_{r_0}(\bm{y})\right|> 0\right\}\right)+\lambda_L(Q_{r_0})\\
%    &=\lambda_L\left(C_1(r_0)\cap \left\{\bm{y}\in\mathbb{R}^L:\left|\tilde{h}_{r_0}(\bm{y})\right|> 0\right\}\right)\\
%    &\le 0.5\lambda_L(C_1(r_0)),
%    \end{aligned}
%\end{equation}
%resulting in a contradiction. Therefore, 
there exists some $\eta>0$ independent of $u$ such that $I_{\eta}=\emptyset$. For any $u>0$ and $r\in[0,\tilde{\delta}_{\rho}]$, let
\begin{equation}\label{def_Tur}
    T_u(r):=C_u(r)\cap \left\{\bm{y}\in\mathbb{R}^L:\left|\tilde{h}_{r}(\bm{y})\right|\ge \eta u^N\right\}.
\end{equation}
By (\ref{def:hr}), (\ref{def:h0}) and (\ref{def:Cur}), it is easy to see
$$T_u(r)=T_1(r)u^L.$$
Then by (\ref{def:Ieps}),
\begin{equation}\label{Tur_low_bd}
    \lambda_L(T_u(r))=\lambda_L(T_1(r))u^L\ge \frac{1}{2}\lambda_L(C_1(r))u^L\ge\frac{1}{2}C_{1,\tilde{\delta}_{\rho}}u^L
\end{equation}
By (\ref{eq:CuHHB}) and $\eta>0$, we also have
\begin{equation}\label{den_dom_shr}
    T_u(r)\subset \widetilde{D}_{u,N-1}(r).
\end{equation}

Now we return to the proof of (ii).
Given $\varepsilon>0$, we can choose $U$ (as mentioned at the start of the proof of (ii)) as follows.
Firstly, choose $\theta\in(0,1)$ satisfying
\begin{equation}\label{eq:gapneeded}
\min_{r\in [0,\tilde{\delta}_{\rho}]}\left\{\sqrt{1-\theta}\gamma_{2,1}(r)-\gamma_{1,1}(r)\right\}>0.
\end{equation}
Such $\theta$ exists because of (\ref{eq:gamma2gamma1}) and the fact that $\gamma_{2,1}(r)$ is bounded from above for $r\in[0,\tilde{\delta}_{\rho}]$. Note that by a similar reasoning as in (\ref{Rel:constdom}),
there exists a constant $\widetilde{C}>0$ such that for any $r\in[0,\tilde{\delta}_{\rho}]$ and $\bm{y}\in\mathbb{R}^L$,
$$
    \tilde{h}_r(\bm{y})\le \widetilde{C}\sum_{v_1,\dots,v_N\in\{1,\dots,L\}}\left|y_{v_1}\cdots y_{v_N}\right|.
$$
Thus, there exists $U_1>0$ such that
\begin{equation}\label{1stAuxRes}
    \tilde{h}_r(\bm{y})\le \exp\left(\frac{1}{2}\theta\bm{y}^T\bm{y}\right),
\end{equation}
for any $r\in [0,\tilde{\delta}_{\rho}]$ and $\bm{y}\in\mathbb{R}^L$ satisfying $\|\bm{y}\|\ge U_1\min_{s\in[0,\tilde{\delta}_{\rho}]}\gamma_{2,1}(s)$ (see (\ref{eq:PositiveMinGamma})). Then
by (\ref{eq:PositiveMinGamma}), (\ref{eq:gapneeded}) and Lemma \ref{Lem:Compare} with $k_2=\sqrt{1-\theta}\gamma_{2,1}(r)$ and $k_1=\gamma_{1,1}(r)$,  there exists $U_2>0$ such that
\begin{equation}\label{2ndAuxRes}
    \frac{(1-\theta)^{-\frac{L}{2}}\int_{uB^c(\bm{0}_L,\sqrt{1-\theta}\gamma_{2,1}(r))}\exp\left(-\frac{1}{2}\bm{y}'^T\bm{y}'\right)d\bm{y}'}{\frac{1}{2}\eta C_{1,\tilde{\delta}_{\rho}} u^{L+N}\exp\left(-\frac{1}{2}\gamma_{1,1}^2(r)u^2\right)}<\varepsilon,   
\end{equation}
for any $u\ge U_2$. 
Then we can simply take $U:=\max(U_1,U_2)$. Indeed,
for any $r\in[0,\tilde{\delta}_{\rho}]$ and $u>\max(U_1,U_2)$,
$$
    \begin{aligned}
  \Psi_u(\bm{u}_0r)&=\frac{\int_{\widetilde{D}_{u,0}(r)\setminus \widetilde{D}_{u,N-1}(r)}\tilde{g}_r(\bm{y})d\bm{y}}{\int_{\widetilde{D}_{u,N-1}(r)}\tilde{g}_r(\bm{y})d\bm{y}}~~\text{(by (\ref{Eq:Phi_g_tilde}))}\\
    &=
    \frac{\int_{\widetilde{D}_{u,0}(r)\setminus \widetilde{D}_{u,N-1}(r)}\tilde{h}_r(\bm{y})\exp\left(-\frac{1}{2}\bm{y}^T\bm{y}\right)d\bm{y}}{\int_{\widetilde{D}_{u,N-1}(r)}\tilde{h}_r(\bm{y})\exp\left(-\frac{1}{2}\bm{y}^T\bm{y}\right)d\bm{y}}~~\text{(by (\ref{eq:h_g_tilde}))}\\
    &\le
    \frac{\int_{B^c(\bm{0}_L,\gamma_{2,u}(r)) }\tilde{h}_r(\bm{y})\exp\left(-\frac{1}{2}\bm{y}^T\bm{y}\right)d\bm{y}}{\int_{T_u(r)}\tilde{h}_r(\bm{y})\exp\left(-\frac{1}{2}\bm{y}^T\bm{y}\right)d\bm{y}} ~~\text{(by (\ref{num_dom_mag}) and (\ref{den_dom_shr}))}\\
    &\le
    \frac{\int_{B^c(\bm{0}_L,\gamma_{2,u}(r))}\exp\left(-\frac{1}{2}(1-\theta)\bm{y}^T\bm{y}\right)d\bm{y}}{\eta u^N\int_{T_u(r)}\exp\left(-\frac{1}{2}\bm{y}^T\bm{y}\right)d\bm{y}} ~~\text{(by (\ref{def_Tur}) and (\ref{1stAuxRes}))}\\
    &\le
    \frac{(1-\theta)^{-\frac{L}{2}}\int_{uB^c(\bm{0}_L,\sqrt{1-\theta}\gamma_{2,1}(r))}\exp\left(-\frac{1}{2}\bm{y}'^T\bm{y}'\right)d\bm{y}'}{\frac{1}{2}\eta C_{1,\tilde{\delta}_{\rho}} u^{L+N}\exp\left(-\frac{1}{2}\gamma_{1,1}^2(r)u^2\right)} ~~(\text{by $\bm{y}'=\sqrt{1-\theta}\bm{y}$, (\ref{def:Cur}), (\ref{def_Tur}), (\ref{Tur_low_bd})})\\
    &< \varepsilon, ~~(\text{by  (\ref{2ndAuxRes})})
    \end{aligned}
$$
as desired.

\end{proof}

\subsection{A Corollary of the Main Results}

Note that (ii) of Lemma \ref{Lem:AS2} holds for $k=N$.
Then by a similar proof to (i) of Theorem \ref{Theo:MR2}, the limit of $\frac{f_{u,N}(\bm{t})}{f_{u,N-1}(\bm{t})+f_{u,N}(\bm{t})}$ as $\|\bm{t}\|\to 0$ also exists for any $u>0$.
The following is an immediate corollary of Theorems \ref{Theo:MR1} and \ref{Theo:MR2}.
\begin{corollary}\label{Cor}
    Let $X$ be qualified under perturbation and satisfy Condition (\ref{Con:GC2}). Then
    $$\lim_{u\to\infty}\lim_{\|\bm{t}\|\to 0}\frac{f_{u,N}(\bm{t})}{f_{u,N-1}(\bm{t})+f_{u,N}(\bm{t})}=\frac{1}{2}.$$
\end{corollary}

Since we are conditioning on having a critical point at the origin with unknown index, the ratio in Corollary \ref{Cor} being very close to $\frac{1}{2}$ implies that a pair of very close critical points should consist of one local maximum and one critical point with index $N-1$.

Intuitively, a connected component of a high excursion set most likely contains exactly one critical point (one global maximum) or three critical points (two local maxima and one critical point with index $N-1$). Thus, Corollary \ref{Cor} predicts that if a connected component contains three critical points, then the critical point with index $N-1$ will be very close to one of the two local maxima.

\newpage
%\appendix
\begin{appendices}

\section{Proof of Lemma \ref{Lem:Sigma}}\label{App:Analytic}
Fix $N\ge 2$. For any $r>0$ and $\bm{t}\in\mathbb{R}^N$, let $B(\bm{t},r)$ be the $N$-dimensional open ball centered at $\bm{t}$ with radius $r$.
Since $\rho(x)$ is four times continuously differentiable on $x\in[0,\delta_{\rho}^2]$  (see $\delta_{\rho}$ in Definition \ref{Def:qualified}), all the partial derivatives of $R(\bm{t})$ up to fourth order are continuous on $\bm{t}\in B(\bm{0}_N,\delta_{\rho})$. 
Recall that $L=N(N+1)/2+2$. For convenience, we will still use $\triangledown^2 X$ for its usual vectorization, but one can easily distinguish between them using context. For any $\bm{t}\in B(\bm{0}_N,\delta_{\rho})\setminus\{\bm{0}_N\}$, denote
\begin{enumerate}
    \item $\bm{V}_{11}(\bm{t})\in\mathbb{R}^{L\times L}$: the covariance matrix of the random $L$-vector 
    $(\triangledown^2X(\bm{t}),X(\bm{t}),X(\bm{0}))$;
    
    \item $\bm{V}_{12}(\bm{t})\in\mathbb{R}^{L\times 2N}$: the covariance matrix between random vectors
    $$(\triangledown^2X(\bm{t}),X(\bm{t}),X(\bm{0}))~~\text{and}~~\left(\triangledown X(\bm{t}),\triangledown X(\bm{0})\right);$$
    
    \item $\bm{V}_{22}(\bm{t})\in\mathbb{R}^{2N\times 2N}$: the covariance matrix of the random $(2N)$-vector
    $$\left(\triangledown X(\bm{t}),\triangledown X(\bm{0})\right).$$
\end{enumerate}
Since $\bm{\Sigma}(\bm{t})$, $\bm{t}\in\mathbb{R}^N\setminus\{\bm{0}_N\}$ is the covariance matrix of the random $L$-vector 
$$\left(\triangledown^2X(\bm{t}),X(\bm{t}),X(\bm{0})|\triangledown X(\bm{t})=\triangledown X(\bm{0})=\bm{0}_{N}\right),$$
by the properties of multivariate Gaussian distribution, we have for any $\bm{t}\in B(\bm{0}_N,\delta_{\rho})\setminus\{\bm{0}_N\}$,
$$\bm{\Sigma}(\bm{t})=\bm{V}_{11}(\bm{t})-\bm{V}_{12}(\bm{t})\bm{V}_{22}^{-1}(\bm{t})\bm{V}_{12}^T(\bm{t}).$$

\subsection{The Blocked Covariance Matrix}
For convenience, we adopt the following notations for $\bm{t}\in B(\bm{0}_N,\delta_{\rho})\setminus\{\bm{0}_N\}$: 
\begin{enumerate}

\item $\bm{G}_{00}(\bm{t}):=\Cov[X(\bm{0}),X(\bm{t})]=R(\bm{t})$;

\item $\bm{G}_{01}(\bm{t}):=\Cov[X(\bm{0}),\triangledown X(\bm{t})]=-\Cov[\triangledown X(\bm{0}),X(\bm{t})]=(R_1(\bm{t}),\dots, R_N(\bm{t}))$;

\item $\bm{G}_{20}(\bm{t}):=\Cov[\triangledown^2X(\bm{0}),X(\bm{t})]=\Cov[X(\bm{0}),\triangledown^2X(\bm{t})]\in\mathbb{R}^{(L-2)\times 1}$; 

\item $\bm{G}_{21}(\bm{t}):=\Cov[\triangledown^2X(\bm{0}),\triangledown X(\bm{t})]=-\Cov[\triangledown X(\bm{0}),\triangledown^2X(\bm{t})]\in\mathbb{R}^{(L-2)\times N}$;

\item $\bm{G}_{22}(\bm{t}):=\Cov[\triangledown^2X(\bm{0}),\triangledown^2 X(\bm{t})]\in\mathbb{R}^{(L-2)\times (L-2)}$.
\end{enumerate}        
Their relationships with the covariance function $R$ are given by Lemma \ref{Lem:MSD}.
In particular, by (\ref{R1})-(\ref{R4}), we have
\begin{enumerate}
    \item $\bm{G}_{00}(\bm{0}):=\Cov[X(\bm{0}),X(\bm{0})]=R(\bm{0})$; 
    \item $\bm{G}_{01}(\bm{0}):=\Cov[X(\bm{0}),\triangledown X(\bm{0})]=(R_1(\bm{0}),\dots, R_N(\bm{0}))=\bm{0}_{1\times N}$;
    \item $\bm{G}_{20}(\bm{0}):=\Cov[\triangledown^2X(\bm{0}),X(\bm{0})]\in\mathbb{R}^{(L-2)\times 1}$;
    \item $\bm{G}_{21}(\bm{0}):=\Cov[\triangledown^2X(\bm{0}),\triangledown X(\bm{0})]=\bm{0}_{(L-2)\times N}$;
    \item $\bm{G}_{22}(\bm{0}):=\Cov[\triangledown^2X(\bm{0}),\triangledown^2 X(\bm{0})]\in\mathbb{R}^{(L-2)\times (L-2)}$.
\end{enumerate}
Immediately, we have
\begin{equation}\label{v11}
\bm{V}_{11}(\bm{t})=
\begin{pmatrix}
\bm{G}_{22}(\bm{0}) & \bm{G}_{20}(\bm{0}) & \bm{G}_{20}(\bm{t}) \\
\bm{G}_{20}^T(\bm{0}) & \bm{G}_{00}(\bm{0}) & \bm{G}_{00}(\bm{t}) \\
\bm{G}_{20}^T(\bm{t}) & \bm{G}_{00}^T(\bm{t}) & \bm{G}_{00}(\bm{0}) 
\end{pmatrix}
\end{equation}
and
\begin{equation}\label{v12}
\bm{V}_{12}(\bm{t})
=
\begin{pmatrix}
-\bm{G}_{21}(\bm{0}) & -\bm{G}_{21}(\bm{t}) \\
-\bm{G}_{01}(\bm{0}) & -\bm{G}_{01}(\bm{t}) \\
\bm{G}_{01}(\bm{t}) & -\bm{G}_{01}(\bm{0})
\end{pmatrix}\\
=
\begin{pmatrix}
\bm{0}_{(L-2)\times N} & -\bm{G}_{21}(\bm{t}) \\
\bm{0}_{1\times N} & -\bm{G}_{01}(\bm{t}) \\
\bm{G}_{01}(\bm{t}) & \bm{0}_{1\times N}
\end{pmatrix}.
\end{equation}
As for $\bm{V}_{22}^{-1}(\bm{t})$,
by (\ref{R2}), we have
$$\Cov[\triangledown X(\bm{0}),\triangledown X(\bm{0})]=-2\rho^{(1)}(0)\bm{I}_N$$
and
$$\Cov[\triangledown X(\bm{t}),\triangledown X(\bm{0})]=-2\rho^{(1)}(\|\bm{t}\|^2)\bm{I}_N-4\rho^{(2)}(\|\bm{t}\|^2)\bm{t}\bm{t}^T.$$
Recall that by (\ref{ineq:rho1rho2}), $\rho^{(1)}(0)<0$. We can define 
$$k_1(\bm{t}):=\frac{\rho^{(1)}(\|\bm{t}\|^2)}{\rho^{(1)}(0)}\text{ and }k_2(\bm{t}):=\frac{2\rho^{(2)}(\|\bm{t}\|^2)}{\rho^{(1)}(0)},$$
for any $\bm{t}\in B(\bm{0}_N,\delta_{\rho})\setminus\{\bm{0}_N\}$. Thus,
$$\bm{V}_{22}(\bm{t})=-2\rho^{(1)}(0)
\begin{pmatrix}
\bm{I}_N & k_1(\bm{t})\bm{I}_N+k_2(\bm{t})\bm{t}\bm{t}^T\\
k_1(\bm{t})\bm{I}_N+k_2(\bm{t})\bm{t}\bm{t}^T & \bm{I}_N
\end{pmatrix}$$
and then
$$\bm{V}_{22}^{-1}(\bm{t})=-\frac{1}{2\rho^{(1)}(0)}
\begin{pmatrix}
\bm{I}_N & k_1(\bm{t})\bm{I}_N+k_2(\bm{t})\bm{t}\bm{t}^T\\
k_1(\bm{t})\bm{I}_N+k_2(\bm{t})\bm{t}\bm{t}^T & \bm{I}_N
\end{pmatrix}^{-1}.$$
To further calculate $\bm{V}_{22}^{-1}(\bm{t})$, we need the following two facts which can be easily checked:
\begin{enumerate}
    \item for any symmetric matrix $\bm{B}\in\mathbb{R}^{N\times N}$ such that $\bm{I}_N-\bm{B}^2$ is invertible,  
    $$
    \begin{pmatrix}
    \bm{I}_N & \bm{B} \\
    \bm{B} & \bm{I}_N
    \end{pmatrix}^{-1}
    =
    \begin{pmatrix}
    (\bm{I}_N-\bm{B}^2)^{-1} & -\bm{B}(\bm{I}_N-\bm{B}^2)^{-1}\\
    -(\bm{I}_N-\bm{B}^2)^{-1}\bm{B} & (\bm{I}_N-\bm{B}^2)^{-1}
    \end{pmatrix};
    $$
    
    \item (The Sherman–Morrison Formula) for any $\bm{w},\bm{v}\in\mathbb{R}^{N}$ such that $1+\bm{v}^T\bm{w}\neq 0$,  
    $$(\bm{I}_N+\bm{w}\bm{v}^T)^{-1}=\bm{I}_N-\frac{\bm{w}\bm{v}^T}{1+\bm{v}^T\bm{w}}.$$
\end{enumerate}
By taking 
$\bm{B}=k_1(\bm{t})\bm{I}_N+k_2(\bm{t})\bm{t}\bm{t}^T$ in Fact 1,
we get
\begin{equation}\label{eq:IBB}
    \begin{aligned}
    \bm{I}_N-\bm{B}^2
    &=\left(1-k_1^2(\bm{t})\right)\left(\bm{I}_N-k_3(\bm{t})\bm{t}\bm{t}^T\right),
    \end{aligned}
\end{equation}
where by letting $k_*(\bm{t}):=k_1(\bm{t})+k_2(\bm{t})\|\bm{t}\|^2$,
$$k_3(\bm{t}):=\frac{2k_1(\bm{t})k_2(\bm{t})+k_2^2(\bm{t})\|\bm{t}\|^2}{1-k_1^2(\bm{t})}=\frac{k_2(\bm{t})(k_1(\bm{t})+k_{*}(\bm{t}))}{1-k_1^2(\bm{t})}.$$
By (\ref{absrho1}), 
\begin{equation}\label{k1lessthan1}
    1-k_1^2(\bm{t})=\frac{(\rho^{(1)}(0))^2-(\rho^{(1)}(\|\bm{t}\|^2))^2}{(\rho^{(1)}(0))^2}>0.  
\end{equation}
Thus, $k_3(\bm{t})$ is well-defined for any $\bm{t}\in B(\bm{0}_N,\delta_{\rho})\setminus\{\bm{0}_N\}$. To apply Fact 2 to the right-hand side of (\ref{eq:IBB}), we still need to show
$1-k_3(\bm{t})\|\bm{t}\|^2\neq 0$ for any $\bm{t}\in B(\bm{0}_N,\delta_{\rho})\setminus\{\bm{0}_N\}$. One can easily check that
$$\left(1-k_3(\bm{t})\|\bm{t}\|^2\right)\left(1-k_1^2(\bm{t})\right)=1-k_*^2(\bm{t}).$$
Taking $\bm{t}':=(t_1',t_2',\dots,t_N')=(0,\dots,0,\|\bm{t}\|)$ and by (\ref{rho1ineq}), we have
$$|k_*(\bm{t})|=|k_*(\bm{t}')|=\left|k_1(\bm{t}')+\|\bm{t}\|^2k_2(\bm{t}')\right|=\left|\frac{\rho^{(1)}(\|\bm{t}'\|^2)+2t_N'^2\rho^{(2)}(\|\bm{t}'\|^2)}{\rho^{(1)}(0)}\right|<1.$$
Thus, for any $\bm{t}\in B(\bm{0}_N,\delta_{\rho})\setminus\{\bm{0}_N\}$, 
\begin{equation}\label{kstarlessthan1}
    1-k_3(\bm{t})\|\bm{t}\|^2=\frac{1-k_*^2(\bm{t})}{1-k_1^2(\bm{t})}>0.
\end{equation}
Then by Fact 2,
we have
$$
\begin{aligned}
(\bm{I}_N-\bm{B}^2)^{-1}
&=\frac{1}{1-k_1^2(\bm{t})}\left(\bm{I}_N+\frac{k_3(\bm{t})\bm{t}\bm{t}^T}{1-k_3(\bm{t})\|\bm{t}\|^2}\right)\\
&=\frac{1}{1-k_1^2(\bm{t})}\left(\bm{I}_N+k_4(\bm{t})\bm{t}\bm{t}^T\right),
\end{aligned}
$$
where 
$$k_4(\bm{t}):=\frac{k_3(\bm{t})}{1-k_3(\bm{t})\|\bm{t}\|^2}=\frac{\left(1-k_1^2(\bm{t})\right)k_3(\bm{t})}{1-k_*^2(\bm{t})}=\frac{k_2(\bm{t})(k_1(\bm{t})+k_*(\bm{t}))}{1-k_*^2(\bm{t})}.$$
Then
$$
\begin{aligned}
&~~~~(\bm{I}_N-\bm{B}^2)^{-1}\bm{B}=\bm{B}(\bm{I}_N-\bm{B}^2)^{-1}\\
&=\frac{1}{1-k_1^2(\bm{t})}\left(\bm{I}_N+k_4(\bm{t})\bm{t}\bm{t}^T\right)\left(k_1(\bm{t})\bm{I}_N+k_2(\bm{t})\bm{t}\bm{t}^T\right)\\
&=\frac{1}{1-k_1^2(\bm{t})}\left(k_1(\bm{t})\bm{I}_N+\left(k_2(\bm{t})+k_1(\bm{t})k_4(\bm{t})+k_2(\bm{t})k_4(\bm{t})\|\bm{t}\|^2\right)\bm{t}\bm{t}^T\right)\\
&=\frac{1}{1-k_1^2(\bm{t})}\left(k_1(\bm{t})\bm{I}_N+k_5(\bm{t})\bm{t}\bm{t}^T\right),
\end{aligned}
$$
where 
$$
\begin{aligned}
k_5(\bm{t}):
&=k_2(\bm{t})+k_1(\bm{t})k_4(\bm{t})+k_2(\bm{t})k_4(\bm{t})\|\bm{t}\|^2\\
&=k_2(\bm{t})+k_{*}(\bm{t})k_4(\bm{t})\\
&=\frac{k_2(\bm{t})(1+ k_1(\bm{t})k_*(\bm{t}))}{1-k_*^2(\bm{t})}.
\end{aligned}
$$
It is easy to see $k_i(\bm{t})$, $i=1,2,\dots,5$ are all well-defined for any $\bm{t}\in B(\bm{0}_N,\delta_{\rho})\setminus\{\bm{0}_N\}$.

In summary, for any $\bm{t}\in B(\bm{0}_N,\delta_{\rho})\setminus\{\bm{0}_N\}$, we have
\begin{equation}\label{v22inverse}
\begin{aligned}
&~~~~-2\rho^{(1)}(0)\bm{V}_{22}^{-1}(\bm{t})\\
&=
\begin{pmatrix}
\frac{1}{1-k_1^2(\bm{t})}\bm{I}_N & -\frac{k_1(\bm{t})}{1-k_1^2(\bm{t})}\bm{I}_N\\
-\frac{k_1(\bm{t})}{1-k_1^2(\bm{t})}\bm{I}_N & \frac{1}{1-k_1^2(\bm{t})}\bm{I}_N\\
\end{pmatrix}+
\begin{pmatrix}
\frac{k_4(\bm{t})}{1-k_1^2(\bm{t})}\bm{t}\bm{t}^T & -\frac{k_5(\bm{t})}{1-k_1^2(\bm{t})}\bm{t}\bm{t}^T\\
-\frac{k_5(\bm{t})}{1-k_1^2(\bm{t})}\bm{t}\bm{t}^T & \frac{k_4(\bm{t})}{1-k_1^2(\bm{t})}\bm{t}\bm{t}^T\\
\end{pmatrix}\\
&= 
\frac{1}{1-k_1^2(\bm{t})}
\begin{pmatrix}
\bm{I}_N & -k_1(\bm{t})\bm{I}_N\\
-k_1(\bm{t})\bm{I}_N & \bm{I}_N\\
\end{pmatrix}+\frac{1}{1-k_1^2(\bm{t})}
\begin{pmatrix}
k_4(\bm{t})\bm{t}\bm{t}^T & -k_5(\bm{t})\bm{t}\bm{t}^T\\
-k_5(\bm{t})\bm{t}\bm{t}^T & k_4(\bm{t})\bm{t}\bm{t}^T\\
\end{pmatrix},
\end{aligned}
\end{equation}
where 
$$k_1(\bm{t})=\frac{\rho^{(1)}(\|\bm{t}\|^2)} {\rho^{(1)}(0)},~~k_2(\bm{t})=\frac{2\rho^{(2)}(\|\bm{t}\|^2)}{\rho^{(1)}(0)},~~k_*(\bm{t})=k_1(\bm{t})+k_2(\bm{t})\|\bm{t}\|^2,$$
$$k_4(\bm{t})
=\frac{k_2(\bm{t})(k_1(\bm{t})+k_*(\bm{t}))}{1-k_*^2(\bm{t})}\text{ and }
k_5(\bm{t})
=\frac{k_2(\bm{t})(1+ k_1(\bm{t})k_*(\bm{t}))}{1-k_*^2(\bm{t})}.
$$
Here one should note that 
\begin{equation}\label{Eq:k4k5}
\begin{aligned}
k_5(\bm{t})-k_4(\bm{t})
&=\frac{k_2(\bm{t})(1-k_1(\bm{t}))(1-k_{*}(\bm{t}))}{1-k_{*}^2(\bm{t})}\\
&=\frac{k_2(\bm{t})(1-k_1(\bm{t}))}{1+k_{*}(\bm{t})}\\
&\to 0\text{ as $\|\bm{t}\|\to 0$.}
\end{aligned}
\end{equation}
In consequence, by (\ref{v11}), (\ref{v12}) and (\ref{v22inverse}),
$$
\begin{aligned}
&~~~~-2\rho^{(1)}(0)(1-k_1^2(\bm{t}))\bm{V}_{12}(\bm{t})\bm{V}^{-1}_{22}(\bm{t})\bm{V}_{12}^T(\bm{t})\\
&=\begin{pmatrix}
\bm{0}_{(L-2)\times N} & -\bm{G}_{21}(\bm{t}) \\
\bm{0}_{1\times N} & -\bm{G}_{01}(\bm{t}) \\
\bm{G}_{01}(\bm{t}) & \bm{0}_{1\times N}
\end{pmatrix}
\begin{pmatrix}
\bm{I}_N & -k_1(\bm{t})\bm{I}_N\\
-k_1(\bm{t})\bm{I}_N & \bm{I}_N\\
\end{pmatrix}
\begin{pmatrix}
\bm{0}_{N\times (L-2)} & \bm{0}_{N\times 1} & \bm{G}_{01}^T(\bm{t}) \\
-\bm{G}_{21}^T(\bm{t}) & -\bm{G}_{01}^T(\bm{t}) & \bm{0}_{N\times 1}
\end{pmatrix}\\
&~~~~
+\begin{pmatrix}
\bm{0}_{(L-2)\times N} & -\bm{G}_{21}(\bm{t}) \\
\bm{0}_{1\times N} & -\bm{G}_{01}(\bm{t}) \\
\bm{G}_{01}(\bm{t}) & \bm{0}_{1\times N}
\end{pmatrix}
\begin{pmatrix}
k_4(\bm{t})\bm{t}\bm{t}^T & -k_5(\bm{t})\bm{t}\bm{t}^T\\
-k_5(\bm{t})\bm{t}\bm{t}^T & k_4(\bm{t})\bm{t}\bm{t}^T\\
\end{pmatrix}
\begin{pmatrix}
\bm{0}_{N\times (L-2)} & \bm{0}_{N\times 1} & \bm{G}_{01}^T(\bm{t}) \\
-\bm{G}_{21}^T(\bm{t}) & -\bm{G}_{01}^T(\bm{t}) & \bm{0}_{N\times 1}
\end{pmatrix}\\
&=
\begin{pmatrix}
\bm{G}_{21}(\bm{t})\bm{G}_{21}^T(\bm{t}) & \bm{G}_{21}(\bm{t})\bm{G}_{01}^T(\bm{t}) & k_1(\bm{t})\bm{G}_{21}(\bm{t})\bm{G}_{01}^T(\bm{t})\\
\bm{G}_{01}(\bm{t})\bm{G}_{21}^T(\bm{t}) & \bm{G}_{01}(\bm{t})\bm{G}_{01}^T(\bm{t}) &k_1(\bm{t})\bm{G}_{01}(\bm{t})\bm{G}_{01}^T(\bm{t}) \\
k_1(\bm{t})\bm{G}_{01}(\bm{t})\bm{G}_{21}^T(\bm{t}) & k_1(\bm{t})\bm{G}_{01}(\bm{t})\bm{G}_{01}^T(\bm{t}) & \bm{G}_{01}(\bm{t})\bm{G}_{01}^T(\bm{t})
\end{pmatrix}\\
&~~~~
+
\begin{pmatrix}
k_4(\bm{t})\bm{G}_{21}(\bm{t})\bm{t}\bm{t}^T\bm{G}_{21}^T(\bm{t}) & k_4(\bm{t})\bm{G}_{21}(\bm{t})\bm{t}\bm{t}^T\bm{G}_{01}^T(\bm{t}) & k_5(\bm{t})\bm{G}_{21}(\bm{t})\bm{t}\bm{t}^T\bm{G}_{01}^T(\bm{t}) \\
k_4(\bm{t})\bm{G}_{01}(\bm{t})\bm{t}\bm{t}^T\bm{G}_{21}^T(\bm{t}) & k_4(\bm{t})\bm{G}_{01}(\bm{t})\bm{t}\bm{t}^T\bm{G}_{01}^T(\bm{t}) & k_5(\bm{t})\bm{G}_{01}(\bm{t})\bm{t}\bm{t}^T\bm{G}_{01}^T(\bm{t})\\
k_5(\bm{t})\bm{G}_{01}(\bm{t})\bm{t}\bm{t}^T\bm{G}_{21}^T(\bm{t}) & k_5(\bm{t})\bm{G}_{01}(\bm{t})\bm{t}\bm{t}^T\bm{G}_{01}^T(\bm{t}) & k_4(\bm{t})\bm{G}_{01}(\bm{t})\bm{t}\bm{t}^T\bm{G}_{01}^T(\bm{t})
\end{pmatrix},
\end{aligned}
$$
and then the blocked version of $\bm{\Sigma}(\bm{t})$ is given by
\begin{equation}\label{block}
\begin{aligned}
&~~~~\bm{\Sigma}(\bm{t})=\bm{V}_{11}(\bm{t})-\bm{V}_{12}(\bm{t})\bm{V}^{-1}_{22}(\bm{t})\bm{V}_{12}^T(\bm{t})\\
&=
\begin{pmatrix}
\bm{G}_{22}(\bm{0}) & \bm{G}_{20}(\bm{0}) & \bm{G}_{20}(\bm{t}) \\
\bm{G}_{20}^T(\bm{0}) & \bm{G}_{00}(\bm{0}) & \bm{G}_{00}(\bm{t}) \\
\bm{G}_{20}^T(\bm{t}) & \bm{G}_{00}^T(\bm{t}) & \bm{G}_{00}(\bm{0}) 
\end{pmatrix}\\
&~~~~+
\frac{1}{2\rho^{(1)}(0)}\frac{1}{1-k_1^2(\bm{t})}
\begin{pmatrix}
\bm{G}_{21}(\bm{t})\bm{G}_{21}^T(\bm{t}) & \bm{G}_{21}(\bm{t})\bm{G}_{01}^T(\bm{t}) & k_1(\bm{t})\bm{G}_{21}(\bm{t})\bm{G}_{01}^T(\bm{t})\\
\bm{G}_{01}(\bm{t})\bm{G}_{21}^T(\bm{t}) & \bm{G}_{01}(\bm{t})\bm{G}_{01}^T(\bm{t}) &k_1(\bm{t})\bm{G}_{01}(\bm{t})\bm{G}_{01}^T(\bm{t}) \\
k_1(\bm{t})\bm{G}_{01}(\bm{t})\bm{G}_{21}^T(\bm{t}) & k_1(\bm{t})\bm{G}_{01}(\bm{t})\bm{G}_{01}^T(\bm{t}) & \bm{G}_{01}(\bm{t})\bm{G}_{01}^T(\bm{t})
\end{pmatrix}\\
&~~~~
+\frac{1}{2\rho^{(1)}(0)}\frac{1}{1-k_1^2(\bm{t})}
\begin{pmatrix}
k_4(\bm{t})\bm{G}_{21}(\bm{t})\bm{t}\bm{t}^T\bm{G}_{21}^T(\bm{t}) & k_4(\bm{t})\bm{G}_{21}(\bm{t})\bm{t}\bm{t}^T\bm{G}_{01}^T(\bm{t}) & k_5(\bm{t})\bm{G}_{21}(\bm{t})\bm{t}\bm{t}^T\bm{G}_{01}^T(\bm{t}) \\
k_4(\bm{t})\bm{G}_{01}(\bm{t})\bm{t}\bm{t}^T\bm{G}_{21}^T(\bm{t}) & k_4(\bm{t})\bm{G}_{01}(\bm{t})\bm{t}\bm{t}^T\bm{G}_{01}^T(\bm{t}) & k_5(\bm{t})\bm{G}_{01}(\bm{t})\bm{t}\bm{t}^T\bm{G}_{01}^T(\bm{t})\\
k_5(\bm{t})\bm{G}_{01}(\bm{t})\bm{t}\bm{t}^T\bm{G}_{21}^T(\bm{t}) & k_5(\bm{t})\bm{G}_{01}(\bm{t})\bm{t}\bm{t}^T\bm{G}_{01}^T(\bm{t}) & k_4(\bm{t})\bm{G}_{01}(\bm{t})\bm{t}\bm{t}^T\bm{G}_{01}^T(\bm{t})
\end{pmatrix}.
\end{aligned}
\end{equation}
Fix a direction $\bm{u}\in\mathbb{S}^{N-1}$, then by (\ref{block}), each element in the covariance matrix $\bm{\Sigma}(\bm{u}r)$ is a continuous function of $r\in(0,\delta_{\rho}]$. 

In the following sections, we will calculate the
asymptotic expansion of $\bm{\Sigma}(\bm{t})$ as $\|\bm{t}\|\to 0$. 
To make the following proofs better organized, we will first calculate the asymptotic expansions associated with coefficients $k_i(\bm{t})$, $i=1,\dots,5$ as $\|\bm{t}\|\to 0$. Then the calculation will be performed separately for each of the three parts in (\ref{block}):
\begin{enumerate}
    \item the main part
    \begin{equation}\notag
    \begin{aligned}
    &~~~~\bm{\Sigma}(\bm{t})[1:(L-2),1:(L-2)]\\
    &=\bm{G}_{22}(\bm{0})+\frac{1}{2\rho^{(1)}(0)(1-k_1^2(\bm{t}))}\bm{G}_{21}(\bm{t})\bm{G}_{21}^T(\bm{t})+\frac{k_4(\bm{t})}{2\rho^{(1)}(0)(1-k_1^2(\bm{t}))}\bm{G}_{21}(\bm{t})\bm{t}\bm{t}^T\bm{G}_{21}^T(\bm{t});
    \end{aligned}
    \end{equation}

    \item the side part
    \begin{equation}\notag
    \begin{aligned}
    &~~~~\bm{\Sigma}(\bm{t})[1:(L-2),L-1]\\
    &=\bm{G}_{20}(\bm{0})+\frac{1}{2\rho^{(1)}(0)(1-k_1^2(\bm{t}))}\bm{G}_{21}(\bm{t})\bm{G}_{01}^T(\bm{t})
    +\frac{k_4(\bm{t})}{2\rho^{(1)}(0)(1-k_1^2(\bm{t}))}\bm{G}_{21}(\bm{t})\bm{t}\bm{t}^T\bm{G}_{01}^T(\bm{t})
    \end{aligned}
    \end{equation}
    and
    \begin{equation}\notag
    \begin{aligned}
    &~~~~\bm{\Sigma}(\bm{t})[1:(L-2),L]\\
    &=\bm{G}_{20}(\bm{t})+\frac{k_1(\bm{t})}{2\rho^{(1)}(0)(1-k_1^2(\bm{t}))}\bm{G}_{21}(\bm{t})\bm{G}_{01}^T(\bm{t})
    +\frac{k_5(\bm{t})}{2\rho^{(1)}(0)(1-k_1^2(\bm{t}))}\bm{G}_{21}(\bm{t})\bm{t}\bm{t}^T\bm{G}_{01}^T(\bm{t});
    \end{aligned}
    \end{equation}
    
    \item the corner part
    \begin{equation}\notag
    \begin{aligned}
    &~~~~\bm{\Sigma}(\bm{t})[L-1,L-1]=\bm{\Sigma}(\bm{t})[L,L]\\
    &=\bm{G}_{00}(\bm{0})+\frac{1}{2\rho^{(1)}(0)(1-k_1^2(\bm{t}))}\bm{G}_{01}(\bm{t})\bm{G}_{01}^T(\bm{t})
    +\frac{k_4(\bm{t})}{2\rho^{(1)}(0)(1-k_1^2(\bm{t}))}\bm{G}_{01}(\bm{t})\bm{t}\bm{t}^T\bm{G}_{01}^T(\bm{t}) 
    \end{aligned}
    \end{equation}
    and
    \begin{equation}\notag
    \begin{aligned}
    &~~~~\bm{\Sigma}(\bm{t})[L-1,L]=\bm{\Sigma}(\bm{t})[L,L-1]\\
    &=\bm{G}_{00}(\bm{t})+\frac{k_1(\bm{t})}{2\rho^{(1)}(0)(1-k_1^2(\bm{t}))}\bm{G}_{01}(\bm{t})\bm{G}_{01}^T(\bm{t})
    +\frac{k_5(\bm{t})}{2\rho^{(1)}(0)(1-k_1^2(\bm{t}))}\bm{G}_{01}(\bm{t})\bm{t}\bm{t}^T\bm{G}_{01}^T(\bm{t}).  
    \end{aligned}
    \end{equation}
\end{enumerate}

\subsection{Asymptotic Expansions of the Coefficients}
Note that by Condition (3) in Definition \ref{Def:qualified}, we have
$$\rho(x)=\rho(0)+\rho^{(1)}(0)x+\frac{1}{2}\rho^{(2)}(0)x^2+\frac{1}{6}\rho^{(3)}(0)x^3+o(x^3)\text{ as $x\downarrow 0$.}$$
In this part, we would like to use the above expansion to expand 
\begin{enumerate}[label=(\roman*)]
    \item $\frac{\|\bm{t}\|^2}{2\rho^{(1)}(0)(1-k_1^2(\bm{t}))}$,
    \item $\frac{k_4(\bm{t})\|\bm{t}\|^2}{2\rho^{(1)}(0)(1-k_1^2(\bm{t}))}$,
    \item $\frac{k_1(\bm{t})\|\bm{t}\|^2}{2\rho^{(1)}(0)(1-k_1^2(\bm{t}))}$,
    \item $\frac{k_5(\bm{t})\|\bm{t}\|^2}{2\rho^{(1)}(0)(1-k_1^2(\bm{t}))}$.
\end{enumerate}

For (i),
note that
$$
\begin{aligned}
\left(\rho^{(1)}(0)+\rho^{(1)}(\|\bm{t}\|^2)\right)^{-1}
&=\left(\rho^{(1)}(0)+\rho^{(1)}(0)\right)^{-1}-\left(\rho^{(1)}(0)+\rho^{(1)}(0)\right)^{-2}\rho^{(2)}(0)\|\bm{t}\|^2+o(\|\bm{t}\|^2)\\
&=\left(2\rho^{(1)}(0)\right)^{-1}-\left(2\rho^{(1)}(0)\right)^{-2}\rho^{(2)}(0)\|\bm{t}\|^2+o(\|\bm{t}\|^2)\\
\end{aligned}
$$
and
$$
\begin{aligned}
\|\bm{t}\|^2\left(\rho^{(1)}(0)-\rho^{(1)}(\|\bm{t}\|^2)\right)^{-1}
&=-\|\bm{t}\|^2\left(\rho^{(2)}(0)\|\bm{t}\|^2+\frac{1}{2}\rho^{(3)}(0)\|\bm{t}\|^4+o(\|\bm{t}\|^4)\right)^{-1}\\
&=-\left(\rho^{(2)}(0)+\frac{1}{2}\rho^{(3)}(0)\|\bm{t}\|^2+o(\|\bm{t}\|^2)\right)^{-1}\\
&=-\rho^{(2)}(0)^{-1}+\frac{1}{2}\rho^{(2)}(0)^{-2}\rho^{(3)}(0)\|\bm{t}\|^2+o(\|\bm{t}\|^2).
\end{aligned}
$$
Then we have
\begin{equation}\label{1exp}
\begin{aligned}
&~~~~\frac{\|\bm{t}\|^2}{2\rho^{(1)}(0)(1-k_1^2(\bm{t}))}\\
&=\frac{1}{2}\rho^{(1)}(0)\|\bm{t}\|^2\left(\rho^{(1)}(0)^{2}-\rho^{(1)}(\|\bm{t}\|^2)^{2}\right)^{-1}\\
&=\frac{1}{2}\rho^{(1)}(0)\left(\rho^{(1)}(0)+\rho^{(1)}(\|\bm{t}\|^2)\right)^{-1}\|\bm{t}\|^2\left(\rho^{(1)}(0)-\rho^{(1)}(\|\bm{t}\|^2)\right)^{-1}\\
&=\frac{1}{2}\rho^{(1)}(0)\left(\left(2\rho^{(1)}(0)\right)^{-1}-\left(2\rho^{(1)}(0)\right)^{-2}\rho^{(2)}(0)\|\bm{t}\|^2+o(\|\bm{t}\|^2)\right)\\
&~~~~\left(-\rho^{(2)}(0)^{-1}+\frac{1}{2}\rho^{(2)}(0)^{-2}\rho^{(3)}(0)\|\bm{t}\|^2+o(\|\bm{t}\|^2)\right)\\
&=-\frac{1}{4}\rho^{(2)}(0)^{-1}+\frac{1}{8}\left(\rho^{(1)}(0)^{-1}+\rho^{(2)}(0)^{-2}\rho^{(3)}(0)\right)\|\bm{t}\|^2+o(\|\bm{t}\|^2)\\
&=:a_0+b_0\|\bm{t}\|^2+o(\|\bm{t}\|^2),
\end{aligned}
\end{equation}
where 
$$a_0=-\frac{1}{4}\rho^{(2)}(0)^{-1}\text{ and }
b_0
=\frac{1}{8}\left(\rho^{(1)}(0)^{-1}+\rho^{(2)}(0)^{-2}\rho^{(3)}(0)\right).
$$

For (ii), we first note that
\begin{equation}\label{kstarexp}
\begin{aligned}
k_*(\bm{t})
&=k_1(\bm{t})+k_2(\bm{t})\|\bm{t}\|^2\\
&=\rho^{(1)}(0)^{-1}\left(\rho^{(1)}(\|\bm{t}\|^2)+2\rho^{(2)}(\|\bm{t}\|^2)\|\bm{t}\|^2\right)\\
&=\rho^{(1)}(0)^{-1}\left(\rho^{(1)}(0)+\rho^{(2)}(0)\|\bm{t}\|^2+\frac{1}{2}\rho^{(3)}(0)\|\bm{t}\|^4+2\rho^{(2)}(0)\|\bm{t}\|^2+2\rho^{(3)}(0)\|\bm{t}\|^4+o(\|\bm{t}\|^4)\right)\\
&=\rho^{(1)}(0)^{-1}\left(\rho^{(1)}(0)+3\rho^{(2)}(0)\|\bm{t}\|^2+\frac{5}{2}\rho^{(3)}(0)\|\bm{t}\|^4+o(\|\bm{t}\|^4)\right)\\
&=1+3\rho^{(1)}(0)^{-1}\rho^{(2)}(0)\|\bm{t}\|^2+\frac{5}{2}\rho^{(1)}(0)^{-1}\rho^{(3)}(0)\|\bm{t}\|^4+o(\|\bm{t}\|^4),
\end{aligned}
\end{equation}
which is followed by
\begin{equation}\notag
k_1(\bm{t})+k_*(\bm{t})
=2+4\rho^{(1)}(0)^{-1}\rho^{(2)}(0)\|\bm{t}\|^2+3\rho^{(1)}(0)^{-1}\rho^{(3)}(0)\|\bm{t}\|^4+o(\|\bm{t}\|^4)    
\end{equation}
and
\begin{equation}\notag
\begin{aligned}
&~~~~1-k_*^2(\bm{t})\\
&=1-\left(1+3\rho^{(1)}(0)^{-1}\rho^{(2)}(0)\|\bm{t}\|^2+\frac{5}{2}\rho^{(1)}(0)^{-1}\rho^{(3)}(0)\|\bm{t}\|^4+o(\|\bm{t}\|^4)\right)^2\\
&=-6\rho^{(1)}(0)^{-1}\rho^{(2)}(0)\|\bm{t}\|^2-\left(5\rho^{(1)}(0)^{-1}\rho^{(3)}(0)+9\rho^{(1)}(0)^{-2}\rho^{(2)}(0)^2\right)\|\bm{t}\|^4+o(\|\bm{t}\|^4),
\end{aligned}
\end{equation}
and thus
\begin{equation}\label{kstarinv}
\begin{aligned}
&~~~~\|\bm{t}\|^2(1-k_*^2(\bm{t}))^{-1}\\
&=-\left(6\rho^{(1)}(0)^{-1}\rho^{(2)}(0)+\left(5\rho^{(1)}(0)^{-1}\rho^{(3)}(0)+9\rho^{(1)}(0)^{-2}\rho^{(2)}(0)^2\right)\|\bm{t}\|^2+o(\|\bm{t}\|^2)\right)^{-1}\\
&=-\left(6\rho^{(1)}(0)^{-1}\rho^{(2)}(0)\right)^{-1}\\
&~~~~+\left(6\rho^{(1)}(0)^{-1}\rho^{(2)}(0)\right)^{-2}\left(5\rho^{(1)}(0)^{-1}\rho^{(3)}(0)+9\rho^{(1)}(0)^{-2}\rho^{(2)}(0)^2\right)\|\bm{t}\|^2+o(\|\bm{t}\|^2)\\
&=-\frac{1}{6}\rho^{(1)}(0)\rho^{(2)}(0)^{-1}+\left(\frac{5}{36}\rho^{(1)}(0)\rho^{(2)}(0)^{-2}\rho^{(3)}(0)+\frac{1}{4}\right)\|\bm{t}\|^2+o(\|\bm{t}\|^2).
\end{aligned}    
\end{equation}
Then we have
\begin{equation}\label{k4t}
\begin{aligned}
&~~~~\|\bm{t}\|^2k_4(\bm{t})\\
&=\|\bm{t}\|^2k_2(\bm{t})\frac{k_1(\bm{t})+k_*(\bm{t})}{1-k_*^2(\bm{t})}\\
&=k_2(\bm{t})(k_1(\bm{t})+k_*(\bm{t}))\left(-\frac{1}{6}\rho^{(1)}(0)\rho^{(2)}(0)^{-1}+\left(\frac{5}{36}\rho^{(1)}(0)\rho^{(2)}(0)^{-2}\rho^{(3)}(0)+\frac{1}{4}\right)\|\bm{t}\|^2+o(\|\bm{t}\|^2)\right)\\
&=2\rho^{(1)}(0)^{-1}\left(\rho^{(2)}(0)+\rho^{(3)}(0)\|\bm{t}\|^2+o(\|\bm{t}\|^2)\right)\\
&~~~~ \left(2+4\rho^{(1)}(0)^{-1}\rho^{(2)}(0)\|\bm{t}\|^2+3\rho^{(1)}(0)^{-1}\rho^{(3)}(0)\|\bm{t}\|^4+o(\|\bm{t}\|^4)\right)\\
&~~~~\left(-\frac{1}{6}\rho^{(1)}(0)\rho^{(2)}(0)^{-1}+\left(\frac{5}{36}\rho^{(1)}(0)\rho^{(2)}(0)^{-2}\rho^{(3)}(0)+\frac{1}{4}\right)\|\bm{t}\|^2+o(\|\bm{t}\|^2)\right)\\
&=-\frac{2}{3}+\left(-\frac{1}{3}\rho^{(1)}(0)^{-1}\rho^{(2)}(0)-\frac{1}{9}\rho^{(2)}(0)^{-1}\rho^{(3)}(0)\right)\|\bm{t}\|^2+o(\|\bm{t}\|^2).
\end{aligned}
\end{equation}
Therefore, by (\ref{1exp}) and (\ref{k4t}),
\begin{equation}\label{k4exp}
\begin{aligned}
&~~~~\frac{\|\bm{t}\|^4k_4(\bm{t})}{2\rho^{(1)}(0)(1-k_1^2(\bm{t}))}\\
&=\left(a_0+b_0\|\bm{t}\|^2+o(\|\bm{t}\|^2)\right)\\
&~~~~\left(-\frac{2}{3}+\left(-\frac{1}{3}\rho^{(1)}(0)^{-1}\rho^{(2)}(0)-\frac{1}{9}\rho^{(2)}(0)^{-1}\rho^{(3)}(0)\right)\|\bm{t}\|^2+o(\|\bm{t}\|^2)\right)\\
&=:a_0'+b_0'\|\bm{t}\|^{2}+o(\|\bm{t}\|^2),
\end{aligned}
\end{equation}
where
$$
\begin{aligned}
a_0'=-\frac{2}{3}a_0=\frac{1}{6}\rho^{(2)}(0)^{-1}
\end{aligned}
$$
and
$$
\begin{aligned}
b_0'
=a_0\left(-\frac{1}{3}\rho^{(1)}(0)^{-1}\rho^{(2)}(0)-\frac{1}{9}\rho^{(2)}(0)^{-1}\rho^{(3)}(0)\right)-\frac{2}{3}b_0
=-\frac{1}{18}\rho^{(2)}(0)^{-2}\rho^{(3)}(0).
\end{aligned}
$$

For (iii), by (\ref{1exp}),
\begin{equation}\label{k1exp}
\begin{aligned}
&~~~~\frac{\|\bm{t}\|^2k_1(\bm{t})}{2\rho^{(1)}(0)(1-k_1^2(\bm{t}))}\\
&=\left(1+\rho^{(1)}(0)^{-1}\rho^{(2)}(0)\|\bm{t}\|^2+o(\|\bm{t}\|^2)\right)\left(a_0+b_0\|\bm{t}\|^2+o(\|\bm{t}\|^2)\right)\\
&=a_0+\left(a_0\rho^{(1)}(0)^{-1}\rho^{(2)}(0)+b_0\right)\|\bm{t}\|^2+o(\|\bm{t}\|^2)\\
&=:a_0''+b_0''\|\bm{t}\|^2+o(\|\bm{t}\|^2),
\end{aligned}    
\end{equation}
where $$a_0''=a_0=-\frac{1}{4}\rho^{(2)}(0)^{-1}$$
and
$$
\begin{aligned}
b_0''
&=-\frac{1}{4}\rho^{(2)}(0)^{-1}\rho^{(1)}(0)^{-1}\rho^{(2)}(0)+\frac{1}{8}\left(\rho^{(1)}(0)^{-1}+\rho^{(2)}(0)^{-2}\rho^{(3)}(0)\right)\\
&=\frac{1}{8}\left(-\rho^{(1)}(0)^{-1}+\rho^{(2)}(0)^{-2}\rho^{(3)}(0)\right).
\end{aligned}
$$

For (iv), by (\ref{Eq:k4k5}), (\ref{1exp}) and (\ref{k4exp}),
\begin{equation}\label{k5exp}
\begin{aligned}
\frac{\|\bm{t}\|^4k_5(\bm{t})}{2\rho^{(1)}(0)(1-k_1^2(\bm{t}))}
&=\frac{\|\bm{t}\|^4k_4(\bm{t})}{2\rho^{(1)}(0)(1-k_1^2(\bm{t}))}+\frac{\|\bm{t}\|^4(k_5(\bm{t})-k_4(\bm{t}))}{2\rho^{(1)}(0)(1-k_1^2(\bm{t}))}\\
&=\frac{\|\bm{t}\|^4k_4(\bm{t})}{2\rho^{(1)}(0)(1-k_1^2(\bm{t}))}+o(\|\bm{t}\|^2)\\
&=a_0'+b_0'\|\bm{t}\|^2+o(\|\bm{t}\|^2).
\end{aligned}
\end{equation}

\subsection{Asymptotic Expansions of the Main Part}
In this part, we would like to get the asymptotic expansion of 
\begin{equation}\label{mainpart}
\begin{aligned}
&~~~~\bm{\Sigma}(\bm{t})[1:(L-2),1:(L-2)]\\
&=\bm{G}_{22}(\bm{0})+\frac{1}{2\rho^{(1)}(0)(1-k_1^2(\bm{t}))}\bm{G}_{21}(\bm{t})\bm{G}_{21}^T(\bm{t})+\frac{k_4(\bm{t})}{2\rho^{(1)}(0)(1-k_1^2(\bm{t}))}\bm{G}_{21}(\bm{t})\bm{t}\bm{t}^T\bm{G}_{21}^T(\bm{t})
\end{aligned}
\end{equation} 
as $\|\bm{t}\|\to 0$. To this end, we still need to expand
\begin{enumerate}[label=(\roman*)]
    \item $\bm{G}_{21}(\bm{t})_{\left(i_1+j_1(j_1-1)/2\right)}\left(\bm{G}_{21}(\bm{t})_{\left(i_2+j_2(j_2-1)/2\right)}\right)^T$
    
    \item $\bm{G}_{21}(\bm{t})\bm{t}\bm{t}^T\bm{G}_{21}^T(\bm{t})\left[i_1+j_1(j_1-1)/2,i_2+j_2(j_2-1)/2\right]$

\end{enumerate}
for any integers $1\le i_1\le j_1\le N$ and $1\le i_2\le j_2\le N$.

Note that $\bm{G}_{21}(\bm{t})=-\Cov[\triangledown^2X(\bm{t}),\triangledown X(\bm{0})]$ can be written in the form:
$$
\bm{G}_{21}(\bm{t})=
\begin{pmatrix}
\bm{B}_{11}(\bm{t}) & \bm{B}_{12}(\bm{t}) & \cdots &  \bm{B}_{1N}(\bm{t})\\ 
\bm{B}_{21}(\bm{t}) & \bm{B}_{22}(\bm{t}) & \cdots &  \bm{B}_{2N}(\bm{t})\\ 
\vdots         & \vdots         & \cdots &  \vdots\\
\bm{B}_{N1}(\bm{t}) & \bm{B}_{N2}(\bm{t}) & \cdots &  \bm{B}_{NN}(\bm{t})
\end{pmatrix},
$$
where by Lemma \ref{Lem:MSD} and (\ref{R3}),
$$
\begin{aligned}
\bm{B}_{jk}(\bm{t}):=
\begin{pmatrix}
R_{1jk}(\bm{t}) \\ R_{2jk}(\bm{t}) \\ \vdots \\ R_{jjk}(\bm{t}) \\
\end{pmatrix}
&=4\rho^{(2)}(\|\bm{t}\|^2)\left(t_k
\begin{pmatrix}
\delta_{1,j} \\ \delta_{2,j} \\ \vdots \\ \delta_{j,j} \\
\end{pmatrix}+
\begin{pmatrix}
t_1 \\ t_2 \\ \vdots \\ t_j \\
\end{pmatrix}\delta_{j,k}+t_j
\begin{pmatrix}
\delta_{1,k} \\ \delta_{2,k} \\ \vdots \\ \delta_{j,k} \\
\end{pmatrix}
\right)+8\rho^{(3)}(\|\bm{t}\|^2)t_jt_k
\begin{pmatrix}
t_1 \\ t_2 \\ \vdots \\ t_j \\
\end{pmatrix}
\end{aligned}
$$
for any integers $1\le j,k\le N$.
For any integers $1\le i\le j\le N$ and $1\le k\le N$,
\begin{equation}\label{GBR}
\bm{G}_{21}(\bm{t})\left[i+\frac{j(j-1)}{2},k\right]=\bm{B}_{jk}(\bm{t})[i]
=R_{ijk}(\bm{t}).
\end{equation}
Let $\bm{v}_k=(\delta_{1,k},\dots,\delta_{N,k})^T$, $k=1,2,\dots,N$. Then  by (\ref{R3}) and (\ref{GBR}),
\begin{equation}\label{G21}
\begin{aligned}
\left(\bm{G}_{21}(\bm{t})_{\left(i+\frac{j(j-1)}{2}\right)}\right)^T
=
\begin{pmatrix}
R_{ij1}(\bm{t}) \\ R_{ij2}(\bm{t}) \\ \vdots \\ R_{ijN}(\bm{t}) \\
\end{pmatrix}
&=4\rho^{(2)}(\|\bm{t}\|^2)(\delta_{i,j}\bm{t}+t_i\bm{v}_j+t_j\bm{v}_i)+8\rho^{(3)}(\|\bm{t}\|^2)t_it_j\bm{t}\\
&=4\left(\rho^{(2)}(0)+\rho^{(3)}(0)\|\bm{t}\|^2\right)(\delta_{i,j}\bm{t}+t_i\bm{v}_j+t_j\bm{v}_i)\\
&~~~~+8\rho^{(3)}(0)t_it_j\bm{t}+o(\|\bm{t}\|^3).
\end{aligned}
\end{equation}
Note that for any integers $1\le i,j\le N$, since
$\bm{v}_j^T\bm{t}=t_j\text{ and }\bm{v}_i^T\bm{v}_j=\delta_{i,j},$
we have
$$(\delta_{i,j}\bm{t}^T+t_i\bm{v}_j^T+t_j\bm{v}_i^T)\bm{t}=\delta_{i,j}\|\bm{t}\|^2+2t_it_j.$$
Then for any integers $1\le i_1\le j_1\le N$ and $1\le i_2\le j_2\le N$, we have
\begin{equation}\label{G21G21}
\begin{aligned}
&~~~~\bm{G}_{21}(\bm{t})_{\left(i_1+j_1(j_1-1)/2\right)}\left(\bm{G}_{21}(\bm{t})_{\left(i_2+j_2(j_2-1)/2\right)}\right)^T\\
&=16\left(\rho^{(2)}(0)+\rho^{(3)}(0)\|\bm{t}\|^2\right)^2(\delta_{i_1,j_1}\bm{t}^T+t_{i_1}\bm{v}_{j_1}^T+t_{j_1}\bm{v}_{i_1}^T)(\delta_{i_2,j_2}\bm{t}+t_{i_2}\bm{v}_{j_2}+t_{j_2}\bm{v}_{i_2})\\
&~~~~+32\rho^{(3)}(0)t_{i_2}t_{j_2}\left(\rho^{(2)}(0)+\rho^{(3)}(0)\|\bm{t}\|^2\right)(\delta_{i_1,j_1}\bm{t}^T+t_{i_1}\bm{v}_{j_1}^T+t_{j_1}\bm{v}_{i_1}^T)\bm{t}\\
&~~~~+32\rho^{(3)}(0)t_{i_1}t_{j_1}\left(\rho^{(2)}(0)+\rho^{(3)}(0)\|\bm{t}\|^2\right)(\delta_{i_2,j_2}\bm{t}^T+t_{i_2}\bm{v}_{j_2}^T+t_{j_2}\bm{v}_{i_2}^T)\bm{t}+o(\|\bm{t}\|^4)\\
&=16\left(\rho^{(2)}(0)^2+2\rho^{(2)}(0)\rho^{(3)}(0)\|\bm{t}\|^2\right)\left(\delta_{i_1,j_1}\|\bm{t}\|^2+2t_{i_1}t_{j_1}\right)\delta_{i_2,j_2}\\
&~~~~+16\left(\rho^{(2)}(0)^2+2\rho^{(2)}(0)\rho^{(3)}(0)\|\bm{t}\|^2\right)\left(\delta_{i_1,j_1}t_{i_2}t_{j_2}+\delta_{j_1,j_2}t_{i_1}t_{i_2}+\delta_{i_1,j_2}t_{j_1}t_{i_2}\right)\\
&~~~~+16\left(\rho^{(2)}(0)^2+2\rho^{(2)}(0)\rho^{(3)}(0)\|\bm{t}\|^2\right)\left(\delta_{i_1,j_1}t_{i_2}t_{j_2}+\delta_{i_2,j_1}t_{i_1}t_{j_2}+\delta_{i_1,i_2}t_{j_1}t_{j_2}\right)\\
&~~~~+32\rho^{(3)}(0)\left(\rho^{(2)}(0)+\rho^{(3)}(0)\|\bm{t}\|^2\right)\left(\delta_{i_1,j_1}t_{i_2}t_{j_2}\|\bm{t}\|^2+\delta_{i_2,j_2}t_{i_1}t_{j_1}\|\bm{t}\|^2+4t_{i_1}t_{j_1}t_{i_2}t_{j_2}\right)\\
&~~~~+o(\|\bm{t}\|^4)\\
&=:a_1\|\bm{t}\|^2+b_1\|\bm{t}\|^4+o(\|\bm{t}\|^4),
\end{aligned}
\end{equation}
where 
$$
\begin{aligned}
a_1
&=16\rho^{(2)}(0)^2\big(\delta_{i_1,j_1}\delta_{i_2,j_2}+2\delta_{i_2,j_2}u_{i_1}u_{j_1}+2\delta_{i_1,j_1}u_{i_2}u_{j_2}\\
&~~~~+\delta_{j_1,j_2}u_{i_1}u_{i_2}+\delta_{i_1,j_2}u_{j_1}u_{i_2}+\delta_{i_2,j_1}u_{i_1}u_{j_2}+\delta_{i_1,i_2}u_{j_1}u_{j_2}\big),
\end{aligned}
$$
$$
\begin{aligned}
b_1&=2\rho^{(2)}(0)^{-1}\rho^{(3)}(0)a_1+32\rho^{(3)}(0)\rho^{(2)}(0)\left(\delta_{i_1,j_1}u_{i_2}u_{j_2}+\delta_{i_2,j_2}u_{i_1}u_{j_1}+4u_{i_1}u_{j_1}u_{i_2}u_{j_2}\right)\\
&=32\rho^{(3)}(0)\rho^{(2)}(0)\big(\delta_{i_1,j_1}\delta_{i_2,j_2}+3\delta_{i_2,j_2}u_{i_1}u_{j_1}+3\delta_{i_1,j_1}u_{i_2}u_{j_2}\\
&~~~~+\delta_{j_1,j_2}u_{i_1}u_{i_2}+\delta_{i_1,j_2}u_{j_1}u_{i_2}+\delta_{i_2,j_1}u_{i_1}u_{j_2}+\delta_{i_1,i_2}u_{j_1}u_{j_2}+4u_{i_1}u_{j_1}u_{i_2}u_{j_2}\big),
\end{aligned}
$$
and we recall that $\bm{u}=(u_1,\dots,u_N)^T\in\mathbb{S}^{N-1}$ and $\bm{t}=\|\bm{t}\|\bm{u}$ in Lemma \ref{Lem:Sigma}.

By (\ref{G21}), we also have
\begin{equation}\label{G21t}
\begin{aligned}
&~~~~\bm{G}_{21}(\bm{t})_{\left(i_1+j_1(j_1-1)/2\right)}\bm{t}=\left(\bm{G}_{21}(\bm{t})\bm{t}\right)_{\left(i_1+j_1(j_1-1)/2\right)}\\
&=\left(4\rho^{(2)}(\|\bm{t}\|^2)(\delta_{i,j}\bm{t}^T+t_i\bm{v}_j^T+t_j\bm{v}_i^T)+8\rho^{(3)}(\|\bm{t}\|^2)t_it_j\bm{t}^T\right)\bm{t}\\
&=4\rho^{(2)}(\|\bm{t}\|^2)(\delta_{i,j}\|\bm{t}\|^2+2t_it_j)+8\rho^{(3)}(\|\bm{t}\|^2)t_it_j\|\bm{t}\|^2\\
&=4\left(\rho^{(2)}(0)+\rho^{(3)}(0)\|\bm{t}\|^2\right)(\delta_{i,j}\|\bm{t}\|^2+2t_it_j)+8\rho^{(3)}(0)t_it_j\|\bm{t}\|^2+o(\|\bm{t}\|^4),
\end{aligned}
\end{equation}
and thus
\begin{equation}\label{G21ttG21}
\begin{aligned}
&~~~~\bm{G}_{21}(\bm{t})\bm{t}\bm{t}^T\bm{G}_{21}^T(\bm{t})\left[i_1+j_1(j_1-1)/2,i_2+j_2(j_2-1)/2\right]\\
&=16\left(\rho^{(2)}(0)+\rho^{(3)}(0)\|\bm{t}\|^2\right)^2(\delta_{i_1,j_1}\|\bm{t}\|^2+2t_{i_1}t_{j_1})(\delta_{i_2,j_2}\|\bm{t}\|^2+2t_{i_2}t_{j_2})\\
&~~~~+4\left(\rho^{(2)}(0)+\rho^{(3)}(0)\|\bm{t}\|^2\right)(\delta_{i_1,j_1}\|\bm{t}\|^2+2t_{i_1}t_{j_1})8\rho^{(3)}(0)t_{i_2}t_{j_2}\|\bm{t}\|^2\\
&~~~~+4\left(\rho^{(2)}(0)+\rho^{(3)}(0)\|\bm{t}\|^2\right)(\delta_{i_2,j_2}\|\bm{t}\|^2+2t_{i_2}t_{j_2})8\rho^{(3)}(0)t_{i_1}t_{j_1}\|\bm{t}\|^2\\
&~~~~+8\rho^{(3)}(0)t_{i_1}t_{j_1}\|\bm{t}\|^28\rho^{(3)}(0)t_{i_2}t_{j_2}\|\bm{t}\|^2+o(\|\bm{t}\|^6)\\
&=16\left(\rho^{(2)}(0)^2+2\rho^{(2)}(0)\rho^{(3)}(0)\|\bm{t}\|^2+\rho^{(3)}(0)^2\|\bm{t}\|^4\right)(\delta_{i_1,j_1}\|\bm{t}\|^2+2t_{i_1}t_{j_1})(\delta_{i_2,j_2}\|\bm{t}\|^2+2t_{i_2}t_{j_2})\\
&~~~~+32\rho^{(3)}(0)\left(\rho^{(2)}(0)+\rho^{(3)}(0)\|\bm{t}\|^2\right)(\delta_{i_1,j_1}\|\bm{t}\|^2+2t_{i_1}t_{j_1})t_{i_2}t_{j_2}\|\bm{t}\|^2\\
&~~~~+32\rho^{(3)}(0)\left(\rho^{(2)}(0)+\rho^{(3)}(0)\|\bm{t}\|^2\right)(\delta_{i_2,j_2}\|\bm{t}\|^2+2t_{i_2}t_{j_2})t_{i_1}t_{j_1}\|\bm{t}\|^2+o(\|\bm{t}\|^6)\\
&=:a_2\|\bm{t}\|^4+b_2\|\bm{t}\|^6+o(\|\bm{t}\|^6),
\end{aligned}
\end{equation}
where 
$$
\begin{aligned}
a_2&=16\rho^{(2)}(0)^2(\delta_{i_1,j_1}+2u_{i_1}u_{j_1})(\delta_{i_2,j_2}+2u_{i_2}u_{j_2})
\end{aligned}
$$
and
$$
\begin{aligned}
b_2
&=32\rho^{(2)}(0)\rho^{(3)}(0)(\delta_{i_1,j_1}+2u_{i_1}u_{j_1})(\delta_{i_2,j_2}+2u_{i_2}u_{j_2})\\
&~~~~+32\rho^{(2)}(0)\rho^{(3)}(0)(\delta_{i_1,j_1}u_{i_2}u_{j_2}+\delta_{i_2,j_2}u_{i_1}u_{j_1}+4u_{i_1}u_{j_1}u_{i_2}u_{j_2})\\
&=32\rho^{(2)}(0)\rho^{(3)}(0)\left(\delta_{i_1,j_1}\delta_{i_2,j_2}+3\delta_{i_1,j_1}u_{i_2}u_{j_2}+3\delta_{i_2,j_2}u_{i_1}u_{j_1}+8u_{i_1}u_{j_1}u_{i_2}u_{j_2}\right).
\end{aligned}
$$
By (\ref{R4}),
\begin{equation}\label{G220}
\begin{aligned}
\bm{G}_{22}(\bm{0})\left[i_1+j_1(j_1-1)/2, i_2+j_2(j_2-1)/2\right]&=
R_{i_1j_1i_2j_2}(\bm{0})\\
&=4\rho^{(2)}(0)\left(\delta_{i_1,j_1}\delta_{i_2,j_2}+\delta_{i_2,j_1}\delta_{i_1,j_2}+\delta_{i_1,i_2}\delta_{j_1,j_2}\right).
\end{aligned}
\end{equation}
Finally, by (\ref{1exp}), (\ref{k4exp}), (\ref{mainpart}), (\ref{G21G21}), (\ref{G21ttG21}) and (\ref{G220})  we have for integers $1\le i_1\le j_1\le N$ and $1\le i_2\le j_2\le N$,
\begin{equation}\notag
\begin{aligned}
&~~~~\bm{\Sigma}(\bm{t})\left[i_1+j_1(j_1-1)/2, i_2+j_2(j_2-1)/2\right]\\
&=4\rho^{(2)}(0)\left(\delta_{i_1,j_1}\delta_{i_2,j_2}+\delta_{i_2,j_1}\delta_{i_1,j_2}+\delta_{i_1,i_2}\delta_{j_1,j_2}\right)\\
&~~~~+\|\bm{t}\|^{-2}\left(a_0+b_0\|\bm{t}\|^{2}+o(\|\bm{t}\|^{2})\right)\left(a_1\|\bm{t}\|^2+b_1\|\bm{t}\|^4+o(\|\bm{t}\|^4)\right)\\
&~~~~+\|\bm{t}\|^{-4}\left(a_0'+b_0'\|\bm{t}\|^{2}+o(\|\bm{t}\|^{2})\right)\left(a_2\|\bm{t}\|^4+b_2\|\bm{t}\|^6+o(\|\bm{t}\|^6)\right)\\
&=4\rho^{(2)}(0)\left(\delta_{i_1,j_1}\delta_{i_2,j_2}+\delta_{i_2,j_1}\delta_{i_1,j_2}+\delta_{i_1,i_2}\delta_{j_1,j_2}\right)\\
&~~~~+\left(a_0+b_0\|\bm{t}\|^2+o(\|\bm{t}\|^2)\right)\left(a_1+b_1\|\bm{t}\|^2+o(\|\bm{t}\|^2)\right)\\
&~~~~+\left(a_0'+b_0'\|\bm{t}\|^2+o(\|\bm{t}\|^2)\right)\left(a_2+b_2\|\bm{t}\|^2+o(\|\bm{t}\|^2)\right)\\
&=:a_3+b_3\|\bm{t}\|^2+o(\|\bm{t}\|^2),
\end{aligned}
\end{equation}
where
$$
\begin{aligned}
a_3&=4\rho^{(2)}(0)\left(\delta_{i_1,j_1}\delta_{i_2,j_2}+\delta_{i_2,j_1}\delta_{i_1,j_2}+\delta_{i_1,i_2}\delta_{j_1,j_2}\right)+a_0a_1+a_0'a_2\\
&=4\rho^{(2)}(0)(\delta_{i_2,j_1}\delta_{i_1,j_2}+\delta_{i_1,i_2}\delta_{j_1,j_2}-\delta_{j_1,j_2}u_{i_1}u_{i_2}-\delta_{i_1,j_2}u_{j_1}u_{i_2}\\
&~~~~-\delta_{i_2,j_1}u_{i_1}u_{j_2}-\delta_{i_1,i_2}u_{j_1}u_{j_2}+2u_{i_1}u_{j_1}u_{i_2}u_{j_2})\\
&~~~~+\frac{8}{3}\rho^{(2)}(0)(\delta_{i_1,j_1}-u_{i_1}u_{j_1})(\delta_{i_2,j_2}-u_{i_2}u_{j_2}),
\end{aligned}
$$
and by letting $\alpha:=\rho^{(1)}(0)^{-1}\rho^{(2)}(0)^2$ and $\beta:=\rho^{(3)}(0)$,
$$
\begin{aligned}
b_3
&=a_0b_1+b_0a_1+a_0'b_2+b_0'a_2\\
&=\left(2\alpha-\frac{14}{9}\beta\right)\delta_{i_1,j_1}\delta_{i_2,j_2}
+\left(4\alpha-\frac{52}{9}\beta\right)\delta_{i_2,j_2}u_{i_1}u_{j_1}
+\left(4\alpha-\frac{52}{9}\beta\right)\delta_{i_1,j_1}u_{i_2}u_{j_2} \\   
&~~~~
+ \left(2\alpha-6\beta\right)\delta_{j_1,j_2}u_{i_1}u_{i_2}
+ \left(2\alpha-6\beta\right)\delta_{i_1,j_2}u_{j_1}u_{i_2}
+ \left(2\alpha-6\beta\right)\delta_{i_2,j_1}u_{i_1}u_{j_2}\\
&~~~~
+ \left(2\alpha-6\beta\right)\delta_{i_1,i_2}u_{j_1}u_{j_2}
+ \frac{64}{9}\beta u_{i_1}u_{j_1}u_{i_2}u_{j_2},
\end{aligned}$$
which can be verified by lengthy but straightforward calculation.

\subsection{Asymptotic Expansions of the Side Parts}
In this part, we would like to get the asymptotic expansion of 
\begin{equation}\label{Side1}
\begin{aligned}
&~~~~\bm{\Sigma}(\bm{t})[1:(L-2),L-1]\\
&=\bm{G}_{20}(\bm{0})+\frac{1}{2\rho^{(1)}(0)(1-k_1^2(\bm{t}))}\bm{G}_{21}(\bm{t})\bm{G}_{01}^T(\bm{t})
+\frac{k_4(\bm{t})}{2\rho^{(1)}(0)(1-k_1^2(\bm{t}))}\bm{G}_{21}(\bm{t})\bm{t}\bm{t}^T\bm{G}_{01}^T(\bm{t})
\end{aligned}
\end{equation}
and
\begin{equation}\label{Side2}
\begin{aligned}
&~~~~\bm{\Sigma}(\bm{t})[1:(L-2),L]\\
&=\bm{G}_{20}(\bm{t})+\frac{k_1(\bm{t})}{2\rho^{(1)}(0)(1-k_1^2(\bm{t}))}\bm{G}_{21}(\bm{t})\bm{G}_{01}^T(\bm{t})
+\frac{k_5(\bm{t})}{2\rho^{(1)}(0)(1-k_1^2(\bm{t}))}\bm{G}_{21}(\bm{t})\bm{t}\bm{t}^T\bm{G}_{01}^T(\bm{t})
\end{aligned}
\end{equation}
as $\|\bm{t}\|\to 0$.
To this end, we still need to expand
\begin{enumerate}[label=(\roman*)]
    \item $\bm{G}_{20}(\bm{t})[i+j(j-1)/2]$,
    \item $(\bm{G}_{21}(\bm{t})\bm{G}_{01}^T(\bm{t}))[i+j(j-1)/2]$,
    \item $(\bm{G}_{21}(\bm{t})\bm{t}\bm{t}^T\bm{G}_{01}^T(\bm{t}))[i+j(j-1)/2]$
\end{enumerate}
for any integers  $1\le i\le j\le N$.

By (\ref{R2}), we have
\begin{equation}\label{G200}
\begin{aligned}
\bm{G}_{20}(\bm{0})[i+j(j-1)/2]=R_{ij}(\bm{0})=2\rho^{(1)}(0)\delta_{i,j}
\end{aligned}
\end{equation}
and
\begin{equation}\label{G20}
\begin{aligned}
&~~~~\bm{G}_{20}(\bm{t})[i+j(j-1)/2]\\
&=R_{ij}(\bm{t})\\
&=2\rho^{(1)}(\|\bm{t}\|^2)\delta_{i,j}+4t_it_j\rho^{(2)}(\|\bm{t}\|^2)\\
&=2\delta_{i,j}\left(\rho^{(1)}(0)+\rho^{(2)}(0)\|\bm{t}\|^2+o(\|\bm{t}\|^2)\right)+4u_iu_j\|\bm{t}\|^2\left(\rho^{(2)}(0)+\rho^{(3)}(0)\|\bm{t}\|^2+o(\|\bm{t}\|^2)\right)\\
&=2\delta_{i,j}\rho^{(1)}(0)+\left(2\delta_{i,j}+4u_iu_j\right)\rho^{(2)}(0)\|\bm{t}\|^2+o(\|\bm{t}\|^2).
\end{aligned}
\end{equation}
By (\ref{R1}),
\begin{equation}\label{G01}
\bm{G}_{01}(\bm{t})=(R_1(\bm{t}),\dots,R_N(\bm{t}))=2\rho^{(1)}(\|\bm{t}\|^2)\bm{t}^T.
\end{equation}
Then by (\ref{G21t}) and (\ref{G01}),
\begin{equation}\label{G21G01}
\begin{aligned}
&~~~~\left(\bm{G}_{21}(\bm{t})\bm{G}_{01}^T(\bm{t})\right)[i+j(j-1)/2]\\
&=2\rho^{(1)}(\|\bm{t}\|^2)\bm{G}_{21}(\bm{t})_{\left(i+j(j-1)/2\right)}\bm{t}\\
&=2\rho^{(1)}(\|\bm{t}\|^2)\left(4\rho^{(2)}(\|\bm{t}\|^2)(\delta_{i,j}\|\bm{t}\|^2+2t_it_j)+8\rho^{(3)}(\|\bm{t}\|^2)t_it_j\|\bm{t}\|^2\right)\\
&=2\left(\rho^{(1)}(0)+\rho^{(2)}(0)\|\bm{t}\|^2+o(\|\bm{t}\|^2)\right)\left(4\rho^{(2)}(\|\bm{t}\|^2)(\delta_{i,j}\|\bm{t}\|^2+2t_it_j)+8\rho^{(3)}(\|\bm{t}\|^2)t_it_j\|\bm{t}\|^2\right)\\
&=8\left(\rho^{(1)}(0)\rho^{(2)}(0)+\left(\rho^{(2)}(0)^2+\rho^{(1)}(0)\rho^{(3)}(0)\right)\|\bm{t}\|^2\right)(\delta_{i,j}+2u_iu_j)\|\bm{t}\|^2\\
&~~~~+16\rho^{(1)}(0)\rho^{(3)}(0)u_iu_j\|\bm{t}\|^4+o(\|\bm{t}\|^4)\\
&=:\widetilde{a}_1\|\bm{t}\|^2+\widetilde{b}_1\|\bm{t}\|^4+o(\|\bm{t}\|^4),
\end{aligned}
\end{equation}
where 
$$\widetilde{a}_1=8\rho^{(1)}(0)\rho^{(2)}(0)(\delta_{i,j}+2u_iu_j)$$
and
$$\widetilde{b}_1=8\left(\rho^{(2)}(0)^2+\rho^{(1)}(0)\rho^{(3)}(0)\right)(\delta_{i,j}+2u_iu_j)+16\rho^{(1)}(0)\rho^{(3)}(0)u_iu_j.$$
By (\ref{G01}),
\begin{equation}\label{tG01}
\bm{t}^T\bm{G}_{01}^T(\bm{t})=2\rho^{(1)}(\|\bm{t}\|^2)\|\bm{t}\|^2.
\end{equation}
Then we have
\begin{equation}\label{G21ttG01}
\begin{aligned}
\left(\bm{G}_{21}(\bm{t})\bm{t}\bm{t}^T\bm{G}_{01}^T(\bm{t})\right)[i+j(j-1)/2]
&=2\rho^{(1)}(\|\bm{t}\|^2)\bm{G}_{21}(\bm{t})_{\left(i+j(j-1)/2\right)}\bm{t}\|\bm{t}\|^2\\
&=:\widetilde{a}_2\|\bm{t}\|^4+\widetilde{b}_2\|\bm{t}\|^6+o(\|\bm{t}\|^6),
\end{aligned}
\end{equation}
where by (\ref{G21G01}),
$$\widetilde{a}_2=\widetilde{a}_1~~\text{and}~~\widetilde{b}_2=\widetilde{b}_1.$$

Now combining (\ref{1exp}), (\ref{k4exp}), (\ref{Side1}),  (\ref{G200}), (\ref{G21G01}) and (\ref{G21ttG01}) implies for any integers $1\le i\le j\le N$,
$$
\begin{aligned}
&~~~~\bm{\Sigma}(\bm{t})[i+j(j-1)/2,L-1]\\
&=2\rho^{(1)}(0)\delta_{i,j}+\left(a_0+b_0\|\bm{t}\|^2+o(\|\bm{t}\|^2)\right)\left(\widetilde{a}_1+\widetilde{b}_1\|\bm{t}\|^2+o(\|\bm{t}\|^2)\right)\\
&~~~~+\left(a_0'+b_0'\|\bm{t}\|^{2}+o(\|\bm{t}\|^2)\right)\left(\widetilde{a}_2+\widetilde{b}_2\|\bm{t}\|^2+o(\|\bm{t}\|^2)\right)\\
&=:\widetilde{a}_3+\widetilde{b}_3\|\bm{t}\|^2+o(\|\bm{t}\|^2),
\end{aligned}
$$
where 
$$
\begin{aligned}
\widetilde{a}_3
&=2\rho^{(1)}(0)\delta_{i,j}+a_0\widetilde{a}_1+a_0'\widetilde{a}_2\\
&=2\rho^{(1)}(0)\delta_{i,j}+8\rho^{(1)}(0)\rho^{(2)}(0)(\delta_{i,j}+2u_iu_j)\left(-\frac{1}{4}\rho^{(2)}(0)^{-1}+\frac{1}{6}\rho^{(2)}(0)^{-1}\right)\\
&=\frac{4}{3}\rho^{(1)}(0)(\delta_{i,j}-u_iu_j),
\end{aligned}
$$
and by $\widetilde{a}_1=\widetilde{a}_2$ and $\widetilde{b}_1=\widetilde{b}_2$,
$$
\begin{aligned}
\widetilde{b}_3
&=a_0\widetilde{b}_1+b_0\widetilde{a}_1+a_0'\widetilde{b}_2+b_0'\widetilde{a}_2\\
&=\left(-\frac{1}{4}\rho^{(2)}(0)^{-1}+\frac{1}{6}\rho^{(2)}(0)^{-1}\right)\\
&~~~~\left(8\left(\rho^{(2)}(0)^2+\rho^{(1)}(0)\rho^{(3)}(0)\right)(\delta_{i,j}+2u_iu_j)+16\rho^{(1)}(0)\rho^{(3)}(0)u_iu_j
\right)\\
&~~~~+\left(\frac{1}{8}\left(\rho^{(1)}(0)^{-1}+\rho^{(2)}(0)^{-2}\rho^{(3)}(0)\right)-\frac{1}{18}\rho^{(2)}(0)^{-2}\rho^{(3)}(0)
\right)\\
&~~~~\left(8\rho^{(1)}(0)\rho^{(2)}(0)(\delta_{i,j}+2u_iu_j)\right)\\
&=\left(\frac{1}{3}\alpha'-\frac{1}{9}\beta'\right)\delta_{i,j}+\left(\frac{2}{3}\alpha'-\frac{14}{9}\beta'\right) u_iu_j,
\end{aligned}
$$
where $\alpha':=\rho^{(2)}(0)$ and $\beta':=\rho^{(1)}(0)\rho^{(2)}(0)^{-1}\rho^{(3)}(0)$.
Also combining (\ref{k1exp}), (\ref{k5exp}), (\ref{Side2}),  (\ref{G20}), (\ref{G21G01}) and (\ref{G21ttG01}) implies for any integers $1\le i\le j\le N$,
$$
\begin{aligned}
\bm{\Sigma}(\bm{t})[i+j(j-1)/2,L]
&=2\rho^{(1)}(0)\delta_{i,j}+\left(2\delta_{i,j}+4u_iu_j\right)\rho^{(2)}(0)\|\bm{t}\|^2+o(\|\bm{t}\|^2)\\
&~~~~+\left(a_0''+b_0''\|\bm{t}\|^2+o(\|\bm{t}\|^2)\right)\left(\widetilde{a}_1+\widetilde{b}_1\|\bm{t}\|^2+o(\|\bm{t}\|^2)\right)\\
&~~~~+\left(a_0'+b_0'\|\bm{t}\|^2+o(\|\bm{t}\|^2)\right)\left(\widetilde{a}_2+\widetilde{b}_2\|\bm{t}\|^2+o(\|\bm{t}\|^2)\right)\\
&=:\widetilde{a}_4+\widetilde{b}_4\|\bm{t}\|^2+o(\|\bm{t}\|^2),
\end{aligned}
$$
where by $a_0=a_0''$,
$$
\begin{aligned}
\widetilde{a}_4
=2\rho^{(1)}(0)\delta_{i,j}+a_0''\widetilde{a}_1+a_0'\widetilde{a}_2
=2\rho^{(1)}(0)\delta_{i,j}+a_0\widetilde{a}_1+a_0'\widetilde{a}_2
=\widetilde{a}_3,
\end{aligned}$$
and by $(2\delta_{i,j}+4u_iu_j)\rho^{(2)}(0)+b_0''\widetilde{a}_1=b_0\widetilde{a}_1$,
$$
\begin{aligned}
\widetilde{b}_4
&=\left(2\delta_{i,j}+4u_iu_j\right)\rho^{(2)}(0)+a_0''\widetilde{b}_1+b_0''\widetilde{a}_1+a_0'\widetilde{b}_2+b_0'\widetilde{a}_2=a_0\widetilde{b}_1+b_0\widetilde{a}_1+a_0'\widetilde{b}_2+b_0'\widetilde{a}_2=\widetilde{b}_3.
\end{aligned}
$$

\subsection{Asymptotic Expansions of the Corner Part}
In this part, we would like to get the asymptotic expansion of 
\begin{equation}\label{Tail1}
\begin{aligned}
&~~~~\bm{\Sigma}(\bm{t})[L-1,L-1]=\bm{\Sigma}(\bm{t})[L,L]\\
&=\bm{G}_{00}(\bm{0})+\frac{1}{2\rho^{(1)}(0)(1-k_1^2(\bm{t}))}\bm{G}_{01}(\bm{t})\bm{G}_{01}^T(\bm{t})
+\frac{k_4(\bm{t})}{2\rho^{(1)}(0)(1-k_1^2(\bm{t}))}\bm{G}_{01}(\bm{t})\bm{t}\bm{t}^T\bm{G}_{01}^T(\bm{t}) 
\end{aligned}
\end{equation}
and
\begin{equation}\label{Tail2}
\begin{aligned}
&~~~~\bm{\Sigma}(\bm{t})[L-1,L]=\bm{\Sigma}(\bm{t})[L,L-1]\\
&=\bm{G}_{00}(\bm{t})+\frac{k_1(\bm{t})}{2\rho^{(1)}(0)(1-k_1^2(\bm{t}))}\bm{G}_{01}(\bm{t})\bm{G}_{01}^T(\bm{t})
+\frac{k_5(\bm{t})}{2\rho^{(1)}(0)(1-k_1^2(\bm{t}))}\bm{G}_{01}(\bm{t})\bm{t}\bm{t}^T\bm{G}_{01}^T(\bm{t}) 
\end{aligned}
\end{equation}
as $\|\bm{t}\|\to 0$.
To this end, we still need to expand
\begin{enumerate}[label=(\roman*)]
    \item $\bm{G}_{00}(\bm{t})$,
    \item $\bm{G}_{01}(\bm{t})\bm{G}_{01}^T(\bm{t})$,
    \item $\bm{G}_{01}(\bm{t})\bm{t}\bm{t}^T\bm{G}_{01}^T(\bm{t})$.
\end{enumerate}

For (i),
it is easy to see
\begin{equation}\label{G00}
\bm{G}_{00}(\bm{t})=R(\bm{t})=\rho(0)+\rho^{(1)}(0)\|\bm{t}\|^2+o(\|\bm{t}\|^2)\text{ and }\bm{G}_{00}(\bm{0})=R(\bm{0})=\rho(0).
\end{equation}
For (ii), by (\ref{G01}), we have
\begin{equation}\label{G01G01}
\begin{aligned}
\bm{G}_{01}(\bm{t})\bm{G}_{01}^T(\bm{t})
&=4\left(\rho^{(1)}(0)+\rho^{(2)}(0)\|\bm{t}\|^2+o(\|\bm{t}\|^2)\right)^2\|\bm{t}\|^2\\
&=4\rho^{(1)}(0)^2\|\bm{t}\|^2+8\rho^{(1)}(0)\rho^{(2)}(0)\|\bm{t}\|^4+o(\|\bm{t}\|^4)\\
&=:\widehat{a}_1\|\bm{t}\|^2+\widehat{b}_1\|\bm{t}\|^4+o(\|\bm{t}\|^4),
\end{aligned}    
\end{equation}
where 
$$\widehat{a}_1=4\rho^{(1)}(0)^2\text{  and  }\widehat{b}_1=8\rho^{(1)}(0)\rho^{(2)}(0),$$
By (\ref{tG01}),
\begin{equation}\label{G01ttG01}
\begin{aligned}
\bm{G}_{01}(\bm{t})\bm{t}\bm{t}^T\bm{G}_{01}^T(\bm{t})
&=4\rho^{(1)}(0)^2\|\bm{t}\|^4+8\rho^{(1)}(0)\rho^{(2)}(0)\|\bm{t}\|^6+o(\|\bm{t}\|^6)\\
&=:\widehat{a}_2\|\bm{t}\|^4+\widehat{b}_2\|\bm{t}\|^6+o(\|\bm{t}\|^6),
\end{aligned}    
\end{equation}
where 
$$\widehat{a}_2=4\rho^{(1)}(0)^2\text{  and  }\widehat{b}_2=8\rho^{(1)}(0)\rho^{(2)}(0).$$
Therefore, combining (\ref{1exp}), (\ref{k4exp}), (\ref{Tail1}), (\ref{G00}), (\ref{G01G01}) and (\ref{G01ttG01}) implies
$$
\begin{aligned}
&~~~~\bm{\Sigma}(\bm{t})[L-1,L-1]=\bm{\Sigma}(\bm{t})[L,L]\\
&=\rho(0)+\left(a_0+b_0\|\bm{t}\|^2+o(\|\bm{t}\|^2)\right)\left(\widehat{a}_1+\widehat{b}_1\|\bm{t}\|^2+o(\|\bm{t}\|^2)\right)\\
&~~~~+\left(a_0'+b_0'\|\bm{t}\|^{2}+o(\|\bm{t}\|^2)\right)\left(\widehat{a}_2+\widehat{b}_2\|\bm{t}\|^2+o(\|\bm{t}\|^2)\right)\\
&=:\widehat{a}_3+\widehat{b}_3\|\bm{t}\|^2+o(\|\bm{t}\|^2),
\end{aligned}
$$
where 
$$
\begin{aligned}
\widehat{a}_3
&=\rho(0)+a_0\widehat{a}_1+a_0'\widehat{a}_2\\
&=\rho(0)+\left(-\frac{1}{4}\rho^{(2)}(0)^{-1}+\frac{1}{6}\rho^{(2)}(0)^{-1}
\right)4\rho^{(1)}(0)^2\\
&=\rho(0)-\frac{1}{3}\rho^{(1)}(0)^2\rho^{(2)}(0)^{-1}\\
\end{aligned}
$$
and by $\widehat{a}_1=\widehat{a}_2$ and $\widehat{b}_1=\widehat{b}_2$,
$$
\begin{aligned}
\widehat{b}_3
&=a_0\widehat{b}_1+b_0\widehat{a}_1+a_0'\widehat{b}_2+b_0'\widehat{a}_2\\
&=(a_0+a_0')\widehat{b}_1+(b_0+b_0')\widehat{a}_1\\
&=\left(-\frac{1}{4}\rho^{(2)}(0)^{-1}+\frac{1}{6}\rho^{(2)}(0)^{-1}\right)8\rho^{(1)}(0)\rho^{(2)}(0)\\
&~~~~+\left(\frac{1}{8}\left(\rho^{(1)}(0)^{-1}+\rho^{(2)}(0)^{-2}\rho^{(3)}(0)\right)
-\frac{1}{18}\rho^{(2)}(0)^{-2}\rho^{(3)}(0)\right)4\rho^{(1)}(0)^2\\
&=-\frac{1}{6}\rho^{(1)}(0)+\frac{5}{18}\rho^{(1)}(0)^2\rho^{(2)}(0)^{-2}\rho^{(3)}(0).
\end{aligned}
$$
Also, combining (\ref{k1exp}), (\ref{k5exp}), (\ref{Tail2}), (\ref{G00}), (\ref{G01G01}) and (\ref{G01ttG01}) implies
$$
\begin{aligned}
&~~~~\bm{\Sigma}(\bm{t})[L-1,L]=\bm{\Sigma}(\bm{t})[L,L-1]\\
&=\rho(0)+\rho^{(1)}(0)\|\bm{t}\|^2+o(\|\bm{t}\|^2)\\
&~~~~+\left(a_0''+b_0''\|\bm{t}\|^2+o(\|\bm{t}\|^2)\right)\left(\widehat{a}_1+\widehat{b}_1\|\bm{t}\|^2+o(\|\bm{t}\|^2)\right)\\
&~~~~+\left(a_0'+b_0'\|\bm{t}\|^{2}+o(\|\bm{t}\|^2)\right)\left(\widehat{a}_2+\widehat{b}_2\|\bm{t}\|^2+o(\|\bm{t}\|^2)\right)\\
&=:\widehat{a}_4+\widehat{b}_4\|\bm{t}\|^2+o(\|\bm{t}\|^2),
\end{aligned}
$$
where by $a_0=a_0''$,
$$
\begin{aligned}
\widehat{a}_4
=\rho(0)+a_0''\widehat{a}_1+a_0'\widehat{a}_2
=\rho(0)+a_0\widehat{a}_1+a_0'\widehat{a}_2
=\widehat{a}_3,
\end{aligned}
$$
and by $\rho^{(1)}(0)+b_0''\widehat{a}_1=b_0\widehat{a}_1$,
$$
\begin{aligned}
\widehat{b}_4=\rho^{(1)}(0)+a_0''\widehat{b}_1+b_0''\widehat{a}_1+a_0'\widehat{b}_2+b_0'\widehat{a}_2=a_0\widehat{b}_1+b_0\widehat{a}_1+a_0'\widehat{b}_2+b_0'\widehat{a}_2=\widehat{b}_3.
\end{aligned}
$$
Hence the proof of Lemma \ref{Lem:Sigma} is completed.
     
\section{Features of the Side Parts}\label{FSP}
Fix $N\ge 2$. Let $X$ be qualified. Suppose that $X$ also satisfies (\ref{Con:GC1}) for some $\tilde{\delta}_{\rho}>0$.
Recall in Section \ref{Sec:S}, $\bm{u}_0=(0,\dots,0,1)\in \mathbb{R}^{N}$. 
In this section, we would like to calculate the side parts (\ref{Side1}) and (\ref{Side2}) when $\bm{t}$ has the form $\bm{t}:=\bm{u}_0r$ for any $r\in[0,\tilde{\delta}_{\rho}]$.
Recall that by (\ref{G200})
\begin{equation}\label{G200prec}
\begin{aligned}
\bm{G}_{20}(\bm{0})[i+j(j-1)/2]=2\rho^{(1)}(0)\delta_{i,j},
\end{aligned}
\end{equation}
by (\ref{G20})
\begin{equation}\label{G20prec}
\bm{G}_{20}(\bm{t})[i+j(j-1)/2]
=2\rho^{(1)}(\|\bm{t}\|^2)\delta_{i,j}+4t_it_j\rho^{(2)}(\|\bm{t}\|^2),
\end{equation}
by (\ref{G21G01})
\begin{equation}\label{G21G01prec}
\begin{aligned}
&~~~~\left(\bm{G}_{21}(\bm{t})\bm{G}_{01}^T(\bm{t})\right)[i+j(j-1)/2]\\
&=2\rho^{(1)}(\|\bm{t}\|^2)\left(4\rho^{(2)}(\|\bm{t}\|^2)(\delta_{i,j}\|\bm{t}\|^2+2t_it_j)+8\rho^{(3)}(\|\bm{t}\|^2)t_it_j\|\bm{t}\|^2\right),
\end{aligned}
\end{equation}
and by (\ref{G21ttG01}) 
\begin{equation}\label{G21ttG01prec}
\begin{aligned}
&~~~~\left(\bm{G}_{21}(\bm{t})\bm{t}\bm{t}^T\bm{G}_{01}^T(\bm{t})\right)[i+j(j-1)/2]\\
&=2\rho^{(1)}(\|\bm{t}\|^2)\|\bm{t}\|^2\left(4\rho^{(2)}(\|\bm{t}\|^2)(\delta_{i,j}\|\bm{t}\|^2+2t_it_j)+8\rho^{(3)}(\|\bm{t}\|^2)t_it_j\|\bm{t}\|^2\right)
\end{aligned}
\end{equation}
for any $\bm{t}\in B(\bm{0}_N,\delta_{\rho})\setminus\{\bm{0}_N\}$ and integers $1\le i\le j\le N$.
Then by taking $\bm{t}:=\bm{u}_0r$ into (\ref{G200prec})-(\ref{G21ttG01prec}), one can easily check
\begin{enumerate}[label=(\roman*)]
    \item for any integers $1\le i<j\le N$
    $$\bm{G}_{20}(\bm{0})[i+j(j-1)/2]=0,$$
    $$\bm{G}_{20}(\bm{t})[i+j(j-1)/2]=0,$$
    $$\left(\bm{G}_{21}(\bm{t})\bm{G}_{01}^T(\bm{t})\right)[i+j(j-1)/2]=0,$$
    and
    $$\left(\bm{G}_{21}(\bm{t})\bm{t}\bm{t}^T\bm{G}_{01}^T(\bm{t})\right)[i+j(j-1)/2]=0;$$
    
    \item for any integer $1\le k\le N-1$
    $$\bm{G}_{20}(\bm{0})[k+k(k-1)/2]=2\rho^{(1)}(0),$$
    $$\bm{G}_{20}(\bm{t})[k+k(k-1)/2]=2\rho^{(1)}(r^2),$$
    $$\left(\bm{G}_{21}(\bm{t})\bm{G}_{01}^T(\bm{t})\right)[k+k(k-1)/2]=8\rho^{(1)}(r^2)\rho^{(2)}(r^2)r^2,$$
    and
    $$\left(\bm{G}_{21}(\bm{t})\bm{t}\bm{t}^T\bm{G}_{01}^T(\bm{t})\right)[k+k(k-1)/2]=8\rho^{(1)}(r^2)\rho^{(2)}(r^2)r^4;$$
    \item when $k=N$
    $$\bm{G}_{20}(\bm{0})[k+k(k-1)/2]=2\rho^{(1)}(0),$$
    $$\bm{G}_{20}(\bm{t})[k+k(k-1)/2]=2\rho^{(1)}(r^2)+4r^2\rho^{(2)}(r^2),$$
    $$\left(\bm{G}_{21}(\bm{t})\bm{G}_{01}^T(\bm{t})\right)[k+k(k-1)/2]=24\rho^{(1)}(r^2)\rho^{(2)}(r^2)r^2+16\rho^{(1)}(r^2)\rho^{(3)}(r^2)r^4,$$
    and
    $$\left(\bm{G}_{21}(\bm{t})\bm{t}\bm{t}^T\bm{G}_{01}^T(\bm{t})\right)[k+k(k-1)/2]=24\rho^{(1)}(r^2)\rho^{(2)}(r^2)r^4+16\rho^{(1)}(r^2)\rho^{(3)}(r^2)r^6.$$
\end{enumerate}
Then by (\ref{Side1}) and (\ref{Side2}), we have for $k=1,2,\dots,N-1$
\begin{equation}\label{SPL1}
    \begin{aligned}
    &~~~~\bm{\Sigma}(\bm{u}_0r)[k+k(k-1)/2,L-1]\\
    &=2\rho^{(1)}(0)+\frac{1}{2\rho^{(1)}(0)(1-k_1^2(\bm{u}_0r))}8\rho^{(1)}(r^2)\rho^{(2)}(r^2)r^2+\frac{k_4(\bm{u}_0r)}{2\rho^{(1)}(0)(1-k_1^2(\bm{u}_0r))}8\rho^{(1)}(r^2)\rho^{(2)}(r^2)r^4\\
    &=2\rho^{(1)}(0)+8\rho^{(1)}(r^2)r^2\rho^{(2)}(r^2)\frac{1+k_4(\bm{u}_0r)r^2}{2\rho^{(1)}(0)(1-k_1^2(\bm{u}_0r))}\\
    &=2\rho^{(1)}(0)\left(1+k_1(\bm{u}_0r)k_2(\bm{u}_0r)r^2\frac{1+k_4(\bm{u}_0r)r^2}{1-k_1^2(\bm{u}_0r)}\right)\\
    &=2\rho^{(1)}(0)\left(1+\frac{k_1(\bm{u}_0r)k_2(\bm{u}_0r)r^2}{1-k_*^2(\bm{u}_0r)}\right)\\
    &=2\rho^{(1)}(0)\left(\frac{1-k_1(\bm{u}_0r)k_*(\bm{u}_0r)-k_2^2(\bm{u}_0r)r^4}{1-k_*^2(\bm{u}_0r)}\right)
    \end{aligned}
\end{equation}
and
\begin{equation}\label{SPL}
    \begin{aligned}
    &~~~~\bm{\Sigma}(\bm{u}_0r)[k+k(k-1)/2,L]\\
    &=2\rho^{(1)}(r^2)+\frac{k_1(\bm{u}_0r)}{2\rho^{(1)}(0)(1-k_1^2(\bm{u}_0r))}8\rho^{(1)}(r^2)\rho^{(2)}(r^2)r^2+\frac{k_5(\bm{u}_0r)}{2\rho^{(1)}(0)(1-k_1^2(\bm{u}_0r))}8\rho^{(1)}(r^2)\rho^{(2)}(r^2)r^4\\
    &=2\rho^{(1)}(r^2)+8\rho^{(1)}(r^2)r^2\rho^{(2)}(r^2)\frac{k_1(\bm{u}_0r)+k_5(\bm{u}_0r)r^2}{2\rho^{(1)}(0)(1-k_1^2(\bm{u}_0r))}\\
    &=2\rho^{(1)}(0)k_1(\bm{u}_0r)\left(1+k_2(\bm{u}_0r)r^2\frac{k_1(\bm{u}_0r)+k_5(\bm{u}_0r)r^2}{1-k_1^2(\bm{u}_0r)}\right)\\
    &=2\rho^{(1)}(0)k_1(\bm{u}_0r)\left(1+\frac{k_2(\bm{u}_0r)k_*(\bm{u}_0r)r^2}{1-k_*^2(\bm{u}_0r)}\right)\\
    &=2\rho^{(1)}(0)k_1(\bm{u}_0r)\left(\frac{1-k_1(\bm{u}_0r)k_*(\bm{u}_0r)}{1-k_*^2(\bm{u}_0r)}\right).
    \end{aligned}
\end{equation}

\section{Proof of Lemma \ref{Lem:sixorder}}\label{subsec:sixorder}
\begin{proof}
Since $X$ is stationary,
it suffices to show that there exist finite constants $K >0$ and $\alpha>0$ such that 
$$
\max_{1\le i_1,i_2,i_3,i_4\le N}\left|r_{i_1i_2i_3i_4}(\bm{0})-r_{i_1i_2i_3i_4}(\bm{t})\right|\le K|\log(\|\bm{t}\|)|^{-(1+\alpha)}
$$
for all $\|\bm{t}\|>0$ small enough.
Since $x\le (-\log(x))^{-2}$ for any $0< x< 1$, we have 
$\|\bm{t}\|<\left(-\log\left(\|\bm{t}\|\right)\right)^{-2}$ when $0<\|\bm{t}\|<1$.
Thus, it suffices to show that there exists a constant $C>0$ such that
\begin{equation}\label{eq:r_fourth-order}
\max_{1\le i_1,i_2,i_3,i_4\le N}\left|r_{i_1i_2i_3i_4}(\bm{0})-r_{i_1i_2i_3i_4}(\bm{t})\right|\le C\|\bm{t}\|    
\end{equation}
for all $\|\bm{t}\|>0$ small enough. Since all of the sixth-order partial derivatives of $r(\bm{t})$ exist at $\bm{t}=\bm{0}$, there exists a constant $\varepsilon>0$ small enough such that for any $1\le i_1,i_2,i_3,i_4,i_5\le N$, $r_{i_1i_2i_3i_4}(\bm{t})$ is continuous and $r_{i_1i_2i_3i_4i_5}(\bm{t})$ is bounded on $\bm{t}\in B(\bm{0},\varepsilon)$, where $B(\bm{0},\varepsilon)$ is the $N$-dimensional open ball centered at the origin with radius $\varepsilon$. Hence (\ref{eq:r_fourth-order}) holds for $\bm{t}\in B(\bm{0},\varepsilon)$.

\end{proof}

\end{appendices}

\end{document}